\documentclass[12pt]{article}

\usepackage{amssymb, amsthm, amsmath}	
\usepackage{enumitem, hyperref}
\usepackage{comment}
\usepackage{url}
\usepackage{color}
\usepackage{float}
\usepackage[margin=1.2in]{geometry}

\usepackage{longtable}

\usepackage{tikz}
\usetikzlibrary{arrows, chains, positioning, shapes.symbols,calc}
\usepackage{tikz-3dplot}
\tdplotsetmaincoords{70}{110}
\usepackage{tikz-cd}

\numberwithin{equation}{section}
\makeatletter
\renewcommand{\subsubsection}[1]{%
	\vspace{.5em}% adjust or remove this for space before
	\noindent\textbf{#1.}%
}
\makeatother

\theoremstyle{plain}
\newtheorem{theorem}{Theorem}[section]
\newtheorem{lemma}[theorem]{Lemma}
\newtheorem{proposition}[theorem]{Proposition}
\newtheorem{corollary}[theorem]{Corollary}

\theoremstyle{definition}
\newtheorem{definition}[theorem]{Definition}
\newtheorem{notation}[theorem]{Notation}
\newtheorem{remark}[theorem]{Remark}
\newtheorem{example}[theorem]{Example}

\newcommand{\C}{\mathbb{C}}

\newcommand{\R}{\mathbb{R}}
\newcommand{\Z}{\mathbb{Z}}
\renewcommand{\P}{\mathbb{P}}
\newcommand{\CP}{\mathbb{CP}}
\newcommand{\inn}{\langle \cdot, \cdot \rangle}

\newcommand{\mE}{\mathcal{E}}

\newcommand{\mG}{\mathcal{G}}

\newcommand{\mO}{\mathcal{O}}
\newcommand{\mH}{\mathcal{H}}

\newcommand{\mT}{\mathcal{T}}
\newcommand{\mL}{\mathcal{L}}
\newcommand{\mN}{\mathcal{N}}

\newcommand{\mR}{\mathcal{R}}
\newcommand{\mS}{\mathcal{S}}

\newcommand{\mV}{\mathcal{V}}
\newcommand{\mLi}{\mathcal{L}_{\textup{irr}}}

\newcommand{\ba}{\mathbf{a}}

\newcommand{\tH}{\widetilde{H}}

\newcommand{\tx}{\widetilde{x}}

\newcommand{\tV}{\widetilde{V}}
\newcommand{\tn}{\widetilde{\nabla}}

\newcommand{\tA}{\widetilde{A}}

\newcommand{\ts}{\widetilde{s}}

\newcommand{\ta}{\tilde{a}}

\newcommand{\tomega}{\widetilde{\omega}}

\newcommand{\tGamma}{\widetilde{\Gamma}}

\newcommand{\txA}{\textup{A}}
\newcommand{\txB}{\textup{B}}
\newcommand{\txC}{\textup{C}}
\newcommand{\gen}{\mathrm{gen}}

\DeclareMathOperator{\rk}{rank}
\DeclareMathOperator{\spn}{span}
\DeclareMathOperator{\codim}{codim}
\DeclareMathOperator{\parch}{par-ch}
\DeclareMathOperator{\parc}{par-c}
\DeclareMathOperator{\pardeg}{par-deg}
\DeclareMathOperator{\ch}{ch}

\DeclareMathOperator{\Gr}{Gr}

\DeclareMathOperator{\Bl}{Bl}
\DeclareMathOperator{\Res}{Res}
\DeclareMathOperator{\End}{End}
\DeclareMathOperator{\tr}{tr}
\DeclareMathOperator{\diag}{diag}

\DeclareMathOperator{\Sing}{Sing}

\DeclareMathOperator{\GL}{GL}

\DeclareMathOperator{\Id}{Id}
\DeclareMathOperator{\Hol}{Hol}

\DeclareMathOperator{\dev}{dev}
\DeclareMathOperator{\diam}{diameter}

\newcommand{\newpar}{\vspace{.5em}}

\newcommand{\p}{\partial}

\newcommand{\hol}{\textup{(holomorphic)}}

\newcommand{\inclusiondot}{\mathrel{\ooalign{$\subset$\cr\hfil\scalebox{1.2}[1]{$\cdot$}\hfil\cr}}}

\newcommand{\Nes}{\mathbf{Nes}}

\newcommand{\Address}{{% additional braces for segregating \footnotesize
		\bigskip
		\begin{flushright}
			\textsc{Loughborough University, Department of Mathematics\\
				Schofield Building, LE11 3TU, United Kingdom} \\
			\small{\href{mailto:m.de-borbon@lboro.ac.uk}{\nolinkurl{m.de-borbon@lboro.ac.uk}}} \\
			\vspace{3mm}
			\textsc{King's College London, Department of Mathematics\\
				Strand, WC2R 2LS, United Kingdom} \\
			\small{\href{mailto:dmitri.panov@kcl.ac.uk}{\nolinkurl{dmitri.panov@kcl.ac.uk}}}
		\end{flushright}
}}

\title{Polyhedral K\"ahler metrics on \(\CP^n\)}
\date{\today}
\author{Martin de Borbon and Dmitri Panov}

\begin{document}
	
\maketitle

\begin{abstract}
	We give necessary and sufficient conditions for the existence of polyhedral K\"ahler metrics on \(\CP^n\) whose singular set is a hyperplane arrangement and whose cone angles are in \((0, 2\pi)\). These conditions take the form of linear and quadratic constraints on the cone angles and are entirely determined by the intersection poset of the arrangement.
	Our proof of existence relies on a parabolic version of the Kobayashi-Hitchin correspondence, due to T. Mochizuki. 
\end{abstract}

%We use Mochizuki's results on the Kobayashi-Hitchin correspondence for parabolic bundles to show existence of polyhedral K\"ahler metrics on \(\CP^n\) with cone singularities along a hyperplane arrangement, provided that certain linear and quadratic constraint equations on the cone angles are satisfied. Conversely, using Ohtsuki's residue formula for Chern classes of vector bundles equipped with logarithmic connections, we show that the above constraint equations are also necessary for the existence of such metrics. 

\tableofcontents

%\newpage

\section*{List of Symbols}
%\addcontentsline{toc}{section}{List of Symbols}

\begin{longtable}{l l}
	\textbf{Symbol} & \textbf{Description} \\
	\hline & \\
	% a
	\(a_L\) & Definition \ref{def:aL} \\
    \(a_x\) & Equation \ref{eq:kltdef} \\
	  \(A_x\) & Equation \ref{eq:Axmatrix} \\
	% b
	\(B_H\) &   Definition \ref{def:B} \\
	\(B(L, M)\) &  Definition \ref{def:B} \\
	% c
	\(\gamma_L\) &  Definition \ref{def:gammaL} \\
	% d
	\(D_L\) & Equation \eqref{eq:DL} \\
	% e
	% F
	% G
	\(\mG\) & The set of proper irreducible subspaces, Definition \ref{def:G} \\
	% H
	\(h\) &  Notation \ref{not:h} \\
	
	\(\mH_L\) & localization of \(\mH\) at \(L\) \\
	% i
	\(\inclusiondot\) & Definition \ref{def:inclusion} \\
	\(\Subset\) & Definition \ref{def:inclusion} \\
	\(\pitchfork_r\) & Definition \ref{def:redint} \\
	% L
	\(\mL\) &  Definition \ref{def:intposet} \\
	\(\mLi\) & Definition \ref{def:G}  \\
	\(\mLi^{\circ}\) & \ref{def:G} \\
	% m
	% O
	%p
	% q
	\(Q\) & Hirzebruch quadratic form, Equation \eqref{eq:hirquadform} \\
	\(Q_L\) & Equation \eqref{eq:QL} \\	
	% r
	\(r(L)\) or \(r_L\) & codimension of a linear subspace \(L \subset \CP^n\) \\
	%s
	% t
	% general
	%\([n]\) &  set of integers \(\{1, \ldots, n\}\) \\
	% resolution
	\(\pi: X \to \CP^n\) & minimal De Concini-Procesi model -or log resolution- of \(\mH\)  \vspace{2mm}\\
    \hline
\end{longtable}
%\newpage

\section{Introduction}

A \emph{polyhedral K\"ahler} (PK) metric on a complex manifold \(X\) (see Definition \ref{def:pkmetric}) is a piecewise Euclidean (polyhedral) metric whose singular set is a divisor \(D \subset X\), and which restricts to a smooth K\"ahler metric on the complement \(X \setminus D\). 
In this paper, we study PK metrics on \(\CP^n\) whose singular set is a finite collection of complex hyperplanes and whose cone angles are less than \(2\pi\). Our main result reduces the classification of such metrics to a purely combinatorial problem on weighted hyperplane arrangements, extending to higher dimensions the main result in \cite{panov} for \(n=2\). Moreover, it establishes the zero curvature case of the conjecture in \cite[p. 4]{miyaokayau}, which characterizes K\"ahler-Einstein metrics with conical singularities along a hyperplane arrangement that have constant holomorphic sectional curvature.

\subsection{Main result}

Our main result is the following.

\begin{theorem}\label{thm:main}
	Let \(\mH\) be a finite collection of complex hyperplanes in \(\CP^n\). 
    For each hyperplane \(H \in \mH\), fix a real number \(\alpha_H \in (0,1)\).
    Then the following conditions \emph{(1)} and \emph{(2)} are equivalent.
	\begin{enumerate}[leftmargin=*, itemsep=0.3em, label=\textup{(\arabic*)}]
		\item There exists a PK metric on \(\CP^n\)  with cone angles \(2\pi\alpha_H\) along the hyperplanes \(H \in \mH\).
		\item The weighted arrangement \((\mH, \ba)\), where \(\ba \in \R^{\mH}\) has components \(a_H = 1- \alpha_H\), satisfies the following conditions:
		
		\textup{(i)}
		\begin{equation}\label{eq:CY}
			\sum_{H \in \mH} a_H = n+ 1 \,;
		\end{equation}
		
		\textup{(ii)} for every non-empty and proper linear subspace \(L \subset \CP^n\) obtained as an intersection of members of \(\mH\),
		\begin{equation}\label{eq:klt}
		\sum_{H \,|\, L \subset H} a_H < r(L) \,,
		\end{equation}
		  where  \(r(L) = \codim_{\CP^n} L\) is the codimension of \(L\) ;
		
		\textup{(iii)} 
		\begin{equation}\label{eq:hir0}
			Q(\ba) = 0 \,,
		\end{equation}
		where \(Q\) is the Hirzebruch quadratic form of \(\mH\)\,.
	\end{enumerate}
\end{theorem}

The Hirzebruch quadratic of \(\mH\), introduced in \cite{miyaokayau},
is the homogeneous polynomial of degree \(2\) on \(\R^{\mH}\) defined by
\begin{equation}\label{eq:hirquadform}
Q(\ba) = \sum_{L \in \mG^2} a_L^2
- \frac{1}{2} \sum_{H \in \mH} B_H \, a_H^2
- \frac{1}{2(n+1)} \, s^2 .
\end{equation}
Here \(\mG^2\) denotes the set of all codimension \(2\) linear subspaces
\(L \subset \CP^n\) obtained as the intersection of three or more hyperplanes in \(\mH\).
For \(L \in \mG^2\), \(a_L\) is the linear function on \(\R^{\mH}\) defined by
\[
2 \, a_L = \sum_{H \,|\, L \subset H} a_H \,.
\]
The coefficients \(B_H\) are the integers defined by
\[
B_H = \bigl|\{ L \in \mG^2 \mid L \subset H \}\bigr| - 1 \,.
\]
Finally, \(s\) denotes the total sum of the weights, \(s = \sum_{H \in \mH} a_H\).

\subsection{Examples}

\begin{example}
	Suppose that \(n=1\) and let \((\mH, \ba)\) be a weighted arrangement in \(\CP^1\) consisting of points \(x_i \in \CP^1\) together with weights \(a_i \in (0,1)\). In this case, the conditions (i), (ii), and (iii) of Theorem \ref{thm:main} reduce to the single equation \(\sum a_i = 2\). The existence of a flat K\"ahler metric on \(\CP^1\) with cone angles \(2\pi\alpha_i\) at the points \(x_i\) (with \(\alpha_i=1-a_i\)) is well-known \cite{troyanov} and the necessary and sufficient condition \(\sum a_i = 2\) is the familiar Gauss-Bonnet constraint \(\sum (1-\alpha_i) = \chi(S^2)\).
\end{example}

In dimensions \(n \geq 2\) the vanishing of the Hirzebruch quadratic form imposes strong conditions on the combinatorics of arrangements, which makes them very rare. For example, when \(n=2\) the complements of such arrangements in \(\CP^2\) are \(K(\pi, 1)\) \cite{panovpetrunin}. One of the simplest PK metrics on \(\CP^2\) we know of is the following.

\begin{example}[7 lines, {\cite[Example 3]{panov}}]
	Let  \(\alpha_1, \alpha_2, \alpha_3 \in (0,1)\) be such that \(\alpha_1 + \alpha_2 + \alpha_3 = 1\).
	Let \(T \subset \R^2\) be an Euclidean triangle with interior angles \(\pi\alpha_1, \pi\alpha_2, \pi\alpha_3\). The double of \(T\) defines a flat K\"ahler metric on \(\CP^1\) with \(3\) cone points of angles  \(2\pi\alpha_1, 2\pi\alpha_2, 2\pi\alpha_3\). The symmetric product \(\textup{Sym}^2(\CP^1) \cong \CP^2\) has an induced polyhedral metric with cone angles \(\pi\) along a smooth conic (the image of the diagonal in \(\CP^1 \times \CP^1\)) and cone angles \(2\pi\alpha_1, 2\pi\alpha_2, 2\pi\alpha_3\) along \(3\) lines tangent to it, which we may assume to be \(\{x=0\}\), \(\{y=0\}\), and \(\{z=0\}\). The pullback of this metric by the ramified cover \([x, y, z] \mapsto [x^2, y^2, z^2]\) is a PK metric on \(\CP^2\) whose singular set is an arrangement of \(7\) lines. The cone angle is \(\pi\) along \(4\) lines (the preimage of the conic) and \(4\pi\alpha_1, 4\pi\alpha_2, 4\pi\alpha_3\) along the other \(3\) lines.
\end{example}

\begin{example}[Braid arrangement]
	The essential braid arrangement consists of the \(\binom{n+2}{2}\) hyperplanes \(H_{ij} \subset \CP^n\) obtained by projection of the diagonal hyperplanes \(\{x_i = x_j\}\) in \(\C^{n+2}\) down to \(\P \, \big(\C^{n+2} / \C \cdot \mathbf{1} \big) = \CP^n\). Let \(a_1, \ldots, a_{n+2} \in (0,1)\) be such that \(\sum a_i = 1\). It is easy to check that the weighted arrangement \(\{(H_{ij}, a_{ij})\}\) with weights \(a_{ij} = a_i+ a_j\) satisfies (i) and (ii) in Theorem \ref{thm:main}. By \cite[Lemma 6.8]{miyaokayau},  \(Q(\ba) = 0\). Therefore, by Theorem \ref{thm:main}, there is a PK metric on \(\CP^n\) with cone angles \(2\pi(1-a_i-a_j)\) along the hyperplanes \(H_{ij}\); c.f. {\cite[Corollary 3.23]{chl}}.
\end{example}

\begin{example}[Complex reflection arrangements]
	Couwenberg-Heckman-Looijenga \cite{chl} construct families of constant holomorphic sectional curvature K\"ahler metrics on complements of irreducible complex reflection arrangements \(\mH \subset \CP^n\), i.e., \(\mH\) is the set of reflecting hyperplanes of an irreducible complex reflection group \(G \subset U(n+1)\).
	Corollary \ref{cor:cra} in the next section recovers their metrics in the flat case, moreover it shows that the metric completion is a polyhedral metric on \(\CP^n\).
\end{example}

\subsection{Applications}

We apply Theorem \ref{thm:main} to prove Conjecture 1.3 in \cite{miyaokayau}, which characterizes PK metrics as extremizers of a quadratic inequality. 

\begin{definition}
A weighted arrangement is a pair \((\mH, \ba)\) consisting of a hyperplane arrangement \(\mH\) in \(\CP^n\) and a weight vector \(\ba \in \R^{\mH}\) with positive components \(a_H >0\). The weighted arrangement \((\mH, \ba)\) is \emph{Calabi-Yau} (CY) if
\begin{equation}\label{eq:CY2}
		\sum_{H \in \mH} a_H = n+ 1 \,.
\end{equation}
The weighted arrangement \((\mH, \ba)\) is \emph{klt}		
if for every non-empty and proper linear subspace \(L \subset \CP^n\) obtained as an intersection of members of \(\mH\),
\begin{equation}\label{eq:klt2}
	\sum_{H \,|\, L \subset H} a_H < r(L) \,,
\end{equation}
where  \(r(L) = \codim_{\CP^n} L\) is the codimension of \(L\).
\end{definition}

\begin{remark}
    Equation \eqref{eq:klt2} applied to a hyperplane \(L = H\) in \(\mH\) reads \(a_H < 1\). Thus, if \((\mH, \ba)\) is klt, then \(a_H \in (0,1)\) for all \(H \in \mH\).
\end{remark}

\begin{remark}
    Conditions (i) and (ii) in Theorem \ref{thm:main} can be restated as saying that \((\mH, \ba)\) is CY and klt respectively.     
\end{remark}

The next result proves Conjecture 1.3 in \cite{miyaokayau} and extends Theorem 1.12 in \cite{panov} from \(n=2\) to higher dimensions.

\begin{corollary}\label{cor:my}
	Suppose that the weighted arrangement \((\mH, \ba)\) is klt and CY, then 
	\begin{equation}\label{eq:my}
		Q(\ba) \leq 0 \,.
	\end{equation}
	Equality holds if and only if there is a PK metric with cone angles \(2\pi(1-a_H)\) along the hyperplanes \(H \in \mH\).
\end{corollary}

\begin{proof}
	Inequality \eqref{eq:my} was first proved in \cite[Theorem 1.1]{miyaokayau} and later, by more elementary methods, in \cite[Theorem 3.1]{dunkl}. By Theorem \ref{thm:main}, \(Q(\ba) = 0\) if and only if  there is a PK metric with cone angles \(2\pi(1-a_H)\) along the hyperplanes \(H \in \mH\).
\end{proof}

In the symmetric case where all weights \(a_H\) are equal, we have the following.

\begin{corollary}\label{cor:hir}
	Let \(\mH \subset \CP^n\) be an arrangement consisting of \(N > n+1\) hyperplanes.
	Suppose that the following conditions are satisfied.
	\begin{enumerate}[label=\textup{(H\arabic*)}]
		\item For every non-empty and proper element \(L \in \mL\)   of the intersection poset of \(\mH\),
		\begin{equation}\label{eq:hir1}
			\frac{m_L}{r(L)} < \frac{N}{n+1} \,,
		\end{equation} 
		where \(m_L = |\mH_L|\) is the multiplicity of \(L\).
		\item For every \(H \in \mH\),
		\begin{equation}\label{eq:hir2}
			\vert \mH^H \vert = \left( 1 - \frac{2}{n+1} \right) N + 1 \,,
		\end{equation}
		where \(\mH^H = \{H' \cap H \,|\, H' \in \mH \setminus \{H\}\}\) is the restriction of \(\mH\) to \(H\).
	\end{enumerate}
	Then there exists a PK metric on \(\CP^n\) with cone angles \(2\pi \cdot \big(1-\frac{n+1}{N}\big)\) along the hyperplanes \(H \in \mH\).
\end{corollary}

\begin{proof}
	Consider the weighted arrangement \((\mH, \ba)\) with \(\ba = \frac{n+1}{N} \cdot \mathbf{1}\). Then \((\mH, \ba)\) satisfies the following conditions.
	\begin{enumerate}[label=\textup{(\roman*)}]
		\item \((\mH, \ba)\) is CY. Indeed, \(\sum a_H = N \cdot \frac{n+1}{N} = n+1\).
		\item \((\mH, \ba)\) is klt. If \(L \in \mL\) is non-empty and proper, then 
		\[
		\sum_{H | L \subset H} a_H = m_L \cdot \frac{n+1}{N} < r(L) \,,
		\]
        where the inequality follows by \eqref{eq:hir1}.
		\item \(Q(\ba) = 0\). By \cite[Lemma 7.12]{miyaokayau}, Equation \eqref{eq:hir2} implies that the vector \(\mathbf{1}\) is in the kernel of the quadratic form \(Q\), by homogeneity \(Q(\ba) = 0\).
	\end{enumerate}
	The existence of the PK metric then follows from Theorem \ref{thm:main}.
\end{proof}

The next result recovers {\cite[Theorem 3.7]{chl}} in the case where \(\kappa_0=1\).

\begin{corollary}\label{cor:cra}
	Let \(\mH\) be an irreducible complex reflection arrangement in \(\CP^n\) with \(N\) hyperplanes.
	Then there is a PK metric on \(\CP^n\) with cone angles \(2\pi \cdot \left(1 - \frac{n+1}{N}\right)\) along the hyperplanes \(H \in \mH\).	
\end{corollary}

\begin{proof}
	We verify that \(\mH\) satisfies (H1) and (H2) in Corollary \ref{cor:hir}:
	(H1) is \cite[Lemma 7.25]{miyaokayau}, and (H2) is \cite[Theorem 6.97]{orlikterao}.
\end{proof}

If the hyperplane arrangement \(\mH\) satisfies (H1) and (H2) in Corollary \ref{cor:hir}, then  we say that \(\mH\) is a \emph{stable Hirzebruch arrangement}. As we show above, if \(\mH\) is an irreducible complex reflection arrangement then \(\mH\) is a stable Hirzebruch arrangement. An open problem, posed by Hirzebruch in the case \(n=2\) \cite{hirzebruch2}, asks whether the converse is true. For a proof in the case of real line arrangements, see \cite{dima}.

\subsection{Outline}	

In Section \ref{sec:pkmetviaconn}, we collect results from \cite{pkc} that relate PK metrics on a complex manifold to flat torsion-free unitary logarithmic connections on the tangent bundle. This correspondence between PK metrics and logarithmic connections is fundamental in our proof of Theorem \ref{thm:main}.

In Section \ref{sec:cohomology} we recall the canonical log resolution \(\pi: X \to \CP^n\) of a hyperplane arrangement \(\mH \subset \CP^n\) given by the minimal De Concini-Procesi model of \(\mH\) \cite{conciniprocesi}. We recall a basis for the cohomology groups \(H^{2k}(X, \Z)\) constructed by Yuzvinsky \cite{yuzvinsky}
given by \emph{basic monomials}. For the cohomology group \(H^4(X, \Z)\), we provide explicit formulas which express non-basic monomials in terms of basic ones. We use these formulas later in Sections \ref{sec:1implies2} and \ref{sec:2implies1}.

In Section \ref{sec:1implies2} we prove the (1) \(\implies\) (2) direction of Theorem \ref{thm:main}. To do this we pullback the Levi-Civita connection of the PK metric to the log resolution \(X\) and use Ohtsuki's formula \cite{ohtsuki} for the Chern classes of a vector bundle in terms of the residues of a logarithmic connection. 

In Section \ref{sec:2implies1} we prove the (2) \(\implies\) (1) direction of Theorem \ref{thm:main}. To do this we consider a natural parabolic structure \(\mE_{*}\) on the pullback of \(T\CP^n\) to the log resolution \(X\). The parabolic bundle \(\mE_{*}\) depends only on the weighted arrangement \((\mH, \ba)\) and was introduced in \cite{miyaokayau}. Items (i) and (ii) imply (by \cite{miyaokayau}) that \(\parc_1(\mE_{*}) = 0\) and that \(\mE_{*}\) is slope stable. We show that item (iii) is equivalent to \(\parch_2(\mE_{*}) = 0\). The existence of the PK metric then follows as a consequence of the Kobayashi-Hitchin correspondence for parabolic bundles due to Mochizuki \cite{mochizuki}.

\subsubsection{Acknowledgments}
We would like to thank the London Mathematical Society's Research in Pairs Grant Scheme which allowed us to work on this project during October 2024 to January 2025.

\section{PK metrics via logarithmic connections}\label{sec:pkmetviaconn}
In this section, we revisit the correspondence from \cite{pkc} between PK metrics on a complex manifold \(X\) and (flat, torsion-free, unitary) logarithmic connections on \(TX\). 

In Section \ref{sec:pkmetric} we recall the notion of a PK metric on a complex manifold.  
In Section \ref{sec:logconnection} we review material on logarithmic connections and introduce the notion of a logarithmic connection \(\nabla\) on \(TX\) being adapted to the weighted arrangement \(\{(D_i, a_i)\}\) (see Definition \ref{def:adaptconn}).
In Section \ref{sec:pktoconn} we prove Proposition \ref{prop:mettoconect} which asserts that if \(g\) is a PK metric on a complex manifold \(X\) whose singular set is an arrangement-like divisor \(D = \sum D_i\) (see Definition \ref{def:arrlike}) then the Levi-Civita connection of \(g\) is a logarithmic connection \(\nabla\) on \(TX\) adapted to the weighted arrangement \(\{(D_i, a_i)\}\) where \(a_i = 1-\alpha_i\) and \(2\pi\alpha_i\) is the cone angle of \(g\) along \(D_i\). 
In Section \ref{sec:conntopk} we prove Theorem \ref{thm:connmet} which shows the converse direction. 

%If \(\nabla\) is a flat, torsion-free, unitary logarithmic connection on \(TX\) with poles along the arrangement-like divisor \(D = \sum D_i\)
%and \(\nabla\) is adapted to \(\{(D_i, a_i)\}\), where \(a_i \in (0,1)\) are such that the pair \((X, \sum a_i D_i)\) is klt, then the Hermitian metric preserved by \(\nabla\) defines a PK metric on \(X\) with cone angles \(2\pi(1-a_i)\) along \(D_i\).

\subsection{PK metrics on complex manifolds}\label{sec:pkmetric}

\subsubsection{The \(2\)-cone}
Let \(\alpha\) be a positive real number.
We denote by \(C(2\pi\alpha)\) the Euclidean cone over a circle of length \(2\pi\alpha\). 
Alternatively, consider the function \(\xi = z^{\alpha}\) defined for \(z\) in the cut complex plane, say \(\C \setminus (-\infty, 0]\), and let \(g_{\alpha}\) be the pullback of the Euclidean metric \(|d\xi|^2\) given by
\begin{equation}
g_{\alpha} = \alpha^2 |z|^{2\alpha-2}|dz|^2 \,.	
\end{equation}
The metric completion of \(g_{\alpha}\) is obtained by gluing the shores of an Euclidean wedge of total angle \(2\pi\alpha\) by means of a rotation and is thus isometric to \(C(2\pi\alpha)\). 

The Levi-Civita connection of \(g_{\alpha}\) in the trivialization of \(T\C\) given by the coordinate vector field \(\p_z\) is given by \(\nabla = d - \omega\), where \(\omega\) is the \(1\)-form defined by \(\nabla \p_z = -\omega \otimes \p_z\).
To calculate \(\omega\), note that the locally defined vector field \( z^{1-\alpha}\p_z = \alpha \p_{\xi}\) is parallel. Let \(a = 1-\alpha\),
then \(\nabla (z^a \p_z) = 0\) expands to \(d(z^a) - z^{a} \omega = 0\), so \(\omega = (a/z)dz\) and 
\begin{equation}\label{eq:conn2cone}
\nabla = d \,-\, a \, \frac{dz}{z} 	\,.
\end{equation}
In particular, \(\nabla\) is a \emph{logarithmic connection} (see Definition \ref{def:logconn}) on \(T\C\) with a pole at \(0\).

\subsubsection{Divisors}
Let \(X\) be a complex manifold of dimension \(n\). Throughout this paper, a \emph{divisor} on \(X\) will mean a closed analytic subset \(D \subset X\) of codimension one. 
Unless otherwise specified, we consider divisors without multiplicity.

%A \emph{defining equation} of \(D\) is a holomorphic function \(h \in \mO(U)\) defined on an open set \(U \subset X\) such that: (i) \(D \cap U = \{h=0\}\), and (ii)  \(h\) is square free, i.e, if \(f\) is a holomorphic function on an open subset \(U' \subset U\) vanishing on \(D \cap U'\), then \(f\) is divisible by \(h\) on \(U'\).

The \emph{smooth locus} of \(D\) is the set of points \(D^{\circ} \subset D\) at which \(D\) is a complex submanifold of \(X\).
The \emph{singular locus} \(\Sing(D) = D \setminus D^{\circ}\) is a closed analytic subset of \(X\) of codimension \(\geq 2\).

%If \(h \in \mO(U)\) is a defining equation of \(D\), then \(\Sing(D) \cap U = \{h = 0 \,,\, dh =0\} \,.\)
%Equivalently, \(x \in D^{\circ}\) if there are complex coordinates \((z_1, \ldots, z_n)\) centred at \(x\) such that \(D \cap U = \{z_1=0\}\). 
%If  \(x \in D\) and \(h\) is a defining equation of \(D\) near \(x\), then \(x \in D^{\circ}\) if and only if \(dh(x) \neq 0\).

%The divisor \(D\) is \emph{irreducible} if it does not properly contain any other divisor. 
%An \emph{irreducible component} of \(D\) is an irreducible divisor \(D_i\) on \(X\) contained in \(D\).
%The divisor \(D\) is equal to the union of all its irreducible components.
%We write \(D = \sum D_i\) and refer to it as the \emph{irreducible decomposition} of \(D\).

\subsubsection{Polyhedral K\"ahler metrics}
Let \((X, d)\) be a complete length space (see \cite{BuragoBuragoIvanov2001}). We say that \((X, d)\) is \emph{polyhedral} if it admits a locally finite triangulation such that each simplex is isometric to a simplex in Euclidean space (see \cite{petruninlebedeva}). 

\begin{definition}\label{def:pkmetric}
	Let \(X\) be a complex manifold of dimension \(n\) and let
	\(D \subset X\) be a divisor with irreducible decomposition \(D = \sum D_i\).
	For each \(D_i\) let \(\alpha_i\) be a positive real number.
	We say that \(g\) is a \emph{polyhedral K\"ahler} (PK) metric on \(X\) with cone angles \(2\pi\alpha_i\) along \(D_i\) if the following hold:
	\begin{enumerate}[label=\textup{(\roman*)}]
		\item \(g\) is a flat K\"ahler metric on \(X \setminus D\)\,;
		\item \(g\) is holomorphically isometric to \(C(2\pi\alpha_i) \times \C^{n-1}\) near points in \(D_i \cap D^{\circ}\)\,;
		\item the distance function \(d_g\) induced by \(g\) on \(X \setminus D\) extends continuously across \(D\) to the whole \(X\) so that \((X, d_g)\) is a complete length space, the distance \(d_g\) induces the original manifold topology of \(X\), and \((X, d_g)\) is polyhedral.
	\end{enumerate}
\end{definition}

\begin{remark}
	A more general definition of PK metric, in the context of piecewise linear manifolds, is given in \cite[Definition 1.1]{panov}. In this level of generality, it becomes a non-trivial question to understand whether a parallel complex structure on the smooth locus extends over the metric singular set as a complex analytic space. Definition \ref{def:pkmetric} circumvents this problem by restricting to the class of PK metrics which admit parallel complex structures that extend smoothly over the whole manifold.
\end{remark}

\begin{remark}
    Definition \ref{def:pkmetric} is equivalent to the seemingly weaker \cite[Definition 6.1]{pkc}. Namely, it is enough to define \(g\) as a polyhedral metric on \(X\) that is K\"ahler on its regular locus. The fact that the singular set of \(g\) is a divisor \(D \subset X\) follows from \cite[Proposition 6.34 (3)]{pkc} and the fact that \(g\) is locally isometric to \(C(2\pi\alpha) \times \C^{n-1}\) near points of the smooth locus \(D^{\circ}\) follows from \cite[Corollary 6.65]{pkc}.
\end{remark}

\subsection{Logarithmic connections}\label{sec:logconnection}
Let \(\mE \to X\) be a holomorphic vector bundle of rank \(r\) and let \(D \subset X\) be a divisor with irreducible decomposition \(D = \sum D_i\). Let \(\nabla\) be a holomorphic connection on \(\mE|_{X \setminus D}\).
	
\begin{definition}\label{def:logconn}
	We say that \(\nabla\) is a \emph{logarithmic connection} on \(\mE\) with poles along \(D\) if for every point \(x \in D^{\circ}\) (the smooth locus of \(D\)) we can find complex coordinates \((z_1, \ldots, z_n)\) centred at \(x\) with \(D = \{z_1=0\}\) and a frame of holomorphic sections \(s_1, \ldots, s_r\) of \(\mE\) over \(U\) such that \(\nabla = d - \Omega\) with
	\begin{equation}\label{eq:logconn}
	\Omega = A_1 \frac{dz_1}{z_1} + \sum_{i> 1} A_i dz_i \,,
	\end{equation}
	where \(A_1, \ldots, A_n\) are \(r \times r\) matrices of holomorphic functions.
\end{definition}

\begin{definition}
	The \emph{residue} of \(\nabla\) is the holomorphic section of \(\End (\mE)|_{D^{\circ}}\) defined in a local trivialization \eqref{eq:logconn} by
	\begin{equation}\label{eq:resdef}
	\Res_D(\nabla) = A_1|_{\{z_1=0\}} \,\,.	
	\end{equation}
	
	We write \(\Res_{D_i}(\nabla)\) for the restriction of \(\Res_{D}(\nabla)\) to \(D_i \cap D^{\circ} \). 
\end{definition}

\begin{remark}
	The fact that Equation \eqref{eq:resdef} defines a section of \(\End(\mE)\) follows easily from the formula for the change of connection form \(\Omega' = G \cdot \Omega \cdot G^{-1} + \hol\) under a change of frame \((s_1, \ldots, s_r) = (s'_1, \ldots, s'_r) \cdot G\), see Appendix \ref{app:conn}.
\end{remark}

\subsubsection{Logarithmic connections with holomorphic residues}

\begin{definition}\label{def:holres}
    Let \(\nabla\) be a logarithmic connection on \(\mE \to X\) with poles along \(D\).
	Suppose that the irreducible components \(D_i\) of \(D\) are smooth.
	We say that \(\nabla\) has \emph{holomorphic residues} if the sections \(\Res_{D_i}(\nabla)\) extend holomorphically across \(\Sing(D)\) as a holomorphic sections of \(\End \mE\) on \(D_i\).
\end{definition}

\begin{lemma}\label{lem:holres}
	Let \(\nabla\) be a logarithmic connection on \(\mE \to X\) with poles along \(D\). Assume that the irreducible components \(D_i\) of \(D\) are smooth. Then \(\nabla\) has holomorphic residues if and only if for every \(x \in D\) there is an open neighbourhood \(U\) of \(x\) such that \(\mE|_U\) is trivial and, in a (and hence any) holomorphic trivialization of \(\mE\) over \(U\), we have 
	\begin{equation}\label{eq:holres}
	\nabla = d \,-\, \sum_{i\,|\, x \in D_i} A_i \frac{dh_i}{h_i} \,+\, \hol \,,
	\end{equation}
	where \(h_i \in \mO(U) \) are defining equations of \(D_i\) and \(A_i\) are \(r \times r\) matrices of holomorphic functions on \(U\).
\end{lemma}

\begin{proof} One direction is clear, if \(\omega\) is locally of the form given by Equation \eqref{eq:holres} then \(\Res_{D_i}(\nabla)=A_i|_{\{h_i=0\}}\) are holomorphic on \(D_i\). 

	To show the converse, note that in a local trivialization we have \(\nabla = d - \Omega\), where the entries of \(\Omega\) are logarithmic \(1\)-forms \(\omega_{ij}\) with holomorphic residues. The statement then follows from Lemma \ref{lem:holres1form} in the appendix applied to the \(1\)-forms \(\omega_{ij}\).
\end{proof}

In the case of simple normal crossing divisors, the holomorphic residues condition is always satisfied, as proved by the next.

\begin{proposition}\label{prop:normcross}
	Suppose that \(\nabla\) is a logarithmic connection with poles along a simple normal crossing divisor. Then \(\nabla\) has holomorphic residues.
\end{proposition}

\begin{proof}
    The condition of having holomorphic residues is local, so we might as well assume that \(\nabla = d - \Omega\) is a connection on the trivial vector bundle \(\underline{\C}^r\) on an open ball \(B \subset \C^n\) with logarithmic poles along the \(k\) coordinate hyperplanes \(D = \{z_1 \cdot \ldots \cdot z_k = 0\}\). The entries of the connection matrix \(\Omega\) are logarithmic \(1\)-forms. 
    
    The proposition then follows from the well-known fact that if \(\omega\) is a logarithmic \(1\)-form on \(B\) with poles along \(D\) then there are  \(a_i \in \mO(B)\) for \(1 \leq i \leq n\) such that
	\begin{equation}\label{eq:normcoross}
	\omega = \sum_{i=1}^{k} a_i \, \frac{dz_i}{z_i} \,+\, \sum_{i>k} a_i \, dz_i \,\,.
	\end{equation} 
    See \cite[p. 449-450]{griffithsharris}, or \cite[Theorem 3.8]{novikovyakovenko}.
\end{proof}

\begin{remark}
    Proposition \ref{prop:normcross} fails if the polar divisor is not normal crossing. For example, the residues of the logarithmic connection on the trivial line bundle on \(\C^2\) defined by \(\nabla = d - \omega\) where \(\omega\) is as in Example \ref{ex:3lines} in the Appendix have a pole at the origin.
\end{remark}

\subsubsection{Flat logarithmic connections}
A logarithmic connection on \(\mE \to X\) with poles along \(D \subset X\) is \emph{flat} if it is so on \(X \setminus D\). In a local representation, the connection \(\nabla = d - \Omega\) is flat if and only if \(d\Omega = \Omega \wedge \Omega\).

\begin{lemma}\label{lem:ressnc}
	Let \(\nabla\) be a flat logarithmic connection on \(\mE \to X\) with poles along a simple normal crossing divisor \(D = \sum D_i\). Then for any pair of irreducible components \(D_i\) and \(D_j\) of \(D\) we have
	\begin{equation}
	[\Res_{D_i}(\nabla), \Res_{D_j}(\nabla)] = 0
	\end{equation}
	along \(D_i \cap D_j\).
\end{lemma}

In the above lemma, note that \(\Res_{D_i}(\nabla)\) and \(\Res_{D_j}(\nabla)\) are defined along \(D_i \cap D_j\) because, by Proposition \ref{prop:normcross}, the connection \(\nabla\) has holomorphic residues. Lemma \ref{lem:ressnc} is standrad and we omit the proof; see for example \cite[Theorem 4.3]{novikovyakovenko}.

\subsubsection{Flat logarithmic connections with poles along a hyperplane arrangement}
Let \(\nabla\) be a flat logarithmic connection on the trivial bundle \(\mE = \underline{\C}^r\) defined on an open neighbourhood \(U \subset \C^n\) of the origin and having poles along a finite collection \(\mH\) of linear hyperplanes \(H \subset \C^n\). Suppose that \(\nabla\) has holomorphic residues, so we can write \(\nabla = d - \Omega\) with
\begin{equation}
\Omega = \sum_{H \in \mH} A_H(z) \cdot \frac{dh}{h} \,+\, \hol \,,    
\end{equation}
where \(A_H\) are \(r \times r\) matrices of holomorphic functions on \(U\) and \(h\) are linear functions with \(H = \{h=0\}\). 

Let \(L \subset \C^n\) be a subspace obtained as intersection of members of \(\mH\) and let \(\mH_L = \{H \in \mH \,|\, L \subset H\}\).
Write
\begin{equation}\label{eq:genlocus}
L^{\gen} = L \setminus \bigcup_{H \in \mH \setminus \mH_L} H    
\end{equation}
for the `generic' locus of \(L\).
For \(x \in L\) define
\begin{equation}
    A_L(x) := \sum_{H \in \mH_L} A_H(x) \,.
\end{equation}

\begin{lemma}\label{lem:ALcommute}
    Suppose that \(L \subset \C^n\) is a codimension \(2\) subspace obtained as intersection of members of \(\mH\) and let \(H \in \mH_L\). Then at every point \(x \in L\) the following holds:
	\begin{equation}\label{eq:commute}
	[A_L(x)\,,\, A_H(x)] = 0 \,.
	\end{equation}
\end{lemma}

\begin{proof}
	By continuity, it is enough to prove Equation \eqref{eq:commute} for \(x \in L^{\gen}\). Let \(x \in L^{\gen}\) and take a small open ball \(B \subset U\) centred at \(x\) with \(B \cap H =  \emptyset\) for all \(H \in \mH \setminus \mH_L\). On the ball \(B\), the polar set of the flat logarithmic connection \(\nabla\) is equal to \(\mH_L\).
	We divide the proof into two cases. 
	
	Case 1: \(\mH_L = \{H, H'\}\). In this case the support of \(\mH_L\) is a simple normal crossing divisor, and
	\(
	[A_L(x), A_H(x)] = [A_{H'}(x), A_H(x)] = 0 \,,
	\)
	where the last equality follows from Lemma \eqref{lem:ressnc}.
	
	Case 2: \(|\mH_L| \geq 3\). Consider the blowup \(\pi: \widetilde{B} \to B\) of \(B\) along \(L\) and let \(\tn = \pi^* \nabla\) be the pullback connection. An easy calculation shows that \(\tn\) has logarithmic poles along the simple normal crossing divisor \(D = E + \sum_{H \in \mH_L} \tH\), where \(E = \pi^{-1}(L)\) is the exceptional divisor and \(\tH\) is the proper transform of \(H\). Moreover, if \(\pi(\tx) = x\) then \(\Res_E(\tn)(\tx) = A_L(x)\) if \(x \in L\) and \(\Res_{\tH}(\tn)(\tx) = A_H(x)\) if \(x \in H\).  
	
	For each \(H \in \mH_L\), the map \(\pi\) restricts to a bijection between \(\tH \cap E\) and \(L\). For \(x \in L\), let \(\tx \in \tH \cap E\) with \(\pi(\tx) = x\). Then 
	\[
	[A_L(x)\,,\, A_H(x)] = [\Res_E(\tn)(\tx) \,,\, \Res_{\tH}(\tn)(\tx)] = 0 \,,
	\]
	where the last equality follows from Lemma \eqref{lem:ressnc}.
\end{proof}

\begin{lemma}\label{lem:flatat0}
	The connection \(\nabla_0 = d - \Omega_0\) given by
	\begin{equation}\label{eq:stdconn0}
	\Omega_0 = \sum_{H \in \mH} A_H(0) \cdot \frac{dh}{h}
	\end{equation}
	with constant residues \(A_H(0) \in M_r(\C)\) is flat.
\end{lemma}

\begin{proof}
	Since \(d\Omega_0 = 0\), the flatness of \(\nabla_0\) is equivalent to \(\Omega_0 \wedge \Omega_0 = 0\). By \cite[Proposition 3.8]{pkc} the equation \(\Omega_0 \wedge \Omega_0 = 0\) holds if and only if for any codimension \(2\) subspace \(L \subset \C^n\) and any \(H \in \mH_L\) we have \(\big[A_L(0), A_H(0)\big] = 0\). The result then follows by Lemma \ref{lem:ALcommute}.
\end{proof}

Connections of the form given by Equation \eqref{eq:stdconn0} are called \emph{standard connections} in \cite{pkc}.

\subsubsection{Locally standard logarithmic connections}

\begin{definition}\label{def:arrlike}
	The divisor \(D\) is said to be \emph{arrangement-like} if each of its irreducible components \(D_i\) is smooth, and for every point \(x \in D\), there exist local complex coordinates \(z = (z_1, \ldots, z_n)\) centred at \(x\) such that the components \(D_i\) passing through \(x\) are locally defined by linear equations, i.e., they correspond to a collection of linear hyperplanes in \(\mathbb{C}^n\). We say that \((z_1, \ldots, z_n)\) are coordinates \emph{adapted} to \(D\) at \(p\).
\end{definition}

\begin{definition} 
    Let \(\nabla\) be a logarithmic connection on \(\mE \to X\) with poles along the arrangement-like divisor \(D = \sum D_i\).
	We say that \(\nabla\) is \emph{locally standard} if, for every \(x \in D\) and any local complex coordinates \(z=(z_1, \ldots, z_n)\) adapted to \(D\) at \(x\), there is a holomorphic trivialization of \(\mE\) near \(x\) such that 
	\begin{equation}\label{eq:locstd}
	\nabla = d - \sum_{i \,|\, x \in D_i} A_i \cdot \frac{dh_i}{h_i} \,,	
	\end{equation}
	where \(A_i \in M_r(\C)\) are constant \(r \times r\) matrices and \(h_i\) are linear functions of \(z_1, \ldots, z_n\) with \(D_i = \{h_i=0\}\).
\end{definition}

Clearly, a locally standard connection has holomorphic residues.
To show the converse we introduce the next.

\begin{definition}\label{def:resonant}
    Let \(\nabla\) be a \emph{flat} logarithmic connection on \(\mE \to X\) with poles along the arrangement-like divisor \(D = \sum D_i\). Assume that \(\nabla\) has holomorphic residues. 
	We say that \(\nabla\) is \emph{non-resonant} if, for every \(x \in D\), the endomorphism of \(\mE|_x\) given by
	\[
	\sum_{i \,|\, x \in D_i} \Res_{D_i}(\nabla)(x)
	\]
	has no two eigenvalues that differ by a positive integer.
\end{definition}

\begin{proposition}\label{prop:locstnd}
	Let \(\nabla\) be a flat logarithmic connection on \(\mE \to X\) with poles along the arrangement-like divisor \(D = \sum D_i\).
	Assume that: \textup{(i)} \(\nabla\) has holomorphic residues, and \textup{(ii)} \(\nabla\) is non-resonant. Then \(\nabla\) is locally standard. 
\end{proposition}

\begin{proof}

\begin{comment}
    Let \(p \in D\) and let \(z=(z_1, \ldots, z_n)\) be adapted coordinates to \(D\) at \(p\). We want to find a holomorphic trivialization of \(\mE\) near \(p = 0\) such that \(\nabla\) is as in \eqref{eq:locstd}.
	
	Since \(\nabla\) has holomorphic residues, taking an arbitrary holomorphic trivialization of \(\mE\) near \(p\), we have \(\nabla = d - \Omega\) with
	\[
	\Omega = \sum_{i=1}^m A_i(z) \cdot \frac{dh_i}{h_i} \,+\, \hol  \,,
	\]
	where \(h_i\) are linear functions of \(z_1, \ldots, z_n\) and \(A_i(z)\) are \(r \times r\) matrices of holomorphic functions.
	Without loss of generality, by introducing terms with \(A_i=0\),
	we can assume that \(m \geq n\) and \(h_i = z_i\) for \(1 \leq i \leq n\). Write  \(\hol = \sum B_j dz_j\) with \(B_j\) holomorphic matrix valued functions. Let \(\tA_i = A_i + z_i B_i\) for  \(1 \leq i \leq n\) and \(\tA_i = A_i\) otherwise; then
	\[
	\Omega = \sum_{i=1}^m \tA_i(z) \cdot \frac{dh_i}{h_i} \,,
	\]
	with \(\tA_i\) holomorphic and \(\tA_i(0) = A_i(0)\) for all \(i\).
	
	The non-resonant condition implies that no two eigenvalues of \(\sum A_i(0)\) differ by a positive integer. 
\end{comment}
    In a local trivialization near \(x \in D\) we can write \(\nabla\) as in Equation \eqref{eq:holres}. Then, it follows from \cite{takano} (see also \cite[\S 7.5]{novikovyakovenko}) that there is a holomorphic trivialization of \(\mE\) near \(x\) such that
	\begin{equation}\label{eq:normalform}
	\nabla = d - \sum_i A_i(0) \cdot \frac{dh_i}{h_i} 	
	\end{equation}
	as wanted.
\end{proof}

\subsubsection{Adapted logarithmic connections}

\begin{definition}\label{def:adaptconn}
	Let \(\nabla\) be a logarithmic connection on \(TX\) with poles along the arrangement-like divisor \(D = \sum D_1\). For each irreducible component \(D_i\) of \(D\) let \(a_i \in \C\). We say that \(\nabla\)  is \emph{adapted} to \(\{(D_i, a_i)\}\) if the following conditions hold:
	\begin{enumerate}[label=\textup{(\roman*)}]
		\item \(\ker(\Res_{D}(\nabla)) = TD\) along \(D^{\circ}\)\,;
		\item \(\Res_{D_i}(\nabla)\) has eigenvalues \((0, a_i)\),\;
		\item \(\nabla\) has holomorphic residues.
	\end{enumerate}
\end{definition}

\begin{remark}
	If \(n=2\) and \(\nabla\) is flat, then our definition of adapted connection agrees with \cite[Definition 4.12]{panov}.
\end{remark}

\subsection{PK metric \(\implies\) logarithmic connection}\label{sec:pktoconn}

\begin{proposition}\label{prop:mettoconect}
    Let \(X\) be a complex manifold and let \(D \subset X\) be an arrangement-like divisor with irreducible decomposition \(D = \sum D_i\). Let \(g\) be a PK metric on \(X\) with cone angles \(2\pi\alpha_i > 0\) along \(D_i\). Let \(\nabla\) be the Levi-Civita connection of \(g\). Then \(\nabla\) is a logarithmic connection on \(TX\) adapted to \(\{(D_i, a_i)\}\), where \(a_i \in (-\infty, 1)\) are given by \(a_i = 1 - \alpha_i\).
\end{proposition}

\begin{proof}
	From Definition \ref{def:adaptconn} (adapted connection) we need to verify that \(\nabla\) satisfies the following properties:
	\begin{enumerate}[label=\textup{(\roman*)}]
		\item \(\nabla\) is holomorphic outside \(D\)\,;
		\item \(\nabla\) has logarithmic poles along \(D\)\,;
		\item \(\ker(\Res_{D_i}(\nabla)) = TD_i\) and \(\Res_{D_i}(\nabla)\) has eigenvalues \((0, a_i)\)\,;
		\item \(\Res_{D_i}(\nabla)\) extend holomorphically across \(\Sing(D)\) as sections of \(\End(TX|_{D_i})\)\,.
	\end{enumerate}
	
	(i) Follows from the fact that \(g\) is a flat K\"ahler metric on \(X \setminus D\).
	 
	(ii) and (iii) By Definition \ref{def:logconn}, to show that \(\nabla\) is logarithmic, it suffices to consider points in the smooth locus \(D^{\circ}\) of \(D\). Since \(D^{\circ}\) is the disjoint union of the intersections \(D^{\circ}\cap D_i\), we might as well assume that \(x \in D^{\circ}\cap D_i\). By item (ii) in the definition of PK metric (Definition \ref{def:pkmetric}) there are holomorphic coordinates \((z_1, \ldots, z_n)\) defined on an open neighbourhood \(U\) of \(x\) such that \(D_i = \{z_1=0\}\) and 
	\[
	g = |z_1|^{2\alpha_i - 2} |dz_1|^2  \,+\, \sum_{j > 1} |dz_j|^2 \,.
	\]
    
	Using Equation \eqref{eq:conn2cone} for the Levi-Civita connection of the \(2\)-cone \(C(2\pi\alpha)\), we deduce
	that, in the holomorphic trivialization of \(TX\) given by the coordinate vector fields \(\p/\p z_1, \ldots, \p / \p z_n\), we have 
	\[
	\nabla = d \,-\, \diag(a_i, 0, \ldots, 0)
	\cdot \frac{dz_1}{z_1} \,.
	\]
	This shows that \(\nabla\) has a logarithmic singularity near \(x\) and \(\Res_{D_i}(\nabla)\) is as in item (iii).
	
	(iv) Since the divisor \(D = \sum D_i\) is arrangement-like, we can apply \cite[Lemma 6.73]{pkc}, which guarantees that \(\nabla\) has holomorphic residues along \(D\). 
\end{proof}

\begin{comment}
	\[
	\Omega = \begin{pmatrix}
	a_i & 0 & \cdots & 0 \\
	0 & 0 & \cdots & 0 \\
	\vdots & \vdots & \ddots & \vdots \\
	0 & 0 & \cdots & 0
	\end{pmatrix}
	\cdot \frac{dz_1}{z_1}
	\]
\end{comment}

\subsubsection{Linear coordinates}

\begin{definition}
	Let \(g\) be a PK metric on the complex manifold \(X\) whose singular set is an arrangement-like divisor \(D\). We say that \((z_1, \ldots, z_n)\) are \emph{linear coordinates} at \(x\) if the Levi-Civita connection of \(g\) is standard in the trivialization of the tangent bundle given by coordinate vector fields.
\end{definition}

\begin{proposition}\label{prop:lincoordmet}
	Let \(g\) be a PK metric on the complex manifold \(X\) whose singular set is an arrangement-like divisor \(D\). Then we can find linear coordinates at every \(x \in X\).
\end{proposition}

\begin{proof}
	This is proved in \cite{pkc}, for completeness we quote the precise results.
	Let \(x \in X\) and let \(C_x\) be the polyhedral tangent cone of \(g\) at \(x\). By \cite[Lemma 6.32]{pkc} \(C_x\) has the natural structure of a complex manifold endowed with a PK cone metric. By \cite[Proposition 6.35]{pkc} we can take complex coordinates \((z_1, \ldots, z_n)\) centred at \(x\) that identify \(C_x \cong \C^n\) as complex manifolds and linearise the (complex) Euler vector field of the cone. By \cite[Proposition 6.71]{pkc} the Levi-Civita connection of \(g\) is standard with respect to the trivialization of the tangent bundle given by the coordinate vector fields \(\p / \p z_1, \ldots, \p / \p z_n\).
\end{proof}

\subsection{Logarithmic connection \(\implies\) PK metric}\label{sec:conntopk}

The main result of this section is Theorem \ref{thm:connmet}. To state it, we introduce the following.

\subsubsection{klt condition}
Let \(X\) be a complex manifold and let \(D \subset X\) be an arrangement-like divisor with irreducible decomposition \(D = \sum_{i} D_i\). For each \(D_i\), let \(a_i \in (0,1)\).
For \(x \in D\), let
\begin{equation}\label{eq:kltdef}
a_x := \frac{1}{\codim \Gamma}\sum_{i \,|\, x \in D_i} a_i \,,
\end{equation}
where \(\Gamma \subset M\) is the connected component of \(\bigcap_{x \in D_i} D_i\) that contains \(x\). In particular, if \(x \in D_i \cap D^{\circ}\) then \(a_x = a_i\).

\begin{definition}\label{def:klt}
	We say that \(\{(D_i, a_i)\}\) is klt if for all \(x \in D\) we have \(a_x \in (0,1)\).
\end{definition}

\begin{remark}
	The above definition agrees with the usual notion of the pair \((M, \Delta)\) being klt, where \(\Delta\) is the \(\R\)-divisor \(\Delta = \sum a_i D_i\).
\end{remark}

\subsubsection{Main result}
The following extends \cite[Theorem 4.13]{panov} to higher dimensions.

\begin{theorem}\label{thm:connmet}
	Let \(X\) be a compact complex manifold. Let \(D \subset X\) be an arrangement-like divisor with irreducible decomposition \(D = \sum D_i\). Let \(a_i \in (0,1)\) and suppose that \(\{(D_i, a_i)\}\) is klt.
	Let \(\nabla\) be a unitary flat torsion-free logarithmic connection on \(TX\) adapted to \(\{(D_i, a_i)\}\). 
	Then the Hermitian metric \(g\) preserved by \(\nabla\) defines a PK metric on \(X\) with cone angles \(2\pi(1-a_i)\) along \(D_i\).
\end{theorem}

The key technical input needed to prove Theorem \ref{thm:connmet} is given next.

\subsubsection{Linear coordinates}

\begin{proposition}\label{prop:linearcoord}
    Let \(X\) be a complex manifold and let \(D \subset X\) be an arrangement-like divisor with irreducible decomposition \(D = \sum D_i\). Let \(a_i \in (0,1)\) be such that \(\{(D_i, a_i)\}\) is klt.
	Let \(\nabla\) be a flat torsion-free logarithmic connection on \(TX\) adapted to \(\{(D_i, a_i)\}\). 
	Then for every \(x \in D\) there exist complex coordinates \((z_1, \ldots, z_n)\) centred at \(x\) such that, in the trivialization of the tangent bundle given the coordinate vector fields \(\p / \p z_1, \ldots, \p / \p z_n\), we have
	\[
	\nabla = d - \sum_{i \,|\, x \in D_i} A_i \frac{dh_i}{h_i} \,,
	\]
	where \(A_i \in M_n(\C)\) are constant matrices and \(h_i\) are linear functions of \(z_1, \ldots, z_n\) with \(D_i = \{h_i = 0\}\).	
\end{proposition}

We refer to coordinates \((z_1, \ldots, z_n)\) as above as \emph{linear coordinates} for \(\nabla\) at \(x\).
The following proof is similar to that of \cite[Proposition 4.5]{pkc}.

\begin{proof}[Proof of Proposition \ref{prop:linearcoord}]
	Let \(x \in D\) and let \(D_1, \ldots, D_m\) be the irreducible components of \(D\) that contain \(x\).
	Taking local coordinates \(t = (t_1, \ldots, t_n)\) adapted to \(D\) at \(x\), we can assume that \(X\) is an open ball \(B \subset \C^n\) about the origin, \(x=0\), and \(D = \sum_{i=1}^m D_i\) where \(D_i \subset \C^n\) are linear hyperplanes.	
	
	Since  \(\nabla\) is adapted it has holomorphic residues, and so by Lemma \ref{lem:holres} in the trivialization of \(T\C^n\) given by the coordinate vector fields \(\p/\p t_1, \ldots, \p/\p t_n\) we have 
	\[
	\nabla = d - \sum_{i=1}^m R_i(t) \cdot \frac{d\ell_i}{\ell_i} \,+\, \hol  \,,
	\]
	where \(R_i\) are holomorphic matrix valued functions defined on \(B\) and \(\ell_i\) are linear functions of \(t_1, \ldots, t_n\) with \(D_i = \{\ell_i = 0\}\).  
	
	\textbf{Step 1: associated standard connection.}
	Let \(A_i = R_i(0)\) and consider the standard connection on \(T\C^n\) given by
	\[
	\nabla_0 = d - \sum_{i=1}^m A_i \frac{d\ell_i}{\ell_i} 
	\]
	in the trivialization of \(T\C^n\) by \(\p/\p t_1, \ldots, \p/\p t_n\).
	By Lemma \ref{lem:flatat0}, the connection \(\nabla_0\) is flat. Moreover, since \(\ker A_i = TD_i\), by \cite[Lemma 3.14]{pkc} the connection \(\nabla_0\) is torsion-free.
	
	\textbf{Step 2: non-resonant condition.}
	We calculate the eigenvalues of the matrix 
    \begin{equation}\label{eq:Axmatrix}
        A_x = \sum_{i=1}^m A_i \,.
    \end{equation}
	To do this, we recall from \cite{pkc} some concepts regarding the arrangement of linear hyperplanes \(\{D_1, \ldots, D_m\}\) in \(\C^n\).
    
	Let \(T = \bigcap_{i=1}^m D_i\) be the centre of the arrangement and let \(L_1, \ldots, L_k\) be the irreducible components of \(T\) (\cite[Definition 2.38]{pkc}), in particular \(T = \bigcap_{j=1}^k L_j\) and \(\codim T = \sum_{j=1}^k \codim L_j\). Since the connection \(\nabla_0\) is flat and torsion free, the collection of residues \(\{A_1, \ldots, A_m\}\) satisfies conditions (F) and (T) of the Non-Zero Weight Assumptions in \cite[Assumptions 3.25]{pkc}. To check condition (NZ) in \cite[Assumptions 3.25]{pkc} we must verify that \(a_L \neq 0\) for all irreducible subspaces \(L\) of the arrangement, where \(a_L\) is the weight at \(L\) (\cite[Definition 3.22]{pkc}). To do this we note that \(a_L = a_y\) for any \(y\) in the generic locus \(L^{\gen}\) of \(L\) (as in Equation \eqref{eq:genlocus}). The klt condition implies that the weights \(a_L\) are in \((0,1)\) and in particular are non-zero.

    Since the collection of residues \(\{A_1, \ldots, A_m\}\) satisfies \cite[Assumptions 3.25]{pkc},
	by \cite[Proposition 3.30]{pkc} we have a direct sum decomposition \(\C^n = T \oplus N_1 \oplus \ldots \oplus N_k\)
	such that
	\[
	A_x = 0 \cdot \Id_T \,\oplus\, \lambda_1 \cdot \Id_{N_1} \,\oplus\, \ldots \,\oplus\, \lambda_k \cdot \Id_{N_k} \,,
	\]
	where \(N_j\) is subspace complimentary to \(L_j\)
	and \(\lambda_j = a_{L_j}\) is the weight at \(L_j\). The klt assumption implies \(\lambda_j \in (0,1)\), so there is no pair of eigenvalues of \(A_p\) that differ by a positive integer, and the connection \(\nabla\) is non-resonant as in Definition \ref{def:resonant}. 
	
	\textbf{Step 3: reduction to essential irreducible case.} From Proposition \ref{prop:locstnd} it follows  that the connections \(\nabla\) and \(\nabla_0\) are holomorphically gauge equivalent near \(0\). More explicitly, there is a frame of holomorphic vector fields \(\{v_1, \ldots, v_n\}\) near \(0\) giving a trivialization of \(T\C^n\) such that \(\nabla = d - \Omega\) with \(\Omega = \sum A_i d\ell_i / \ell_i\). We can then define (as in \cite[Section 4.2]{pkc}) integrable distributions \(\mT\) and \(\mN_1, \ldots, \mN_k\) spanned by vector fields of the frame \(\{v_1, \ldots, v_n\}\) with \(T\C^n = \mT \oplus \mN_1 \oplus \ldots \oplus \mN_k\) restricting to \(T \oplus N_1 \oplus \ldots \oplus N_k\) along \(T\). The distribution \(\mT\) is spanned by the parallel vector fields of the frame \(\{v_1, \ldots, v_n\}\), while \(\nabla \mN_j \subset \mN_j\). The connection \(\nabla\) splits near the origin (as in \cite[Proposition 4.49]{pkc}) as a product \((\mT|_0, d) \times (\mN_1|_0, \nabla|_{\mN_1}) \times \ldots (\mN_k|_0, \nabla|_{\mN_k})\) where \(\mT|_0\) and \(\mN_j|_0\) are the leaves through the origin. Fixing one leaf \(\mN_j|_0\), we reduce to the case that \(\{D_1, \ldots, D_m\}\) is an essential irreducible arrangement of hyperplanes in \(\C^n\) going through the origin and \(A_x = \sum A_i\) is a scalar operator with \(A_x = \lambda \cdot \Id\) for some \(\lambda \in (0,1)\).
	
	\textbf{Step 4: local dilation vector field.} We begin by constructing a holomorphic vector field \(e_x\) on the arrangement complement \(B \setminus D\) by the following procedure.
	Fix an arbitrary base point \(y \in B^{\circ} = B \setminus D\).
	Let \(C = \C \cdot y\) be the complex line spanned by \(y\). Then \(\nabla|_C = d - A_x dt / t + \hol\), where \(A_x = \sum A_i\) and \(t\) is a coordinate on \(C\). The holonomy of \(\nabla\) along the loop \(\gamma(s) = e^{2\pi i s} q\) for \(s \in [0,1]\) equals \(\Hol_{\gamma}(\nabla) = \exp(2\pi i A_x)\). Since \(A_x = \lambda \cdot I\), we see that \(\Hol_{\gamma}(\nabla)\) is equal to scalar multiplication by the unit complex number \(c = \exp(2\pi i \lambda) \neq 1\).
	Let \(\widetilde{B^{\circ}}\) be the universal cover of \(B^{\circ}\) and consider the developing map \(\dev: \widetilde{B^{\circ}} \to \C^n\) of the affine structure on \(B^{\circ}\) defined by \(\nabla\). The developing map is equivariant with respect to the actions of \(\pi_1(B^{\circ}, y)\) on \(\widetilde{B^{\circ}}\) by deck transformations and on \(\C^n\) by the holonomy representation \(\widetilde{\Hol}:  \pi_1(B^{\circ}, y) \to \GL_n(\C) \ltimes \C^n\) (the tilde here is used to distinguish it from the holonomy of the connection, which is equal to the linear projection of \(\widetilde{\Hol}\)). Since \([\gamma]\) is in the centre of \(\pi_1(B^{\circ}, y)\) and \(\widetilde{\Hol}([\gamma])\) fixes a unique point \(p \in \C^n\), we conclude that any holonomy transformation \(\widetilde{\Hol}([\delta])\) for \([\delta] \in \pi_1(B^{\circ}, y)\) must also fix \(p\) and thus commute with dilations centred at \(p\). It follows from this that the pullback by the developing map of the Euler vector field on \(\C^n\) generated by dilations centred at \(p\) is invariant under deck transformations and thus it descends as a holomorphic vector field \(e_x\) on \(B^{\circ}\).
	
	\emph{Claim}: \(e_x\) extends holomorphically over \(B\) as a vector field tangent to all the hyperplanes \(D_1,\ldots, D_m\). In particular, \(e_x(0)=0\).
	
    \emph{Proof of the claim}: consider first the case \(n=1\). Then we can write \(\nabla = d - \lambda\frac{f(t)}{t} dt\) with \(f\) holomorphic and \(f(0) = 1\). It is standard to show that we can change coordinates \(z=z(t)\) so that the developing map is \(z \mapsto z^{1-\lambda}\) (see \cite[Section 2.2.1]{novikovtahar}) and \(e_x = (1-\lambda)^{-1} z \cdot \p /\p z\).
	Consider now the case \(n \geq 2\). It follows from the \(n=1\) case together with Step 3 that \(e_x\) extends holomorphically over \(D^{\circ}\) and \(e_x\) is tangent to \(D\) at points of \(D^{\circ}\). By Hartogs, \(e_x\) extends holomorphically over the ball \(B\). Since \(e_x\) is tangent to every hyperplane and the arrangement is essential, we must have \(e_x(0) = 0\). This finishes the proof of the claim.
	
	\textbf{Step 5: linearisation.} 
    Let \(de_x\) be the linear part of the vector field \(e_x\) at \(0\). 
    
    \emph{Claim 1}: \(de_x\) is a scalar operator, i.e., \(de_x = \mu \cdot \Id\) for some \(\mu \in \C\). 
    
    \emph{Proof:} since \(e_x\) is tangent to all of the hyperplanes \(D_1, \ldots, D_m\), we must have that the linear endormophism \(de_x\) of \(T_0 \C^n \cong \C^n\) must preserve all the hyperplanes of the arrangement. Since the arrangement \(\{D_1, \ldots, D_m\}\) is essential and irreducible, by \cite[Lemma 2.33]{pkc}, we must have \(de_x = \mu \cdot \Id\) for some \(\mu \in \C\), as wanted.

    \emph{Claim 2}: \(\mu \neq 0\). 
    
    \emph{Proof}: by contradiction, suppose \(\mu=0\). Consider a holomorphic probe curve \(c(t)\) for \(t\) in the unit disc \(\{|t|< 1\} \subset \C\) with \(c(0)=0\), \(c'(0) = v \neq 0\) and such that \(c(t) \in B^{\circ}\) for \(t \neq 0\). Then the pullback connection \(\tn = c^* \nabla\) is of the form \(\tn = d - A_x dt / t + \hol\). If \(\mu = 0\) then \(c^*e_x = O(t^2)\) and so 
    \begin{equation}\label{eq:contr1}
       \nabla_v e_x = \tn_{\p/\p t} (c^*e_x) (0) = 0 \,. 
    \end{equation}
    On the other hand, if we extend \(v\) as a constant vector field on  \(B\) then 
    \begin{equation}\label{eq:contr2}
        \nabla_v e_x = v
    \end{equation}
    on \(B^{\circ}\) by construction of \(e_x\). Indeed, \(e_x\) is the pullback of the Euler vector field on \(\C^n\) which satisfies Identity (\ref{eq:contr2}) for any vector field on \(\C^n\).
    Moreover, since \(e_x\) is tangent to all of the hyperplanes, the vector field \(\nabla_v e_x\) is holomorphic. Equations \eqref{eq:contr1} and \eqref{eq:contr2} give us \(v=0\) contradicting that \(v \neq 0\). This finishes the proof of the claim.
    
    By Claims 1 and 2 the linear part of \(e_x\) at \(0\) is non-degenerate and equal to \(\mu \cdot \Id\) for some \(\mu \neq 0\). By Poincar\'e-Dulac, we can find complex coordinates \((z_1, \ldots, z_n)\) such that 
	\[
	e_x = \mu \cdot \sum_{i=1}^n z_i \frac{\p}{\p z_i} \,.
	\]
	In such coordinates, the connection \(\nabla\) is standard in the trivialization given by coordinate vector fields (see \cite[Lemma 4.79]{pkc}) as wanted.\footnote{The equation \(\nabla e_x = \Id\) implies a posteriori that \(\mu = (1-\lambda)^{-1}\).} 
\end{proof}

\subsubsection{Proof of Theorem \ref{thm:connmet}}
With Proposition \ref{prop:linearcoord} in hand, the proof of Theorem \ref{thm:connmet} follows as in \cite[Section 5]{pkc}. For completeness, we include the details.

\begin{proof}[Proof of Theorem \ref{thm:connmet}]
	We show that the Hermitian metic \(g\) preserved by \(\nabla\) satisfies the three items in Definition \ref{def:pkmetric}.
	
	(i) Let \(x \in X \setminus D\). Since \(\nabla\) is flat we can find a frame of parallel holomorphic vector fields \(v_1, \ldots, v_n\) around \(x\). After a linear transformation, we might assume that the frame is unitary with respect to \(g\). Since \(\nabla\) is torsion-free, the vector fields \(v_i\) pairwise commute. By Frobenius, we can find holomorphic coordinates \(z_1, \ldots, z_n\) centred at \(x\) such that \(v_i = \p / \p z_i\). In these coordinates, \(g\) is equal to the pullback of the Euclidean metric on \(\C^n\), and therefore \(g\) is flat and K\"ahler.
	
	(ii) Let \(x \in D_i \cap D^{\circ}\). To simplify notation, write \(a = a_i\) and let \(\alpha = 1-a\). By Proposition \ref{prop:linearcoord}, we can choose holomorphic coordinates \(z_1, \ldots, z_n\) centred at \(x\) such that \(\nabla \p_{z_1} = -(a/z_1) dz_1 \otimes \p_{z_1}\) and \(\nabla \p_{z_j} = 0\) for \(j \geq 2\), with \(D = \{z_1=0\}\). The holonomy of \(\nabla\) along a simple positive loop around \(\{z_1=0\}\) acts on \(\p_{z_1}\) by scalar multiplication by \(e^{2\pi i a} \neq 1\) and fixes \(\p_{z_j}\) for \(j \geq 2\), which implies \(g(\p_{z_1}, \p_{z_j}) = 0\). Making a linear change in the coordinates \(z_2, \ldots, z_n\) we can assume that the vectors \(\p_{z_j}\) for \(j \geq 2\) are orthonormal. 	
	On the other hand, since the locally defined vector field \(z_1^a \p_{z_1}\) is parallel, we must have \(|\p_{z_1}|_g^2 = C |z_1|^{-2a}\) for some \(C>0\). Changing \(z_1\) by a scalar multiple, we can assume \(C=1\). Therefore, in the coordinates \(z_1, \ldots, z_n\) we have \(g = |z_1|^{2\alpha-2} |dz_1|^2 + \sum_{j \geq 2} |dz_j|^2\) as we wanted to show. 
	
	(iii) 
	Let \(d_g : X^{\circ} \times X^{\circ} \to \R_{\geq 0}\) be the distance function induced by the Riemannian metric \(g\) on \(X^{\circ}:= X \setminus D\). We want to show the following: (1) \(d_g\) extends as a continuous function \(d_g: X \times X \to \R_{\geq 0}\); (2) \((X, d_g)\) is a metric space and the topology induced by \(d_g\) on \(X\) is the same as the original topology of the complex manifold \(X\); and (3) \((X, d_g)\) is polyhedral. 
	
	Proof of (1): By \cite[Lemma 5.36]{pkc} it suffices to show that for every point \(x \in D\) and every \(\epsilon>0\) there is an open neighbourhood \(x \in U \subset X\) such that \(\diam(U \cap X^{\circ }, d_g) < \epsilon\). To show this, we use that the existence of linear coordinates (given by Proposition \ref{prop:linearcoord}) implies that every point \(x\) has a neighbourhood \(U\) such that the restriction of \(g\) to \(U \cap X^{\circ}\) is isometric to an (incomplete) Riemannian cone (see \cite[Lemma 5.15]{pkc}) with \(x\) at the vertex. The argument of \cite[Lemma 5.39 item (2)]{pkc} then applies verbatim.
    
	Proof of (2): By \cite[Lemma 5.37]{pkc} it suffices to show that for any \(x, y \in D\) with \(x \neq y\) we have \(d_g(x, y) > 0\), which follows by the same argument as in \cite[Lemma 4.39 item (4)]{pkc}.
    
	Proof of (3): Since \(X\) is compact, by \cite{petruninlebedeva} it suffices to show that every
	\(x \in X\) has an open neighbourhood which admits an isometric embedding into a Euclidean cone, sending \(x\) to the vertex of the cone. This follows by \cite[Theorem 1.1]{pkc}, which can be applied thanks to the existence of linear coordinates (Proposition \ref{prop:linearcoord}) together with the klt assumption, which implies that the non-integer and positivity conditions of \cite[Theorem 1.1]{pkc} are satisfied.  
\end{proof}

\section{Cohomology of wonderful compactifications}\label{sec:cohomology}
In Section \ref{sec:conciniprocesi} we review the minimal De Concini-Procesi model \(\pi: X \to \CP^n\) of a hyperplane arrangement \(\mH \subset \CP^n\) introduced in \cite{conciniprocesi}. In Section \ref{sec:basic} we recall the \(\Z\)-basis of the free abelian group \(H^4(X, \Z)\) in terms of \emph{basic monomials} from \cite{yuzvinsky}. In Section \ref{sec:nonbasic} we write the non-basic monomials in \(H^4(X, \Z)\) as explicit linear combinations of basic monomials. In Section \ref{sec:cordproj} we summarize our calculations in Corollary \ref{cor:proj}, for ease of reference in Sections \ref{sec:1implies2} and \ref{sec:2implies1}. 

\subsection{Background on hyperplane arrangements}\label{sec:conciniprocesi}

Let \(\mH\) be a hyperplane arrangement in \(\CP^n\), i.e., \(\mH\) is a finite collection of complex hyperplanes \(H \subset \CP^n\). The arrangement \(\mH\) is \emph{essential} if
the common intersection of all its members is empty. The arrangement \(\mH\) is \emph{irreducible} if there is no pair of disjoint linear subspaces \(P, Q \subset (\CP^n)^*\) such that \(\mH^* \subset P \cup Q\) and both \(P \cap \mH^*\), \(Q \cap \mH^*\) are non-empty, where \(\mH^*\) is the configuration of points in \((\CP^n)^*\) dual to \(\mH\). Throughout this section we assume that \(\mH\) is essential and irreducible.

\begin{definition}\label{def:intposet}
	Let \(\mL = \mL(\mH)\) be the set of all subspaces \(L \subset \CP^n\) obtained by intersecting members of \(\mH\). We allow \(\emptyset\) as well as \(\CP^n\) to be part of \(\mL\).
	The \emph{intersection poset} is of the set \(\mL\) endowed with the partial order given by reverse inclusion.
\end{definition}

The intersection poset \(\mL\) is a \emph{geometric lattice}, see \cite[Proposition 3.8]{stanley}. The unique minimal element is \(\hat{0} = \CP^n\) and the unique maximal element is \(\hat{1} = \emptyset\).
The \emph{rank} of \(L \in \mL\) is given by
\begin{equation}
r(L) = \codim_{\CP^n}L \,, 	
\end{equation}
where \(\codim_{\CP^n}L = n - \dim L\) is the codimension of \(L \subset \CP^n\). We agree that \(r(\emptyset) = n+1\) and \(r(\CP^n) = 0\).
Write
\begin{equation}
	\mL^i = \{L \in \mL \,|\, r(L) = i\} \,.
\end{equation}

\begin{definition}\label{def:G}
	We call an element \(L \in \mL\) an \emph{irreducible subspace} if the localized arrangement \(\mH_L\) is irreducible.
    
	Let \(\mG = \mG(\mH)\) be the subset of \(\mL \setminus \{\hat{0}\}\) consisting of proper subspaces \(L \subsetneq \CP^n\) that are irreducible. Write \(\mG^i = \{L \in \mG \,|\, r(L) = i\}\).
	In particular, \(\mG^1 = \mH\)\,. 
    
    We denote by \(\mLi = \mG \setminus \{\emptyset\}\) the set of non-empty and proper irreducible subspaces and let \(\mLi^{\circ} = \mLi \setminus \mH\).
\end{definition}

\begin{remark}
The assumption that \(\mH\) is essential and irreducible implies that \(\emptyset \in \mG\). Indeed \(\mH\) essential implies  \(\emptyset \in \mL\) and clearly \(\mH_\emptyset=\mH\), which is irreducible.
\end{remark}

\noindent \textbf{The minimal De Concini-Procesi model.}
Let
\begin{equation}\label{eq:conciniprocesi}
	\pi: X \to \CP^n
\end{equation}
be the \emph{minimal De Concini-Procesi model} of \(\mH\) obtained by successive blowup of the elements in \(\mLi^{\circ}\) in increasing order of dimension, see \cite{conciniprocesi} and also \cite[\S 3.1]{miyaokayau}.
The preimage of the arrangement 
\[
D = \pi^{-1}(\mH)
\]
is a simple normal crossing divisor.

The irreducible components of \(D\) are in bijection with the elements of \(\mG' := \mG \setminus \{\emptyset\}\) (i.e., \(\mG' = \mLi\)).
For each \(L \in \mG'\) there exists a unique irreducible divisor \(D_L \subset X\) with \(\pi(D_L) = L\).
The divisors \(D_L\) are the irreducible components of
\begin{equation}\label{eq:DL}
D = \sum_{L \in \mG'} D_L \,.	
\end{equation}

\begin{definition}\label{def:nested}
A set \(\mS \subset \mG\) is \emph{nested} if, given any collection of pairwise non-comparable elements \(L_1, \ldots, L_k \in \mS\) (i.e., \(L_i \not\subset L_j\) for \(i \neq j\)) with \(k \geq 2\)\,,  their common intersection is reducible:
\[
\bigcap_{i=1}^k L_i \notin \mG \,.
\]
Let \(\Nes\) be the collection of all nested sets \(\mS\).	
\end{definition}

\begin{definition}\label{def:redint}
Let \(L_1, L_2 \in \mG\). If \(L_1 \cap L_2 \notin \mG\) then
we say that \(L_1\) and \(L_2\) have \emph{reducible intersection} and we write
\begin{equation}
L_1 \pitchfork_r L_2 \,.	
\end{equation} 
\end{definition}

\begin{lemma}\label{lem:redint}
	If \(L_1 \pitchfork_r L_2\) then the intersection \(L_1 \cap L_2\) is non-empty and 
	\[
	L_1 + L_2 = \CP^n \,.
	\]	
\end{lemma}

\begin{proof}
	The intersection \(L = L_1 \cap L_2\) is non-empty because \(\emptyset \in \mG\) (in this section all arrangements are essential and irreducible). The fact that the sum is \(\CP^n\) follows from the fact that \(L_1\) and \(L_2\) are the \emph{irreducible components} of \(L\), see \cite[Lemma 2.21]{miyaokayau}.
\end{proof}

As an immediate consequence of Lemma \ref{lem:redint} we have the following.

\begin{corollary}\label{cor:nestedpair}
	Suppose that \(\mS = \{L_1, L_2\}\) is nested and that \(L_1 + L_2 \subsetneq \CP^n\). Then either \(L_1 \subset L_2\) or \(L_2 \subset L_1\).
\end{corollary}

\subsection{The cohomology groups \(H^{2k}(X, \Z)\)}\label{sec:basic}

For the rest of this section, we denote by \(X\) the minimal De Concini Procesi model of the (essential, irreducible) hyperplane arrangement \(\mH \subset \CP^n\).
We recall a well known set of generators and relations for the cohomology ring \(H^*(X, \Z)\).

\begin{definition}\label{def:gammaL}
	For \(L \in \mG\) we let \(\gamma_L \in H^2(X, \Z)\) be given by
	\begin{equation}
	\gamma_{L} = \begin{cases}
	-\pi^*c_1\big(\mO_{\P^n}(1)\big) &\text{ if } L = \emptyset \,, \\
	\phantom{-} c_1(D_L)  &\text{ if } L \in \mG' \,.
	\end{cases}
	\end{equation}
\end{definition}

\begin{theorem}[\cite{conciniprocesi2}]
	The cohomology ring \(H^*(X, \Z)\) is the algebra generated by the classes \(\gamma_L \in H^2(X, \Z)\) for \(L \in \mG\) with relations:
	\begin{equation}\label{eq:R1}\tag{R1}
	\forall \mS \notin \Nes: \,\,
	\prod_{L \in \mS} \gamma_L = 0 \,, 
	\end{equation}
	
	\begin{equation}\label{eq:R2}\tag{R2}
	\forall H \in \mH: \,\,
	\sum_{L \subset H} \gamma_L = 0 \,.
	\end{equation}
\end{theorem}

\begin{remark}
	Equation \eqref{eq:R1} holds in the stronger sense that if \(\mS \subset \mG\) is not nested then the common intersection of the divisors \(D_L\) with \(L \in \mS\) is empty, see item (2) in Theorem \cite[p.476]{conciniprocesi}.
\end{remark}

\begin{remark} Recall that in Equation \eqref{eq:R2} we allow \(L\) to be empty. 
	Using \(\gamma_{\emptyset} = -\pi^*c_1\big(\mO_{\P^n}(1)\big)\)
	in Equation \eqref{eq:R2} we obtain
	\begin{equation*}
	\forall H \in \mH: \,\,
	\gamma_H = \pi^*c_1\big(\mO_{\P^n}(1)\big) -  \sum_{ \emptyset \subsetneq L \subsetneq H} \gamma_L  \,.
	\end{equation*}
\end{remark}

\subsubsection{Yuzvinsky's theorem}
The cohomology groups \(H^{2k}(X, \Z)\) are free abelian groups. We recall a well known set of \(\Z\)-bases of these groups found by Yuzvinsky \cite{yuzvinsky}.

\begin{definition}
	Let \(\mS \subset \mG\) be a nested set and let \(m: \mS \to \Z_{> 0}\) be such that
	\[
	\sum_{L \in \mS} m(L) = k \,.
	\]
	Let \(\gamma(\mS, m)\) be the product
	\begin{equation}\label{eq:gammamonomial}
	\gamma(\mS, m) = 
	\prod_{L \in \mS} \gamma_L^{m(L)} \,.	
	\end{equation} 
	We say that \(\gamma(\mS, m) \in H^{2k}(X, \Z)\) is a monomial of degree \(k\). 
\end{definition}

\begin{definition}
	Let \(\gamma(\mS, m)\) be as above. We say that \(\gamma(\mS, m)\) is \emph{basic} if
	\begin{equation}\label{eq:basic}
	\forall L \in \mS: \,\,	m(L) < r(L) - r(M) \,, \,\, \textup{ where } \,\, M = \bigcap_{L' \in \mS \,|\, L' \supsetneq L }  L'  \,.
	\end{equation}
\end{definition}

\begin{remark}
	If \(\{L' \in \mS \,|\, L' \supsetneq L\} = \emptyset\) then \(M = \CP^n\) and \(r(M) = 0\).
\end{remark}

Let \(\Delta_k\) be the set of basic monomials of degree \(k\).
\begin{theorem}[{\cite[Corollary 3.9]{yuzvinsky}}]\label{thm:basis}
	The set \(\Delta_k\) is a \(\Z\)-basis of \(H^{2k}(X, \Z)\).
\end{theorem}

\begin{example}
	If \(k=1\) then \(\mS = \{L\}\) for some \(L \in \mG\) and
	\(\gamma(\mS, m) = \gamma_L\). Equation \eqref{eq:basic} requires that \(1 < r(L)\). Thus \(\Delta_1 = \{\gamma_L \,|\, r(L) \geq 2\}\).
\end{example}

\subsubsection{The cohomology group \(H^4(X, \Z)\)}
Next, we write down the elements of \(\Delta_2\). To do this, we introduce some handy notation.

\begin{definition}\label{def:inclusion}
	Let \(L\) and \(M\) be linear subspaces of \(\CP^n\).
	If \(L \subset M\) and \(r(L) = r(M) + 1\) then we write
	\begin{equation}
		L \inclusiondot M  \,.
	\end{equation}
	If \(L \subset M\) and \(r(L) \geq r(M) + 2\) then we write
	\begin{equation}
		L \Subset M \,.
	\end{equation}
\end{definition}

\begin{lemma}\label{lem:basisH4}
	The basis \(\Delta_2\) of \(H^4(X, \Z)\) consists of the following monomials:
	\begin{enumerate}[label=\textup{(\arabic*)}]
		\item \(\gamma_L^2\) with \(r(L) \geq 3\);
		\item \(\gamma_{L_1} \cdot \gamma_{L_2}\) with \(L_1 \Subset L_2\) and \(r(L_2) \geq 2\);
		\item \(\gamma_{L_1} \cdot \gamma_{L_2}\) with \(L_1 \pitchfork_r L_2\) and \(r(L_1), r(L_2) \geq 2\).
	\end{enumerate}
\end{lemma}

\begin{proof}
	Suppose that \(\gamma(\mS, m) \in \Delta_2\). We distinguish two cases.
	\begin{itemize}
		\item \textbf{Case A:}  \(\mS = \{L\}\) for some \(L \in \mG\) . Then
		\(\gamma(\mS, m) = \gamma_L^2\) and Equation \eqref{eq:basic} implies \(2 < r(L)\). Therefore, \(\gamma(\mS, m)\) is as in item (1).
		
		\item \textbf{Case B:}  \(\mS = \{L_1, L_2\}\) for a pair of distinct \(L_1, L_2 \in \mG\).
		Then \(\gamma(\mS, m) = \gamma_{L_1} \cdot \gamma_{L_2}\). We distinguish two subcases.
		\begin{itemize}
			\item \textbf{Case B1:} \(L_1\) and \(L_2\) are comparable. Up to relabelling we can assume that \(L_1 \subsetneq L_2\). Then Equation \eqref{eq:basic} implies
			\[
			1 < r(L_2) \,\, \text{ and } \,\, 1 < r(L_1) - r(L_2) \,. 
			\]
			Therefore, \(\gamma(\mS, m)\) is as in item (2).
			\item \textbf{Case B2:} \(L_1\) and \(L_2\) are not comparable. Since \(\mS\) is nested, the intersection \(L_1 \cap L_2\) is reducible, i.e., \(L_1 \pitchfork_r L_2\).
			Equation \eqref{eq:basic} implies
			\[
			1 < r(L_1) \,\, \text{ and } \,\, 1 < r(L_2) \,. 
			\]
			Therefore, \(\gamma(\mS, m)\) is as in item (3). 
		\end{itemize}
	\end{itemize} 
	Conversely, the above proof clearly shows that the monomials in (1), (2), (3) are basic.
\end{proof}

\subsection{Non-basic monomials}\label{sec:nonbasic}

Suppose that \(\gamma(\mS, m)\) is a degree \(2\) monomial in \(H^4(X, \Z)\) that is \emph{not} basic. Then \(\gamma(\mS, m)\) has to be of the following type:  

\begin{enumerate}[label={(M\arabic*)}]
	\item \label{it:mH} \(\gamma_{H}^2\) with \(H \in \mH\);
	\item \label{it:mL} \(\gamma_{L}^2\) with \(r(L) = 2\);
	\item \label{it:mLM} \(\gamma_{L} \cdot \gamma_M\) with \(L \inclusiondot M\) and \(r(M) \geq 2\);
	\item \label{it:mLH3} \(\gamma_L \cdot \gamma_{H}\) with \(L \subset H\), \(H \in \mH\) and \(r(L) \geq 3\);
	\item \label{it:mLH2} \(\gamma_L \cdot \gamma_{H}\) with \(L \subset H\), \(H \in \mH\) and \(r(L) = 2\);
    \item \label{it:LredH} \(\gamma_L \cdot \gamma_{H}\) with \(L \pitchfork_r H\), \(H \in \mH\) and \(r(L) \geq 2\);
    \item \label{it:HredH} \(\gamma_H \cdot \gamma_{H'}\) with \(H \pitchfork_r H'\) and \(H, H' \in \mH\).
\end{enumerate}

We will write the above monomials as a linear combination of basic monomials. The key to do this is given by the next. 

\begin{proposition}
	In \(H^4(X, \Z)\) the following relations hold.
	\begin{enumerate}
		\item If \(L \inclusiondot M\) then
		\begin{equation}\label{eq:incdot}
		\gamma_{M} \cdot \left( \sum_{L' \subset L} \gamma_{L'} \right) = 0 \,.
		\end{equation}
		\item If \(r(L) = 2\) then
		\begin{equation}\label{eq:cod2square}
		\left( \sum_{L' \subset L} \gamma_{L'} \right)^2 = 0 \,.
		\end{equation}
	\end{enumerate}
\end{proposition}

\begin{proof}
	Equations \eqref{eq:incdot} and \eqref{eq:cod2square} follow from relations \eqref{eq:R1} and \eqref{eq:R2}, see the proof of Theorem 1 in \cite{feichtner}. Equations \eqref{eq:incdot} and \eqref{eq:cod2square} are particular cases of the relations in \(H^*(X, \Z)\) given by Equation (5.1.1) in \cite{conciniprocesi}.
\end{proof}

We will use the following notation.

\begin{definition}\label{def:B}
	For \(L_1, L_2 \in \mG\) with \(L_1 \subset L_2\) we let
	\begin{equation}
	B(L_1, L_2) = -1 + \big| \{L' \in \mG \,|\, L_1 \subset L' \inclusiondot L_2 \} \big| \,.
	\end{equation}
	If \(L_1 = \emptyset\) and  \(L_2 = H \in \mH\) then
	\(B_H:= B(\emptyset, H) = | \{L \in \mG^2 \,|\, L \subset H\} | -1\) is the number of irreducible codimension \(2\) subspaces contained in \(H\) minus \(1\), same as in \cite[Definition 3.36]{miyaokayau}.
\end{definition}

\begin{lemma}[type \ref{it:mLM}]\label{lem:mLM}
	If \(L \inclusiondot M\) with \(L, M \in \mG\) then
	\begin{equation}\label{eq:mLM}
	\gamma_{L} \cdot \gamma_{M} = - \sum_{L' \subsetneq L} \gamma_{L'} \cdot \gamma_{M}   \,.
	\end{equation}
\end{lemma}

\begin{proof}
	Write
	\[
	\sum_{L' \subset L} \gamma_{L'} = \gamma_L + \sum_{L' \subsetneq L} \gamma_{L'} 
	\]
    and use Equation \eqref{eq:incdot}.
\end{proof}

\begin{lemma}
	If \(M \in \mG\) then
	\begin{equation}\label{eq:sumlneq}
	\sum_{L \subsetneq M} \gamma_{L} \cdot \gamma_{M} = - \sum_{L \Subset M} B(L, M) \cdot \gamma_L \cdot \gamma_M \,.
	\end{equation}
\end{lemma}

\begin{proof}
	Splitting the left hand side of \eqref{eq:sumlneq} as a sum over \(L \Subset M\) and  \(L \inclusiondot M\)\,, 
	\[
	\sum_{L \subsetneq M} \gamma_{L} \cdot \gamma_{M} =
	\sum_{L \Subset M} \gamma_{L} \cdot \gamma_{M} +
	\sum_{L \inclusiondot M} \gamma_{L} \cdot \gamma_{M} \,.
	\]
	Equation \eqref{eq:mLM} implies
	\[
	\sum_{L \inclusiondot M} \gamma_{L} \cdot \gamma_{M} = 
	- \sum_{L \Subset M} \big| \{L' \,|\, L \subset L' \inclusiondot M \} \big| \gamma_{L} \cdot \gamma_{M} \,.
	\]
	Therefore
	\[
	\begin{aligned}
	\sum_{L \subsetneq M} \gamma_{L} \cdot \gamma_{M} &=
	\sum_{L \Subset M} \gamma_{L} \cdot \gamma_{M} +
	\sum_{L \inclusiondot M} \gamma_{L} \cdot \gamma_{M} \\
	&= \sum_{L \Subset M} \left(1 - \big| \{L' \,|\, L \subset L' \inclusiondot M \} \big| \right) \gamma_{L} \cdot \gamma_{M} \\
	&= - \sum_{L \,|\, L \Subset M} B(L, M) \cdot \gamma_{L} \cdot \gamma_{M} 
	\end{aligned}
	\]
	where the last equality follows from the definition of \(B(L, M)\).
\end{proof}

\begin{lemma}\label{lem:sumsuared}
	If \(L \in \mG\) then
	\begin{equation}\label{eq:sumsquare}
	\left( \sum_{L' \subset L} \gamma_{L'} \right)^2 = 
	\sum_{L' \subset L} \gamma_{L'}^2  \,+\, 2 \cdot \sum_{L'' \subsetneq L' \subset L}  \gamma_{L''} \cdot \gamma_{L'} \,.
	\end{equation}
\end{lemma}

\begin{proof}
	Expanding the square
	\[
	\left( \sum_{L' \subset L} \gamma_{L'} \right)^2 = 
	\sum_{L' \subset L} \gamma_{L'}^2  + \sum_{(L', L'')}  \gamma_{L''} \cdot \gamma_{L'} 
	\]
	where the last sum is over all pairs \((L', L'') \in \mG \times \mG\) with \(L' \neq L''\) such that \(L' \subset L\) and \(L'' \subset L\). By relation \eqref{eq:R1} we can assume that the sum is over pairs \((L', L'')\) for which the sets \(\mS = \{L', L''\}\) are nested, otherwise \(\gamma_{L''} \cdot \gamma_{L'} = 0\). On the other hand, since 
	\[
	L' + L'' \subset L \subsetneq \CP^n \,,
	\]
	Corollary \ref{cor:nestedpair} implies that either \(L' \subsetneq L''\) or \(L'' \subsetneq L'\) and \eqref{eq:sumsquare} follows.
\end{proof}

\begin{lemma}[type \ref{it:mL}]\label{lem:mL}
	If \(r(L) = 2\) then
	\begin{equation}\label{eq:mL}
	\gamma_L^2 = 
	- \sum_{L' \subsetneq L} \gamma_{L'}^2  \,+\, 2 \cdot \sum_{L'' \Subset L' \subset L}  B(L'', L') \cdot \gamma_{L''} \cdot \gamma_{L'}   \,.
	\end{equation}	
\end{lemma}

\begin{proof}
	Equations \eqref{eq:cod2square} and \eqref{eq:sumsquare}  give us
	\[
	0 = \left( \sum_{L' \subset L} \gamma_{L'} \right)^2 = 
	\sum_{L' \subset L} \gamma_{L'}^2  \,+\,  2 \cdot \sum_{L'' \subsetneq L' \subset L}   \gamma_{L''} \cdot \gamma_{L'}  \,.
	\]
	On the other hand, by Equation \eqref{eq:sumlneq}, for each \(L' \subset L\) we have 
	\[
	\sum_{L'' \subsetneq L'} \gamma_{L''} \cdot \gamma_{L'} =
	-\sum_{L'' \Subset L'} B(L'', L') \cdot \gamma_{L''} \cdot \gamma_{L'} \,.
	\]
	Altogether, we obtain
	\[
	0  = 
	\sum_{L' \subset L} \gamma_{L'}^2  \,-\,  2 \cdot \sum_{L'' \Subset L' \subset L} B(L'', L') \cdot   \gamma_{L''} \cdot \gamma_{L'}  
	\]
	from which Equation \eqref{eq:mL} follows. The monomials in the right hand side of Equation \eqref{eq:mL} are clearly basic.
\end{proof}

\begin{lemma}[type \ref{it:mH}]\label{lem:mH}
	If \(H \in \mH\) then
	\begin{equation}\label{eq:mH}
	\gamma_H^2 = - \sum_{L \Subset H} B(L, H) \cdot \gamma_L^2 \,+\,  2 \cdot \sum_{L'' \Subset L' \subsetneq H}  B(L'', L') \cdot B(L', H) \cdot  \gamma_{L''} \cdot \gamma_{L'} \,.
	\end{equation}
\end{lemma}

\begin{proof} 
	Using relation \eqref{eq:R2} we have \(\gamma_H = - \sum_{L \subsetneq H} \gamma_L\).
	Taking the square and expanding as in Lemma \ref{lem:sumsuared} we have
    \begin{equation}\label{eq:pfH1}
    \gamma_H^2 =  
	\sum_{L \subsetneq H} \gamma_L^2 \,+\,  2 \cdot \sum_{L'' \subsetneq L' \subsetneq H}   \gamma_{L''} \cdot \gamma_{L'} \,.  
    \end{equation}
	Using Equation \eqref{eq:sumlneq} we get
    \begin{equation}\label{eq:pfH2}
    \sum_{L'' \subsetneq L' \subsetneq H}   \gamma_{L''} \cdot \gamma_{L'} =
	- \sum_{L'' \Subset L' \subsetneq H} B(L'', L') \cdot \gamma_{L''} \cdot \gamma_{L'} \,.    
    \end{equation}
	On the other hand,
    \begin{equation}\label{eq:pfH3}
     \sum_{L \subsetneq H} \gamma_L^2 = \sum_{L \Subset H} \gamma_L^2 \,+\, \sum_{L \inclusiondot H} \gamma_L^2 \,.   
    \end{equation}
	Using Equation \eqref{eq:mL} we have
	\[
	\sum_{L \inclusiondot H} \gamma_L^2 = \sum_{L \inclusiondot H} \left(
	- \sum_{L' \subsetneq L} \gamma_{L'}^2  \,+\, 2 \cdot \sum_{L'' \Subset L' \subset L}  B(L'', L') \cdot \gamma_{L''} \cdot \gamma_{L'} \right) \,.
	\]
    Clearly, 
    \[
    \sum_{L \inclusiondot H}
	\sum_{L' \subsetneq L} \gamma_{L'}^2 = \sum_{L' \Subset H} \big| \{L \,|\, L' \subset L \inclusiondot H\} \big| \cdot \gamma_{L'}^2 \,, 
    \] 
    and similarly 
    \[\sum_{L \inclusiondot H} \sum_{L'' \Subset L' \subset L}  B(L'', L') \cdot \gamma_{L''} \cdot \gamma_{L'} = \sum_{L'' \Subset L' \subsetneq H} \big| \{L \,|\, L' \subset L \inclusiondot H\} \big| \cdot
	B(L'', L') \cdot \gamma_{L''} \cdot \gamma_{L'} \,.
    \]
    Thus, we deduce that
    \begin{equation}\label{eq:pfH4}
    \begin{aligned}
         \sum_{L \inclusiondot H} \gamma_L^2 &=  
       - \sum_{L' \Subset H} \big| \{L \,|\, L' \subset L \inclusiondot H\} \big| \cdot \gamma_{L'}^2 \\
       &+\, 2 \cdot \sum_{L'' \Subset L' \subsetneq H} \big| \{L \,|\, L' \subset L \inclusiondot H\} \big| \cdot B(L'', L') \cdot \gamma_{L''} \cdot \gamma_{L'} \,.
    \end{aligned}
    \end{equation}
	Equations \eqref{eq:pfH1}, \eqref{eq:pfH2}, \eqref{eq:pfH3}, and \eqref{eq:pfH4} imply Equation \eqref{eq:mH}.
\end{proof}

\begin{lemma}[type \ref{it:mLH3}]\label{lem:mLH3}
	Let \(L \in \mG\) with  \(r(L) \geq 3\) and let \(H \in \mH\) be such that \(L \subset H\). Then
	\begin{equation}\label{eq:mLH3}
	\gamma_{L} \cdot \gamma_H = - \gamma_{L}^2 + \sum_{L' \Subset L} B(L', L) \cdot \gamma_{L'} \cdot \gamma_{L} - \sum_{L \Subset L' \subsetneq H} \gamma_{L} \cdot \gamma_{L'} + \sum_{L'' \subsetneq L \inclusiondot L' \subsetneq H} \gamma_{L''} \cdot \gamma_{L'} \,.	
	\end{equation}
\end{lemma}

\begin{proof}
	By \eqref{eq:R2}, we have
	\[
	\gamma_{L} \cdot \gamma_{H} = -\gamma_{L} \cdot \left(\sum_{L' \subsetneq H} \gamma_{L'}\right) \,.
	\]
	By \eqref{eq:R1} the product \(\gamma_{L} \cdot \gamma_{L'}\) vanishes unless the set \(\{L, L'\}\) is nested. By Corollary \ref{cor:nestedpair}, since
	\(L + L' \subset H\)\,, if \(\gamma_{L} \cdot \gamma_{L'} \neq 0\) then either \(L \subset L'\) or \(L' \subset L\). Therefore
	\[
	\gamma_{L} \cdot \gamma_{H} = - \gamma_{L}^2 - \sum_{L' \subsetneq L} \gamma_{L'} \cdot \gamma_{L} - \sum_{L \subsetneq L' \subsetneq H} \gamma_{L} \cdot \gamma_{L'}
	\]
	By Equation \eqref{eq:sumlneq}
	\[
	\sum_{L' \subsetneq L} \gamma_{L'} \cdot \gamma_{L} =  -\sum_{L' \Subset L} B(L', L) \cdot \gamma_{L'} \cdot \gamma_{L} \,.
	\]
    Clearly, \(\sum_{L \subsetneq L' \subsetneq H} \gamma_{L} \cdot \gamma_{L'} =
	\sum_{L \Subset L' \subsetneq H} \gamma_{L} \cdot \gamma_{L'} \,+\,
	\sum_{L \inclusiondot L' \subsetneq H} \gamma_{L} \cdot \gamma_{L'} \). Using Equation \eqref{eq:mLM} for the term \(\sum_{L \inclusiondot L' \subsetneq H} \gamma_{L} \cdot \gamma_{L'}\), we obtain
	\[
	\sum_{L \subsetneq L' \subsetneq H} \gamma_{L} \cdot \gamma_{L'} 
	= \sum_{L \Subset L' \subsetneq H} \gamma_{L} \cdot \gamma_{L'} \,-\,
	\sum_{L'' \subsetneq L \inclusiondot L' \subsetneq H} \gamma_{L''} \cdot \gamma_{L'} \,.
	\]
	Equation \eqref{eq:mLH3} follows from the above.
\end{proof}

\begin{lemma}[type \ref{it:mLH2}]\label{lem:mLH2}
	Let \(L \in \mG\) with  \(r(L) = 2\) and let \(H \in \mH\) be such that \(L \subset H\). Then
	\begin{equation}\label{eq:mLH2}
	\gamma_{L} \cdot \gamma_H = \sum_{L' \subsetneq L} \gamma_{L'}^2 - \sum_{L' \Subset L} B(L', L) \cdot \gamma_{L'} \cdot \gamma_{L} - 2 \cdot \sum_{L'' \Subset L' \subsetneq L} B(L'', L') \cdot \gamma_{L''} \cdot \gamma_{L'} \,.
	\end{equation}
\end{lemma}	

\begin{proof}
	As in the proof of Lemma \ref{lem:mLH3} and using \(r(L)=2\) we have
	\[
	\begin{aligned}
	\gamma_{L} \cdot \gamma_{H} &= - \gamma_{L}^2 - \sum_{L' \subsetneq L} \gamma_{L'} \cdot \gamma_{L} \\
	&= - \gamma_{L}^2 \,+\, \sum_{L' \Subset L} B(L', L) \cdot \gamma_{L'} \cdot \gamma_{L} \,.
	\end{aligned}
	\]
	On the other hand, by Equation \eqref{eq:mL}
	\[
	\gamma_L^2 = 
	- \sum_{L' \subsetneq L} \gamma_{L'}^2  \,+\, 
	2 \cdot \sum_{L' \Subset L}  B(L', L) \cdot \gamma_{L'} \cdot \gamma_{L} \,+\, 
	2 \cdot \sum_{L'' \Subset L' \subsetneq L}  B(L'', L') \cdot \gamma_{L''} \cdot \gamma_{L'}   
	\]
	and Equation \eqref{eq:mLH2} follows from this.
\end{proof}

\begin{lemma}[type \ref{it:LredH}]\label{lem:transv1}
	If \(L \pitchfork_r H\) with \(r(L) \geq 2\) and \(H \in \mH\) then
	\begin{equation}\label{eq:transvnonbasmon}
		\gamma_L \gamma_H = - \sum_{L' \subsetneq H \,|\, L' \Subset L} \gamma_{L'} \gamma_L \,-\,  \sum_{L' \subsetneq H \,|\, L' \pitchfork_r L} \gamma_{L'} \gamma_L \,.
	\end{equation}
\end{lemma}

\begin{proof}
	Replace \(\gamma_{H} = - \sum_{L' \subsetneq H} \gamma_{L'}\) to obtain
	\[
	\gamma_L \gamma_H = - \sum_{L' \subsetneq H} \gamma_{L'} \gamma_L
	\]
	where the sum runs over all pairs \(L', L\) such that \(\{L, L'\}\) is nested. Since \(L\) is transverse to \(H\), there are no terms with \(L \subset L'\). On the other hand, if \(L'\) is irreducible and \(L' \subset L \cap H\), then since \(L\cap H \subsetneq L\) is reducible we must have that \(L' \subsetneq L \cap H\) and therefore \(L' \Subset L\). Equation \eqref{eq:transvnonbasmon} follows from this.
\end{proof}

\begin{lemma}[type \ref{it:HredH}]\label{lem:transv2}
	If \(H \pitchfork_r H'\) with \(H, H' \in \mH\) then
	\begin{equation}\label{eq:transvHH}
		\begin{gathered}
		\gamma_H \gamma_{H'} = \sum_{L \subsetneq H \cap H'} \gamma_L^2 - \sum_{(L, M) \,|\, L \Subset M \subsetneq H\cap H'} 2 B(L, M) \cdot \gamma_L \gamma_M \\
		+ \sum_{(L, M) \,|\, L \Subset M \,,\, M \pitchfork_r H} \gamma_L \gamma_M
		\,\,+ \sum_{(L, M) \,|\, M \Subset L \,,\, L \pitchfork_r H'} \gamma_M \gamma_L
		\,\,+ \sum_{(L, M) \,|\, L \pitchfork_r M} \gamma_L \gamma_M
		\end{gathered}
	\end{equation}
	where the sums runs over all ordered pairs \((L, M)\) with \(L \subsetneq H\) and \(M \subsetneq H'\) 
\end{lemma}

\begin{proof}
	Replacing \(\gamma_H = - \sum_{L \subsetneq H} \gamma_L\) and \(\gamma_{H'} = - \sum_{M \subsetneq H} \gamma_M\) we obtain
	\[
	\gamma_H \gamma_{H'} = \sum_{(L, M)} \gamma_L \gamma_M
	\]
	where the sum runs over all ordered pairs \((L, M)\) with \(L \subsetneq H\) and \(M \subsetneq H'\) such that \(\{L, M\}\) is nested. The nested condition gives us 
	\begin{equation}\label{eq:sum4terms}
	\sum_{(L, M)} \gamma_L \gamma_M = \sum_{L \subsetneq H \cap H'} \gamma_L^2  
	+ \sum_{L \subsetneq M} \gamma_L \gamma_M + \sum_{M \subsetneq L} \gamma_M \gamma_L
	+ \sum_{L \pitchfork_r M} \gamma_L \gamma_M \,.	
	\end{equation}
	Consider the second right hand term, namely the sum over the pairs with \(L \subsetneq M\). We split the sum into two parts as follows:
	\[
	\sum_{L \subsetneq M} \gamma_L \gamma_M = \sum_{L \subsetneq M \,|\, M \subset H} \gamma_L \gamma_M + \sum_{L \subsetneq M \,|\, M \not\subset H} \gamma_L \gamma_M \,.
	\] 
    By Lemma \ref{lem:mLM}, if \(L' \inclusiondot M\) then \(\gamma_{L'} \gamma_M = - \sum_{L'' \subsetneq L'} \gamma_{L''} \gamma_M\). It follows that
	\[
	\begin{aligned}
	\sum_{L \subsetneq M \subset H} \gamma_L \gamma_M  &= \sum_{L \Subset M \subset H} \big( 1 - \big|\{L' \,|\, L \subset L' \inclusiondot M\}\big| \big) \cdot \gamma_L \gamma_M \\
	&=  - \sum_{L \Subset M \subset H} B(L, M) \cdot \gamma_L \gamma_M 
	\end{aligned}
	\]
	and by the same argument,
	\[
	\sum_{L \subsetneq M \not\subset H} \gamma_L \gamma_M  = \sum_{L \Subset M \not\subset H} \big( 1 - \big|\{L' \,|\, L \subset L' \inclusiondot M \text{ and } L' \subset H \}\big| \big) \cdot \gamma_L \gamma_M \,.
	\]
	On the other hand, the conditions \(L' \inclusiondot M\) together with \(M \not\subset H\) and \(L' \subset H\) imply that \(L' = M \cap H\). Since \(L'\) must also be irreducible, we deduce that
	\[
	\big|\{L' \,|\, L \subset L' \inclusiondot M \text{ and } L' \subset H \}\big| = \begin{cases}
	1 &\text{ if } M \cap H \in \mG \,; \\
	0 &\text{ if } M \pitchfork_r H \,.
	\end{cases}
	\]
	Therefore
	\begin{equation}\label{eq:inclusionsum}
			\sum_{L \subsetneq M} \gamma_L \gamma_M  = - \sum_{L \Subset M \,|\, M \subset H} B(L, M) \cdot \gamma_L \gamma_M \,+\, \sum_{L \Subset M \,|\, M \pitchfork_r H}  \gamma_L \gamma_M \,.
	\end{equation}
	The same expression with \(L\) and \(M\) interchanged holds for the sum over the pairs \((L, M)\) with \(M \subsetneq L\). Equation \eqref{eq:transvHH} follows from Equations \eqref{eq:sum4terms} and \eqref{eq:inclusionsum}.
\end{proof}

\subsection{Coordinate projections}\label{sec:cordproj}
Let \(\Lambda\) be the free \(\Z\)-module \(\Lambda = H^4(X, \Z)\) equipped with the basis of basic monomials \(\Delta_2\) as in Lemma \ref{lem:basisH4}.
Let \(\Delta_2^*\) be the dual basis consisting of elements \(\delta: \Lambda \to \Z\) of the following form: 
\begin{itemize}
    \item \(\delta_L\) for \(L \in \mG\) with \(r(L) \geq 3\) ; 
    \item \(\delta_{L_1 \Subset L_2}\) for \(L_1, L_2 \in \mG\) with \(L_1 \Subset L_2\) and \(r(L_2) \geq 2\) ; 
    \item \(\delta_{L_1 \pitchfork_r L_2}\) for \(L_1, L_2 \in \mG\) with \(L_1 \pitchfork_r L_2\) and \(r(L_1), r(L_2) \geq 2\).
\end{itemize}

We refer to the elements of the dual basis \(\Delta_2^*\) as the \emph{coordinate projections} on \(H^4(X, \Z)\).
The results of Section \ref{sec:nonbasic} can be summarized in the next.

\begin{corollary}\label{cor:proj}
	The coordinate projections of the non-basic monomials of \(H^4(X, \Z)\) are given as follows.
	\begin{enumerate}[label = \textup{(\roman*)}]
		\item If \(H \in \mH\) then
		\[
        \begin{aligned}
        \delta_L(\gamma_{H}^2) &= 
		\begin{cases}
			-B(L, H) &\textup{ if } H \supset L \,, \\
			\phantom{-}0 &\textup{ otherwise ;}
		\end{cases}    \\
        \delta_{L_1 \Subset L_2}(\gamma_{H}^2) &= 
		\begin{cases}
			-2B(L_1, L_2) \cdot B(L_2, H) &\textup{ if } H \supset L_2 \,,\\
			0 &\textup{ otherwise ;} 
		\end{cases} \\
         \delta_{L_1 \pitchfork_r L_2}(\gamma_{H}^2) &= 0 \,.
        \end{aligned} 
		\]

		\item If \(r(L') = 2\) then
		\[
        \begin{aligned}
        \delta_L(\gamma_{L'}^2) &= 
		\begin{cases}
			-1 &\textup{ if } L' \supset L \,,\\
			\phantom{-}0 &\textup{ otherwise ;} 
		\end{cases}    \\
        \delta_{L_1 \Subset L_2}(\gamma_{L'}^2) &= 
		\begin{cases}
			2B(L_1, L_2) &\textup{ if } L' \supset L_2 \,,\\
			0 &\textup{ otherwise ;}
		\end{cases} \\
         \delta_{L_1 \pitchfork_r L_2}(\gamma_{L'}^2) &= 0 \,.
        \end{aligned}
		\]

		\item If \(L'' \inclusiondot L'\) with \(r (L') \geq 2\) then
		\[
        \begin{aligned}
            \delta_L (\gamma_{L''} \cdot \gamma_{L'}) &= 0 \,; \\
            \delta_{L_1 \Subset L_2} (\gamma_{L''} \cdot \gamma_{L'}) &=
		\begin{cases}
			-1 &\text{ if } L_1 \subsetneq L'' \inclusiondot L' = L_2 \,,\\
			\phantom{-}0 &\textup{ otherwise ;}
		\end{cases}	\\
         \delta_{L_1 \pitchfork_r L_2}(\gamma_{L''} \cdot \gamma_{L'}) &= 0 \,.
        \end{aligned}
		\]

		\item If \(r(L') \geq 3\) and \(H \in \mH\) is such that \(H \supset L'\) then
		\[
		\begin{aligned}
		\delta_L(\gamma_{L'} \cdot \gamma_{H}) 
		&= \begin{cases}
		-1 &\textup{ if } L' = L \,,\\
		\phantom{-}0 &\textup{ otherwise ;}
		\end{cases} \\
		\delta_{L_1 \Subset L_2}(\gamma_{L'} \cdot \gamma_{H}) &= 
		\begin{cases}
		B(L_1, L_2) &\textup{ if } L_1 \Subset L' = L_2 \,,\\
		-1 &\textup{ if }  L' = L_1 \Subset L_2 \subset H  \,,\\
		\phantom{-}1 &\textup{ if } L_1 \subsetneq L' \inclusiondot L_2 \subsetneq H \,,\\
		\phantom{-}0 &\textup{ otherwise ;}
		\end{cases}\\
		\delta_{L_1 \pitchfork_r L_2}(\gamma_{L'} \cdot \gamma_{H}) &= 0 \,.
		\end{aligned}
		\]

		\item If \(r(L') = 2\) and \(H \in \mH\) is such that \(H \supset L'\) then
		\[
		\begin{aligned}
		\delta_L(\gamma_{L'} \cdot \gamma_{H}) &= 
		\begin{cases}
		1 &\textup{ if } L' \supsetneq L \,, \\
		0 &\textup{ otherwise ;}
		\end{cases} \\
		\delta_{L_1 \Subset L_2}(\gamma_{L'} \cdot \gamma_{H}) &= 
		\begin{cases}
		-B(L_1, L_2) &\textup{ if } L' = L_2 \,,\\
		-2 B(L_1, L_2) &\text{ if } L_2 \subsetneq L' \,,\\
		\phantom{-}0 &\textup{ otherwise ;}
		\end{cases}\\
		\delta_{L_1 \pitchfork_r L_2}(\gamma_{L'} \cdot \gamma_{H}) &= 0 \,.
		\end{aligned}
		\]
		
		\item If \(L' \pitchfork_r H\) with \(r(L') \geq 2\) and \(H \in \mH\) then
		\[
		\begin{aligned}
		\delta_L (\gamma_{L'} \cdot \gamma_{H}) &= 0 \,; \\
		\delta_{L \Subset M}(\gamma_{L'} \cdot \gamma_{H}) &= \begin{cases}
		-1 &\text{ if } M = L'  \textup{ and } L \subset H \,, \\
		0 &\textup{ otherwise ; }
		\end{cases}\\
		\delta_{L \pitchfork_r M}(\gamma_{L'} \cdot \gamma_{H}) &= \begin{cases}
		-1 &\textup{ if } L = L' \textup{ and } M \subset H \,, \\
		-1 &\textup{ if } M = L' \textup{ and } L \subset H \,, \\
		0 &\textup{ otherwise . }
		\end{cases}
		\end{aligned} 
		\]
		
		\item If \(H \pitchfork_r H'\) with \(H, H' \in \mH\) then
		\[
		\begin{aligned}
		\delta_L (\gamma_{H} \cdot \gamma_{H'}) &= 
		\begin{cases}
		1 &\textup{ if } L \subset H \cap H' \,,\\
		0 &\textup{ otheriwise ;}
		\end{cases} \\
		\delta_{L \Subset M}(\gamma_{H} \cdot \gamma_{H'}) &= 
		\begin{cases}
		-2B(L, M) &\textup{ if } M \subset H \cap H' \,,\\
		1 &\textup{ if } M \subset H' \,,\, L \subset H \textup{ and } M \pitchfork_r H \,, \\
		0 &\textup{ otherwise ; }
		\end{cases} \\
		\delta_{L \pitchfork_r M}(\gamma_{H} \cdot \gamma_{H'}) &= 
		\begin{cases}
		1 &\textup{ if } L \subset H \textup{ and } M \subset H' \,, \\
		1 &\textup{ if } M \subset H \textup{ and } L \subset H' \,, \\
		0 &\textup{ otherwise . }
		\end{cases}	
		\end{aligned}
		\]
	\end{enumerate}
\end{corollary}

\begin{proof}
	Item (i) follows from Equation \eqref{eq:mH}; (ii) follows from Equation \eqref{eq:mL}; (iii) follows from Equation \eqref{eq:mLM}; (iv) follows from Equation \eqref{eq:mLH3}; (v) follows from Equation \eqref{eq:mLH2}; (vi) follows from Equation \eqref{eq:transvnonbasmon}; and (vii) follows from Equation \eqref{eq:transvHH}.
\end{proof}

\section{PK metric \(\implies\) topological constraints}\label{sec:1implies2}
In this section, we show the direction (1) \(\implies\) (2) of Theorem \ref{thm:main}. That is, we show
that the existence of a PK metric implies the \emph{topological constraints} (i), (ii), (iii). Our key tool is Ohtsuki's residue formula \cite{ohtsuki} for the Chern numbers of a holomorphic vector bundle equipped with a logarithmic connection, which we recall in Section \ref{sec:othsuki}. In Section \ref{sec:ohtsuki2} we apply Ohtsuki's formula to calculate \(c_1(\mE)\) and \(c_2(\mE)\), where \(\mE = \pi^*T\CP^n\) and \(\pi: X \to \CP^n\) is the minimal De Concini-Procesi model of \(\mH\), in terms of the residues of the pullback of the Levi-Civita connection of a PK metric. In Section \ref{sec:ohtsbasic} we use our results from Section \ref{sec:cohomology} to express our formulas for \(c_1(\mE)\) and \(c_2(\mE)\) in terms of basic monomials.  
The proof of Theorem \ref{thm:main} (1) \(\implies\) (2) is finally carried in Section \ref{sec:pf1implies2}.

\subsection{Ohtsuki's formula}\label{sec:othsuki}

Let \(M_n(\C)\) be the space of \(n \times n\) matrices with entries in \(\C\). The \emph{Chern polynomial} \(c_k(A)\) of \(A \in M_n(\C)\) is the coefficient of \(t^k\) in the expansion
\begin{equation}\label{eq:chernpol}
\det (I + At) = \sum_{k=0}^n c_k(A) t^k \,.
\end{equation}

The Chern polynomial satisfies the following properties: (i) \(c_k(A)\) is a homogeneous polynomial of the entries of \(A\) of degree \(k\); and (ii) \(c_k(A)\) is invariant by conjugation, i.e., \(c_k(A) = c_k(GAG^{-1})\) for all \(G \in \GL_n(\C)\). 
Let \(c_k(A_1, \ldots, A_k)\) be the polarization of the Chern polynomial \(c_k(A)\). More precisely, \(c_k(A_1, \ldots, A_k)\) is the unique multilinear symmetric map such that \(c_k(A, \ldots, A) = c_k(A)\). 
The next result is proved in \cite{ohtsuki}.

\begin{theorem}\label{thm:ohtsuki}
	Let \(X\) be a compact complex manifold of dimension \(n\) and let \(D = \sum_{i \in I} D_i\) be a simple normal crossing divisor, where \(I\) is a finite set of indices. Let \(\mE\) be a holomorphic vector bundle on \(X\) and let \(\nabla\) be a logarithmic connection on \(\mE\) with poles along \(D\). For each irreducible component \(D_i\) of \(D\), let \(A_i = \Res_{D_i}(\nabla)\).
	
	Fix an integer \(k\) with \(1 \leq k \leq n\). Assume that the intersections \(D_I= D_{i_1} \bigcap \ldots \bigcap D_{i_k}\) are connected for any \(k\)-tuple \(I = (i_1, \ldots, i_k) \in I^k\). Then the following holds:
	\begin{equation}\label{eq:ohtsukithm}
	c_k(\mE) = \sum_{(i_1, \ldots, i_k) \in I^k} c_k\big( A_{i_1} \,,\, \ldots \,,\, A_{i_k} \big) \cdot \prod_{s=1}^k c_1([D_{i_s}]) \,,
	\end{equation}
	where \(c_k(\mE) \in H^{2k}(X, \Z)\) is the \(k\)-th Chern class of \(\mE\).
\end{theorem}

\begin{remark}
	Since \(D\) is normal crossing, the residues \(A_i\) are holomorphic sections of \(\End \mE\) defined on the whole divisors \(D_i\) (see Proposition \ref{prop:normcross}). The polarized Chern polynomial \(c_k\big(A_{i_1} \,,\, \ldots \,,\, A_{i_k} \big)\) is pointwise defined along \(D_I\). The assumption that \(D_I\) is connected, together with compactness, imply that \(c_k\big(A_{i_1} \,,\, \ldots \,,\, A_{i_k} \big)\) is constant.
\end{remark}

We refer to Equation \eqref{eq:ohtsukithm} as \emph{Ohtsuki's formula}.

\subsection{Calculation of Chern classes in terms of residues}\label{sec:ohtsuki2}

\subsubsection{The Levi-Civita connection and its pullback}
Let \(g\) be a PK metric on \(\CP^n\) with cone angles \(2\pi\alpha_H \in (0, 2\pi)\) along the hyperplanes \(H\) of an arrangement \(\mH\).
Let \(\nabla\) be the Levi-Civita connection of \(g\). By Proposition \ref{prop:mettoconect}, \(\nabla\) is a logarithmic connection on \(T\CP^n\) adapted to the weighted arrangement \(\{(H, a_H)\}\) with \(a_H = 1 - \alpha_H\). 

For \(H \in \mH\) we let \(A_H := \Res_H(\nabla)\). Since \(\nabla\) is adapted to \((\mH, \ba)\), the residue \(A_H\) is holomorphic on the whole \(H\) and satisfies \(A_H = 0 \cdot TH \oplus a_H \cdot N_H\), where \(N_H\) is a holomorphic line subbundle of \(T\CP^n|_H\). More generally, if \(L\) is a non-empty and proper linear subspace of \(\CP^n\) obtained as intersection of elements in \(\mH\), we let \(A_L\) be the holomorphic endomorphism of \(T\CP|_L\) given by
\begin{equation}
A_L := \sum_{H \,|\, L \subset H} A_H \,.	
\end{equation}

\begin{lemma}
	If \(L \in \mLi\) then there is a holomorphic subbundle \(N_L \subset T\CP^n|_L\) such that
	\begin{equation}\label{eq:AL}
	A_L = 0 \cdot TL \oplus a_L \cdot N_L \,.	
	\end{equation}
\end{lemma}

\begin{proof}
	Let \(x \in L\) and let \((z_1, \ldots, z_n)\) be linear coordinates centred at \(x\) as guaranteed by Proposition \ref{prop:lincoordmet}. Thus, we might assume that \(\nabla\) is a (flat, torsion-free) standard connection. The klt condition (Lemma \ref{lem:kltmet}) implies that the Non-Zero Weights Assumptions in \cite{pkc} are satisfied. The result follows by \cite[Lemma 3.28]{pkc}.
\end{proof}

\begin{remark}
Note that the formula \(a_L = (1/ r(L)) \sum_{H \supset L} a_H\) follows from Equation \eqref{eq:AL} by taking the trace.	
\end{remark}

\begin{lemma}\label{lem:commuteresidues}
	If \(L \subset M\) or \(L \pitchfork_r M\) then
	\begin{equation}
		[A_L, A_M] = 0 \,.
	\end{equation}
\end{lemma}

\begin{proof}
	Same as before, taking linear coordinates we can assume that \(\nabla\) is a (flat, torsion free) standard connection. If \(L \subset M\) then \([A_L, A_M] = 0\) by \cite[Proposition 3.8]{pkc}. If \(L \pitchfork_r M\) then \(H \pitchfork_r H'\) for any \(H, H' \in \mH\) with \(L \subset H\) and \(M \subset H'\). By \cite[Lemma 3.20]{pkc} \([A_H, A_{H'}] = 0\) and therefore \([A_L, A_M] = 0\).
\end{proof}

\subsubsection{The pullback connection}
Let \(X \xrightarrow{\pi} \CP^n\) be the minimal De Concini-Procesi model of \(\mH\). Let \(\mE = \pi^*(T\CP^n)\) be the pull-back tangent bundle and let \(\tn = \pi^* \nabla\) be the pullback of the Levi-Civita connection. 

\begin{lemma}\label{lem:pullbackconnection}
	The pullback connection \(\tn\) is a logarithmic connection on \(\mE \to X\) with polar set the simple normal crossing divisor  \(D = \pi^{-1}(\mH)\). Moreover, if \(D_L\) is an irreducible component of \(D = \bigcup_{L \in \mLi} D_L\) with \(\pi(D_L) = L\), then
	\begin{equation}
		\Res_{D_L}(\tn) = \pi^*A_L \,.
	\end{equation}
\end{lemma}

\begin{proof}
	By Proposition \ref{prop:mettoconect} the connection \(\nabla\) is adapted. In particular, \(\nabla\) has holomorphic residues. The statement then follows from Proposition \ref{prop:pullback1form} in the appendix after taking local trivializations.
\end{proof}

Write \(\tA_L = \Res_{D_L}(\tn)\).

\subsubsection{General formula for \(c_k(\mE)\)}
Applying Ohtsuki's formula to the logarithmic connection \(\tn\) on \(\mE \to X\) we obtain
\begin{equation}\label{eq:ohtsukithm2}
c_k(\mE) = \sum_{(L_1, \ldots, L_k) \in \mLi^k} c_k\big(A_{L_1} \,,\, \ldots \,,\, A_{L_k} \big) \cdot \prod_{s=1}^k \gamma_{L_s} \,,
\end{equation}
where we have used that \(\gamma_L = c_1(D_L)\) for all \(L \in \mLi\) and 
\[
c_k\big(\tA_{L_1} \,,\, \ldots \,,\, \tA_{L_k} \big) = c_k\big(A_{L_1} \,,\, \ldots \,,\, A_{L_k} \big) \,,
\] 
as \(\tA_L = \pi^*A_L\) for all \(L \in \mLi\) by Lemma \ref{lem:pullbackconnection}.

\begin{notation}
	We write \(r_L = r(L) = \codim_{\CP^n}L\).
\end{notation}

\subsubsection{First Chern class}

\begin{notation}\label{not:h}
	Let \(h\) be the generator of \(H^2(\CP^n, \Z)\) given by \(h = c_1\big(\mO_{\P^n}(1)\big)\).
\end{notation}

\begin{lemma}
	Ohtsuki's formula for \(c_1(\mE)\) is as follows:
	\begin{equation}\label{eq:c1ohtsuki}
		(n+1) \cdot \pi^*h = \sum_{L \in \mLi} r_L \cdot a_L \cdot \gamma_L \,.
	\end{equation}
\end{lemma}

\begin{proof}
	Equation \eqref{eq:ohtsukithm2} for \(k=1\) asserts that \(c_1(\mE) = \sum c_1(A_L) \cdot \gamma_L\).
	Clearly, since \(c_1(T\CP^n) = (n+1)\cdot h\), we have
	 \(c_1(\mE) = (n+1) \cdot \pi^*h\). On the other hand, if \(A \in M_n(\C)\) then the Chern polynomial \(c_1(A)\) is given by \(c_1(A) = \tr A\). By Equation \(\eqref{eq:AL}\), we have \(\tr A_L = r_L \cdot a_L\) and Equation \eqref{eq:c1ohtsuki} follows from this.
\end{proof}

Let \(\eta \in H^2(X, \R)\) be given by
\begin{equation}\label{eq:etadef}
\eta := (n+1) \cdot \pi^*h - \sum_{L \in \mLi} r_L \cdot a_L \cdot \gamma_L 	
\end{equation}
so that Equation \eqref{eq:c1ohtsuki} is equivalent to \(\eta = 0\).

\begin{lemma}
	In the basis of \(H^2(X, \R)\) given by \(\pi^*h\) and \(\gamma_L\) for \(L \in \mLi^{\circ}\) we have
	\begin{equation}\label{eq:etabasis}
		\eta = \left(n+1 - \sum_{H \in \mH} a_H\right) \cdot \pi^*h \,.
	\end{equation}
\end{lemma}

\begin{proof}
	For \(H \in \mH\) we have 
	\[
	\gamma_{H} = \pi^*h - \sum_{L \in \mLi^{\circ} \,|\, L \subset H} \gamma_L \,.
	\]
	Therefore 
	\[
	\begin{aligned}
		\sum_{L \in \mLi} r_L \cdot a_L \cdot \gamma_L  &= \left(\sum_{H \in \mH} a_H\right) \cdot \pi^*h - \sum_{L \in \mLi^{\circ}} \underbrace{\left(r_L \cdot a_L - \sum_{H \in \mH \,|\, H \supset L} a_H\right)}_{=0} \cdot \gamma_L \\
		&= \left(\sum_{H \in \mH} a_H\right) \cdot \pi^*h 
	\end{aligned}
	\]
	and
	\[
	\eta = (n+1) \cdot \pi^*h - \sum_{L \in \mLi} r_L \cdot a_L \cdot \gamma_L = \left(n+1 - \sum_{H \in \mH} a_H\right) \cdot \pi^*h
	\]
	as we wanted to show.
\end{proof}

\subsubsection{Second Chern class}

\begin{lemma}
	Ohtsuki's formula for \(c_2(\mE)\) is as follows
	\begin{equation}\label{eq:ohtsuki2}
	\begin{gathered}
	\binom{n+1}{2} \pi^* h^2 = \sum_{L \in \mLi^{\circ}} \frac{r_L(r_L-1)}{2} \cdot a_L^2 \cdot \gamma_L^2 \\ \,+\, \sum_{L \subsetneq M} r_M (r_L -1) \cdot a_L a_M \cdot \gamma_L \gamma_M \,+\, \sum_{L \pitchfork_r M} r_L r_M \cdot a_L a_M \cdot \gamma_L \gamma_M \,,
	\end{gathered}
	\end{equation}
	where we sum only once for each pair in the last sum.
\end{lemma}

\begin{proof}
	Equation \eqref{eq:ohtsukithm2} for \(k=2\) is
	\begin{equation*}
		c_2(\mE) = \sum_{(L, M) \in \mLi^2} c_2 \big(A_{L}, A_{M} \big) \cdot \gamma_{L} \gamma_{M} \,.
	\end{equation*}
	Note that we only need to sum over the pairs \((L, M) \in \mLi^2\) such that the set \(\{L, M\}\) is nested, otherwise \(\gamma_L \gamma_M = 0\). Therefore, we might assume that either, \(L \subset M\), or \(M \subset L\), or \(L \pitchfork_r M\).
	\begin{equation}\label{eq:ohtspfc2}
	c_2(\mE) = \sum_{L \in \mLi} c_2(A_L) \cdot \gamma_L^2 \,+\, \sum_{L \subsetneq M} 2c_2(A_L, A_M) \cdot \gamma_L \gamma_M \,+\, \sum_{L \pitchfork_r M} 2c_2(A_L, A_M) \cdot \gamma_L \gamma_M \,.	
	\end{equation}

	We proceed to calculate the coefficients \(c_2(A_L, A_M)\). To begin with, note that if \(A \in M_n(\C)\) then the second Chern polynomial \(c_2(A)\) is given by\footnote{This can be easily checked by considering the case that \(A = \diag(a_1, \ldots, a_n)\). Then \(\det(I + At) = \prod (1+a_it)\) and 
	\(c_2(A) = \sum_{i < j} a_ia_j = (1/2) \big( (\sum a_i)^2 - \sum a_i^2 \big)\)\,.}
	\begin{equation}\label{eq:c2A}
	c_2(A) = \dfrac{(\tr A)^2 - \tr(A^2)}{2} \,,
	\end{equation}
	and for \(A, B \in M_n(\C)\) the symmetric form \(c_2(A, B)\) is
	\begin{equation}\label{eq:c2AB}
	2c_2(A, B) = c_2(A + B) - c_2(A) - c_2(B) \,.	
	\end{equation}

\textbf{Claim 1:} if \(L \in \mLi\) then \(c_2(A_L)\) is given by
\begin{equation}\label{eq:c2AL}
c_2(A_L) = \frac{r_L(r_L-1)}{2} \cdot a_L^2 \,.
\end{equation}	

\emph{Proof of Claim 1.} We can diagonalize and assume that \(A_L = \diag (a_L, \ldots a_L, 0 \ldots, 0)\) where \(a_L\) is in the first \(r_L\) entries. Then \((\tr A_L)^2 = r_L^2 a_L^2\) while \(\tr(A_L^2) = r_L a_L^2\), and Equation \eqref{eq:c2AL} follows from Equation \eqref{eq:c2A}. \,\, \qedsymbol \newpar

\textbf{Claim 2:}  if \(L, M \in \mLi\) and \(L \subset M\) then \(c_2(A_L, A_M)\) is given by
\begin{equation}\label{eq:c2LMinc}
2 c_2(A_L, A_M) = r_M (r_L -1) \cdot a_L a_M \,;
\end{equation}

\emph{Proof of Claim 2.} By Lemma \ref{lem:commuteresidues} \([A_L, A_M] = 0\) so we can simultaneously diagonalize and assume that \(A_L = \diag (a_L, \ldots, a_L, 0, \ldots, 0)\) with \(a_L\) in the first \(r_L\) entries and \(A_M = \diag (a_M, \ldots, a_M, 0, \ldots, 0)\) with \(a_M\) in the first \(r_M\) entries and \(r_M \leq r_L\). Then
\[
A_L + A_M = \diag (a_L + a_M, \ldots, a_L + a_M, a_L, \ldots, a_L, 0 \ldots, 0) \,,
\]
\((\tr(A_L + A_M))^2 = (r_L a_L + r_M a_M)^2\), and \(\tr (A_L + A_M)^2 = r_M (a_L + a_M)^2 + (r_L-r_M)a_L^2\). From this we get that
\[
c_2(A_L + A_M) = c_2(A_L) + c_2(A_M) + r_M (r_L-1) a_L a_M
\]
and Equation \eqref{eq:c2LMinc} follows from Equation \eqref{eq:c2AB}. \,\, \qedsymbol \newpar

\textbf{Claim 3:} if \(L, M \in \mLi\) and \(L \pitchfork_r M\) then \(c_2(A_L, A_M)\) is given by
\begin{equation}\label{eq:c2LMtransv}
	2 \cdot c_2(A_L, A_M) = r_L r_M \cdot a_L a_M \,.
\end{equation}

\emph{Proof of Claim 3.} By Lemma \ref{lem:commuteresidues} \([A_L, A_M] = 0\) so we can simultaneously diagonalize and assume that \(A_L = \diag (a_L, \ldots, a_L, 0, \ldots, 0)\) with \(a_L\) in the first \(r_L\) entries and \(A_M = \diag (0, \ldots, 0, a_M, \ldots, a_M)\) with \(a_M\) in the last \(r_M\) entries and \(r_M + r_L \leq n\) (because the intersection \(L \cap M\) is transversal), so there is no overlap between the \(a_L\) and \(a_M\) entries. Then
\[
A_L + A_M = \diag (a_L, \ldots, a_L, 0, \ldots, 0, a_M \ldots, a_M) \,,
\]
\((\tr(A_L + A_M))^2 = (r_L a_L + r_M a_M)^2\), and \(\tr (A_L + A_M)^2 = r_L a_L^2 + r_M a_M^2\). We get that
\[
c_2(A_L + A_M) = c_2(A_L) + c_2(A_M) + r_L r_M a_L a_M
\]
and Equation \eqref{eq:c2LMtransv} follows from Equation \eqref{eq:c2AB}. \,\, \qedsymbol \newpar

To finish the proof of the lemma, we note that \(c_2(\mE) = \binom{n+1}{2} \pi^* h^2\) as
\[
c_2(T\CP^n) = \binom{n+1}{2} h^2 \,.
\]
Equation \eqref{eq:ohtsuki2} then follows from Equation \eqref{eq:ohtspfc2} together with Equations \eqref{eq:c2AL}, \eqref{eq:c2LMinc}, and \eqref{eq:c2LMtransv}.
\end{proof}

\subsection{Ohtsuki's formula in terms of basic monomials}\label{sec:ohtsbasic}

Define \(\Omega \in H^4(X, \R)\) by
\begin{equation}\label{eq:Omega}
	\Omega := \binom{n+1}{2} \pi^* h^2 - \big( \txA + \txB + \txC \big)
\end{equation}
with
\begin{equation}\label{eq:ABC}
	\begin{aligned}
	\txA &= \sum_{L \in \mLi^{\circ}} \frac{r_L(r_L-1)}{2} \cdot a_L^2 \cdot \gamma_L^2 \,; \\ 
	\txB &= \sum_{L \subsetneq M} r_M (r_L -1) \cdot a_L a_M \cdot \gamma_L \gamma_M \,; \\
	\txC &= \sum_{L \pitchfork_r M} r_L r_M \cdot a_L a_M \cdot \gamma_L \gamma_M \,.
	\end{aligned}
\end{equation}

Our definition of \(\Omega\) is such that Ohtsuki's formula \eqref{eq:ohtsuki2} for \(c_2(\mE)\) is equivalent to \(\Omega = 0\).
Recall that \(H^4(X, \R)\) has a basis given by the \emph{basic monomials}: \(\gamma_L^2\) for \(r_L \geq 3\); \(\gamma_L \gamma_M\) for \(L \Subset M\) with \(r_M \geq 2\); and \(\gamma_L \gamma_M\) with \(L \pitchfork_r M\) and \(r_L, r_M \geq 2\).
Here, we allow \(L = \emptyset\) with \(r_{\emptyset} = n+1\) and \(\gamma_{\emptyset} = -\pi^*h\). Using basic monomials, we can write

\begin{equation}\label{eq:Omegabasis}
	\Omega = \sum_{r_L \geq 3} c_L \cdot \gamma_L^2 \,+\, \sum_{L \Subset M} c_{L \Subset M} \cdot \gamma_L \gamma_M \,+\, \sum_{L \pitchfork_r M} c_{L \pitchfork_r M} \cdot \gamma_L \gamma_M
\end{equation}
for some unique real coefficients \(c_L\), \(c_{L \Subset M}\), and \(c_{L \pitchfork_r}\). We will calculate these coefficients in terms of the Hirzebruch  quadratic form of localized arrangements.

\begin{definition}
   Let \(\mL\) be the set of all subspaces \(L \subset \CP^n\) obtained as intersections of members of \(\mH\). 
The \emph{localization} of \(\mH\) at \(L \in \mL\) consists of all hyperplanes in \(\mH\) that contain \(L\), we denote it by \(\mH_L := \{H \in \mH \,|\, H \supset L\}\).
Similarly, we write \(\mG^2_L = \{L' \in \mG^2 \,|\, L' \supset L \}\). 

Let \(\mG\) denote the set of all proper (but possibly empty) \emph{irreducible subspaces} \(L \subset \CP^n\) (see Definition \ref{def:G}). For \(L \in \mG\) we define \(Q_L : \R^{\mH} \to \R\) by

\begin{equation}\label{eq:QL}
	Q_L(\ba) = \sum_{L' \in \mG^2_L} a_{L'}^2 \,-\,  \frac{1}{2} \cdot \sum_{H \in \mH_L} B(L, H) \cdot a_H^2 \,-\, \frac{r(L)}{2} \cdot a_L^2 \,,
\end{equation}
where \(B(L, H) = |\{L' \in \mG^2_L \,|\, L' \subset H\} | -1\) and \(a_L= r(L)^{-1} \sum_{H \supset L} a_H\). 
\end{definition}

\begin{remark} 
	If \(L= \emptyset\) then \(Q_{\emptyset} = Q\) is the Hirzebruch quadratic form of \(\mH\) (Equation (\ref{eq:hirquadform})).
    We use the convention that \(\codim_{\CP^n}\emptyset = n+1\).
\end{remark}

\begin{remark}
    If \(L \in \mG\) and \(r(L)=1\) or \(r(L)=2\), then the quadratic form \(Q_L\) is identically zero.
\end{remark}

\begin{remark}\label{rmk:restrweights}
	The quadratic form \(Q_L\) is naturally identified with the Hirzebruch quadratic form, say \(Q'\), of the essential arrangement \(\mH_L \big/ L\) in \(\CP^{r(L)-1}\) obtained by intersecting the members of the localization \(\mH_L\) with a projective linear subspace \(P \subset \CP^n\) complimentary to \(L\), i.e., such that \(P \cap L = \emptyset\) and \(\dim L + \dim P = n-1\).
	To explain this identification, note that an element \(\ba \in \R^{\mH}\) corresponds to a function \(\ba: \mH \to \R\) with \(\ba(H) = a_H\). The inclusion \(\mH_L \subset \mH\) induces a linear surjection \(\pi_L : \R^{\mH} \to \R^{\mH_L}\) via restriction, and  \(Q_L(\ba) = Q'(\pi_L(\ba))\).
\end{remark}

\begin{proposition}\label{prop:ohtsukibasis}
	The coefficients of \(\Omega\) in Equation \eqref{eq:Omegabasis} are given as follows:
	\begin{enumerate}[label=\textup{(\arabic*)}]
		\item if \(r(L) \geq 3\) then \(c_L = Q_L(\ba)\) ;
		\item if \(L \Subset M\) and \(r(M) \geq 3\) then \(c_{L \Subset M} = -2B(L, M) \cdot Q_M(\ba)\) ;
		\item if \(L \Subset M\) and \(r(M) = 2\) then \(c_{L \Subset M} = 0\) ;
		\item if \(L \pitchfork_r M\) and \(r(L), r(M) \geq 2\) then \(c_{L \pitchfork_r M} = 0\).
	\end{enumerate}
\end{proposition}

\subsubsection{An alternative formula for the Hirzebruch quadratic form}
Let \(\mR^2\) be the set of all reducible codimension \(2\) subspaces, so that \(\mL^2 = \mG^2 \cup \mR^2\). 
An alternative formula for the Hirzebruch quadratic form is as follows:
\begin{equation}\label{eq:quadformalt}
	Q(\ba) =  \frac{n}{2(n+1)} \cdot s^2  \,-\, \sum_{L \in \mG^2}  a_L^2 \,-\, \sum_{H \cap H' \in \mR^2} a_H a_{H'} \,,
\end{equation}
where the last sum runs over all codimension \(2\) reducible subspaces, that is, over all subsets \(\{H, H'\} \subset \mH\) with \(H \cap H' \in \mR^2\). 
The equivalence of Equations \eqref{eq:hirquadform} and \eqref{eq:quadformalt} can be easily proved by completing squares, or alternatively see \cite[Corollary 7.6]{miyaokayau}.

Similarly, if \(L \in \mG\) and \(Q_L\) is the Hirzebruch quadratic of the localized essential arrangement \(\mH_L/L\) (see Equation \eqref{eq:QL}) then we have the alternative formula
\begin{equation}\label{eq:quadforlocalt}
	Q_L(\ba) = \binom{r_L}{2} \cdot a_L^2 \,-\, \sum_{L' \in \mG_L^2}  a_{L'}^2 \,-\, \sum_{H \cap H' \in \mR_L^2} a_H a_{H'} \,,
\end{equation}
where \(\mR^2_L = \{L' \in \mR^2 \,|\, L' \supset L\}\) and we used that \(a_L = s_L / r_L\) with \(s_L = \sum_{H \supset L} a_H\).

\subsubsection{Proof of Proposition \ref{prop:ohtsukibasis}}
Let \(\delta_L\), \(\delta_{L\Subset M}\) and \(\delta_{L \pitchfork_r M}\) be the corresponding dual basis of \(H^4(X, \R)^*\) as in Section \ref{sec:cordproj}. The coefficients of \(\Omega\) in Equation \eqref{eq:Omegabasis} are given by
\[
c_L = \delta_L(\Omega) \,\,,\,\, c_{L \Subset M} = \delta_{L \Subset M}(\Omega) \,, \,\,\text{ and }\,\, c_{L \pitchfork_r M} = \delta_{L \pitchfork_r M}(\Omega) \,.
\]
To calculate these coefficients, expand the non-basic monomials in Equations \eqref{eq:Omega} and \eqref{eq:ABC} in terms of basic ones. 
Proposition \ref{prop:ohtsukibasis} then follows from Lemmas \ref{lem:cempty}, \ref{lem:coefL}, \ref{lem:cLSubsetM}, and \ref{lem:cLpitchforkM} below. 

\begin{lemma}\label{lem:cempty}
	We have \(c_{\emptyset} = Q(\ba)\).
\end{lemma}

\begin{proof}
	Equation \eqref{eq:Omega} implies
	\[
	c_{\emptyset} = \binom{n+1}{2} - \left( \delta_{\emptyset}(\txA) + \delta_{\emptyset}(\txB) + \delta_{\emptyset}(\txC) \right) \,.
	\]

    \textbf{Consider first the A term.}
	By Corollary \ref{cor:proj}(ii) we have \(\delta_{\emptyset}(\gamma_L^2) = -1\) for all \(L \in \mLi^{n-2}\). On the other hand, if \(L \in \mLi = \mG \setminus \{\emptyset\}\) has codimension \(\geq 3\) then \(\gamma_L^2\) is basic and \(\delta_{\emptyset}(\gamma_L^2) = 0\). Therefore,
	\[
	\begin{aligned}
		\delta_{\emptyset} (\txA) &= \sum_{L \in \mLi^{\circ}} \frac{r_L(r_L-1)}{2} a_L^2 \cdot \delta_{\emptyset} \big(\gamma_L^2\big) \\
		&= -\sum_{L \in \mG^2} a_L^2 \,.
	\end{aligned}
	\] 
	
	\textbf{Consider now the \(\txB\) term.} If \(r_M \geq 2\) and \(L \Subset M\) then the monomial \(\gamma_L\gamma_M\) is basic so \(\delta_{\emptyset}(\gamma_L\gamma_M) = 0\). Therefore, we only need to consider pairs \(L, M \in \mG \setminus \{\emptyset\}\) with \(L \subsetneq M\) in one of the following two cases.
	\begin{itemize}
		\item Case 1: \(r_M \geq 2\) and \(L \inclusiondot M\). By Corollary \ref{cor:proj}(iii) we have \(\delta_{\emptyset}(\gamma_L\gamma_M) = 0\). 
		\item Case 2: \(r_M = 1\) so \(M = H\) is a hyperplane \(H \in \mH\). We divide this into two further subcases.
		\begin{itemize}
			\item Subcase 2A: \(L \in \mG^2\). Then by Corollary \ref{cor:proj}(v) we have \(\delta_{\emptyset}(\gamma_L\gamma_H) = 1\).
			\item Subcase 2B: \(L \in \mG^{\geq 3} \setminus \{\emptyset\}\). Then by Corollary \ref{cor:proj}(iv) we have \(\delta_{\emptyset}(\gamma_L\gamma_H) = 0\).
		\end{itemize}
	\end{itemize}
	From the above, it follows that
	\[
	\begin{aligned}
	\delta_{\emptyset}(\txB) &= \sum_{L \subsetneq M} r_M (r_L -1) \cdot a_L a_M \cdot \delta_0(\gamma_L \gamma_M) \\
	&= \sum_{L \in \mG^2} \sum_{H \in \mH \,|\, H \supset L} a_L a_H = 2 \cdot \sum_{L \in \mG^2} a_L^2 \,.
	\end{aligned}
	\]
	
	\textbf{We calculate the C term.} The monomials \(\gamma_L \gamma_M\) with \(L \pitchfork_r M\) are basic if \(r_M\) and \(r_L\) are both \(\geq 2\), so we only have to consider the case when at least one of \(L\) or \(M\) is a hyperplane. By 
	Corollary \ref{cor:proj} (vi), if \(L \in \mG^{\geq 2}\) and \(H \in \mH\) are such that \(L \pitchfork_r H\) then \(\delta_{\emptyset}(\gamma_L \gamma_H) = 0\). On the other hand, if \(H, H' \in \mH\) have reducible intersection, then
	Corollary \ref{cor:proj} (vii) implies that \(\delta_{\emptyset} (\gamma_{H} \gamma_{H'}) =1\). Therefore:
	\[
	\begin{aligned}
	\delta_{\emptyset}(\txC) &= \sum_{L \pitchfork_r M} r_L r_M \cdot a_L a_M \cdot \delta_{\emptyset}(\gamma_L \gamma_M) \\
	&= \sum_{H \pitchfork_r H' \in \mR^2} a_H a_{H'} \,.
	\end{aligned}
	\]
    
	From the calculations of the A, B, C terms; we conclude that
	\[
	\delta_{\emptyset}(\Omega) = \binom{n+1}{2} - \sum_{L \in \mG^2} a_L^2 - \sum_{H \pitchfork_r H' \in \mR^2} a_H a_{H'} \,.
	\]
	Since \(\sum_{H \in \mH} a_H = n +1\), the above expression is equal to \(Q(\ba)\), see Equation \eqref{eq:quadformalt}.
\end{proof}

\begin{lemma}\label{lem:coefL}
	If \(L \in \mG \setminus \{\emptyset\}\) has \(r_L \geq 3\)  then \(c_L = Q_L(\ba)\).
\end{lemma}

\begin{proof}
	Equation \eqref{eq:Omega} implies
	\[
	c_{L} =  - \left( \delta_{L}(\txA) + \delta_{L}(\txB) + \delta_{L}(\txC) \right) \,.
	\]
	If \(L' \in \mG^2\) then, by Corollary \ref{cor:proj}(ii) we have 
	\[
	\delta_{L}(\gamma_{L'}^2) = \begin{cases}
	-1 &\text{ if } L' \supset L \\
	0 &\text{ otherwise. } 
	\end{cases} 
	\] 
	On the other hand, if \(L' \in \mG \setminus \{\emptyset\}\) has rank \(r_L \geq 3\) then \(\gamma_{L'}^2\) is basic so
	\[
	\delta_{L}(\gamma_{L'}^2) = 
	\begin{cases}
	1 &\text{ if } L' = L \\
	0 &\text{ otherwise. }
	\end{cases} 
	\]
	We conclude that
	\[
	\delta_{L} (\txA) = \binom{r_L}{2} a_L^2 \,-\, \sum_{L' \in \mG_L^2} a_{L'}^2 \,.
	\] 
	
	Consider now the \(\txB\) term. If \(r_M \geq 2\) and \(L' \Subset M\) then the monomial \(\gamma_L\gamma_M\) is basic so \(\delta_{L}(\gamma_{L'}\gamma_M) = 0\). Therefore, we only need to consider pairs \(L', M \in \mG \setminus \{\emptyset\}\) with \(L \subsetneq M\) in one of the following cases.
	\begin{itemize}
		\item Case 1: \(r_M \geq 2\) and \(L' \inclusiondot M\). By Corollary \ref{cor:proj}(iii) we have \(\delta_{L}(\gamma_{L'}\gamma_M) = 0\). 
		\item Case 2: \(r_M = 1\) so \(M = H\) is a hyperplane \(H \in \mH\). We divide this into two further subcases.
		\begin{itemize}
			\item Subcase 2A: \(L' \in \mG^2\). Then by Corollary \ref{cor:proj}(v) we have
			\[
			\delta_{L}(\gamma_{L'}\gamma_{H}) = 
			\begin{cases}
			1 &\text{ if } L' \supset L \\
			0 &\text{ otherwise. } 
			\end{cases} 
			\] 
			\item Subcase 2B: \(L' \in \mG^{\geq 3} \setminus \{\emptyset\}\). Then by Corollary \ref{cor:proj}(iv) we have 
			\[
			\delta_{L}(\gamma_{L'}\gamma_{H}) = 
			\begin{cases}
			-1 &\text{ if } L' = L \\
			0 &\text{ otherwise. } 
			\end{cases} 
			\] 
		\end{itemize}
	\end{itemize}
	From the above, it follows that
	\[
	\begin{aligned}
	\delta_{L}(\txB) &= \sum_{\emptyset \subsetneq L' \subsetneq M} r_M (r_{L'} -1) \cdot a_{L'} a_M \cdot \delta_{L}(\gamma_{L'} \gamma_M) \\
	&= \sum_{L' \in \mG_L^2} \sum_{H \in \mH \,|\, H \supset L'} a_{L'} a_H \,-\, (r_L-1) a_L \cdot \sum_{H \in \mH_L} a_H \\
	&= 2 \sum_{L' \in \mG_L^2} a_{L'}^2 \,-\, r_L (r_L-1) a_L^2 \,.
	\end{aligned}
	\]
	
	For the C term, the monomials \(\gamma_{L'} \gamma_M\) with \(L' \pitchfork_r M\) are basic if both \(r_M\) and \(r_{L'}\) are \(\geq 2\), so we might assume that \(M\) is a hyperplane. By Corollary \ref{cor:proj} (vi), if \(L' \in \mG^{\geq 2}\) and \(H \in \mH\) are such that \(L' \pitchfork_r H\) then \(\delta_{L}(\gamma_{L'} \gamma_H) = 0\). On the other hand, if \(H, H' \in \mH\) have reducible intersection, then Corollary \ref{cor:proj} (vii) implies that \(\delta_{L} (\gamma_{H} \gamma_{H'}) =1\) provided \(L\) lies in the intersection of \(H\) and \(H'\). Therefore:
	\[
	\begin{aligned}
	\delta_{L}(\txC) &= \sum_{L' \pitchfork_r M} r_{L'} r_M \cdot a_{L'} a_M \cdot \delta_{L}(\gamma_{L'} \gamma_M) \\
	&= \sum_{H \pitchfork_r H' \in \mR_L^2} a_H a_{H'} \,.
	\end{aligned}
	\]
	Using that \( - c_L = \delta_L(\txA) + \delta_L(\txB) + \delta_L(\txC)\) and adding the terms \(\delta_L(\txA)\), \(\delta_L(\txB)\), and \(\delta_L(\txC)\) we obtain:
	\[
	\begin{aligned}
	-c_L &= \binom{r_L}{2} a_L^2 \,-\, \sum_{L' \in \mG_L^2} a_{L'}^2 \,+\,  2 \sum_{L' \in \mG_L^2} a_{L'}^2 \,-\, r_L (r_L-1) a_L^2 \,+\, \sum_{H \pitchfork_r H' \in \mR_L^2} a_H a_{H'} \\
	&= -\binom{r_L}{2} a_L^2 \,+\, \sum_{L' \in \mG_L^2} a_{L'}^2 \,+\, \sum_{H \pitchfork_r H' \in \mR_L^2} a_H a_{H'} 
	\end{aligned}
	\]
	by Equation \eqref{eq:quadforlocalt} this last expression equals \(- Q_L(\ba)\), so \(c_L = Q_L(\ba)\).
\end{proof}

\begin{lemma}\label{lem:cLSubsetM}
	If \(L \Subset M\) and \(r_M \geq 2\) then 
	\[
	c_{L \Subset M} = \begin{cases}
	0 &\textup{ if } r_M = 2 \,; \\
	-2B(L, M) \cdot Q_M(\ba) &\textup{ if } r_M \geq 3 .
	\end{cases} 
	\]
\end{lemma}

\begin{proof}
	Equation \eqref{eq:Omega} implies
	\[
	c_{L \Subset M} =  - \left( \delta_{L \Subset M}(\txA) + \delta_{L \Subset M}(\txB) + \delta_{L \Subset M}(\txC) \right) \,.
	\]

    \textbf{Calculation of the A term.}
	If \(L' \in \mG^2\) then, by Corollary \ref{cor:proj}(ii) we have 
	\[
	\delta_{L \Subset M}(\gamma_{L'}^2) = 
	\begin{cases}
	2 B(L, M) &\text{ if } L' \supset M \\
	0 &\text{ otherwise. } 
	\end{cases} 
	\] 
	On the other hand, if \(L' \in \mG \setminus \{\emptyset\}\) has rank \(r_L \geq 3\) then \(\gamma_{L'}^2\) is basic so \(\delta_{L \Subset M}(\gamma_{L'}^2) = 0\).
	We conclude that
	\[
	\delta_{L \Subset M} (\txA) = 2 B(L, M) \sum_{L' \in \mG_M^2} a_{L'}^2 \,.
	\] 
	
	\textbf{Calculation of the B term.} If \(L' \Subset M'\) and \(r_{M'} \geq 2\) then the monomial \(\gamma_{L'}\gamma_{M'}\) is basic so 
	\[
	\delta_{L \Subset M}(\gamma_{L'}\gamma_{M'}) = 
	\begin{cases}
	1 &\text{ if } (L', M') = (L, M) \\
	0 &\text{ otherwise. }
	\end{cases}
	\] 
	Suppose that \(L', M' \in \mG \setminus \{\emptyset\}\) with \(L' \subsetneq M'\) is such that \(\gamma_{L'}\gamma_{M'}\) is non-basic. Then we are in one of the following cases.
	\begin{itemize}
		\item Case 1: \(r_{M'} \geq 2\) and \(L' \inclusiondot M'\). By Corollary \ref{cor:proj}(iii) we have
		\[
		\delta_{L \Subset M}(\gamma_{L'}\gamma_{M'}) = 
		\begin{cases}
		-1 &\text{ if } L \subsetneq L'\inclusiondot M' = M \\
		0 &\text{ otherwise. }
		\end{cases}
		\]
		\item Case 2: \(r_{M'} = 1\) so \(M' = H\) is a hyperplane \(H \in \mH\). We divide this into two further subcases.
		\begin{itemize}
			\item Subcase 2A: \(L' \in \mG^2\). Then by Corollary \ref{cor:proj}(v) we have
			\[
			\delta_{L \Subset M}(\gamma_{L'}\gamma_H) = 
			\begin{cases}
			-B(L, M) &\text{ if } L' = M \\
			-2B(L, M) &\text{ if } L' \supsetneq M \\
			0 &\text{ otherwise. }
			\end{cases}
			\] 
			
			\item Subcase 2B: \(L' \in \mG^{\geq 3} \setminus \{\emptyset\}\). Then by Corollary \ref{cor:proj}(iv) we have 
			\[
			\delta_{L \Subset M}(\gamma_{L'}\gamma_H) = 
			\begin{cases}
			B(L, M) &\text{ if } L' = M \\
			-1 &\text{ if } L' = L \text{ and } H \supset M \\
			1 &\text{ if } L \subsetneq L' \inclusiondot M \text{ and } H \supset M \\
			0 &\text{ otherwise. }
			\end{cases}
			\] 
		\end{itemize}
	\end{itemize}
	
	From the above, it follows that
	\[
	\begin{aligned}
	\delta_{L \Subset M}(\txB) &= \sum_{\emptyset \subsetneq L' \subsetneq M'} r_{M'} (r_{L'} -1) \cdot a_{L'} a_{M'} \cdot \delta_{L \Subset M}(\gamma_{L'} \gamma_{M'}) \\
	&=  \txB 1 + \txB 2 + \txB 3
	\end{aligned} 
	\]
	where the terms \(\txB 1\), \(\txB 2\), and \(\txB 3\) are given as follows. 
	
	\begin{itemize}
		\item 	The term \(\txB 1\) corresponds to the basic monomial \(\gamma_L\gamma_M\) and the sum of the non-basic monomials of Case 1, it given by
		\[
		\txB 1 =  r_M (r_L -1) a_L a_M \,-\,  r_M^2 a_M  \sum_{L' \,|\, L \subsetneq L' \inclusiondot M} a_{L'} \,,
		\]
        where in the last sum we used that \(r(L') = r(M) + 1\).
		\item The term \(\txB 2\) corresponds to the sum of the non-basic monomials of Subcase 2A. 
		If \(r_M \geq 3\) then
		\[
		\begin{aligned}
		\txB 2 &= - 2B(L, M)  \sum_{L' \in \mG_M^2} \sum_{H \in \mH_{L'}} a_{L'} a_H \\
		&= - 4 B(L, M) \sum_{L' \in \mG_M^2} a_{L'}^2 \,.
		\end{aligned}
		\]
		If \(r_M = 2\) then
		\[
		\begin{aligned}
		\txB 2 &= - B(L, M) a_M  \sum_{H \in \mH_M} a_H \\
		&= - 2 B(L, M) a_M^2 \,.
		\end{aligned}
		\]
		\item The term \(\txB 3\) corresponds to the sum of the non-basic monomials of Subcase 2B. 
		If \(r_M \geq 3\), then
		\[
		\begin{aligned}
		\txB 3 &=  B(L, M) (r_M-1) a_M \sum_{H \in \mH_M} a_H 
		\,-\, (r_L-1) a_L \sum_{H \in \mH_M} a_H  \,+\, 
		r_M  \sum_{L \subsetneq L' \inclusiondot M} \sum_{H \in \mH_M} a_{L'} a_H \\
		&= 2B(L, M) \binom{r_M}{2} a_M^2 \,-\, (r_L-1) r_M a_L a_M \,+\, r_M^2 a_M  \sum_{L \subsetneq L' \inclusiondot M}  a_{L'} \\
		&= 2B(L, M) \binom{r_M}{2} a_M^2 \,-\, \txB 1 \,.
		\end{aligned}
		\]
		If \(r_M = 2\), since \(r(L') \geq 3\) and so \(L' \neq M\), we have
		\[
		\begin{aligned}
		\txB 3 &= - (r_L-1) r_M a_L a_M \,+\, r_M^2 a_M  \sum_{L \subsetneq L' \inclusiondot M}  a_{L'} \\
		&= -\txB 1 \,.
		\end{aligned}
		\]
	\end{itemize}
	
	Adding the B1, B2, B3 terms we obtain that:
	\begin{itemize}
		\item if \(r_M \geq 3\) then 
		\[
		\delta_{L \Subset M}(\txB) = 2B(L, M) \binom{r_M}{2} a_M^2 - 4 B(L, M) \sum_{L' \in \mG_M^2} a_{L'}^2 \,;
		\]
		\item if \(r_M =2\) then
		\[
		\delta_{L \Subset M}(\txB) =
		- 2B(L, M) a_M^2  \,.
		\]
	\end{itemize}

	\textbf{Calculation of the C term.} The monomials \(\gamma_{L'} \gamma_{M'}\) with \(L' \pitchfork_r M'\) are basic if both \(r_{M'}\) and \(r_{L'}\) are \(\geq 2\), so we might assume that at least one of \(L'\) or \(M'\) is a hyperplane. 
	Using Corollary \ref{cor:proj} (vi) and (vii) to calculate the terms \(\delta_{L \Subset M} (\gamma_{L'}\gamma_{M'})\) when at least one of \(L'\) or \(M'\) is a hyperplane, we obtain:
	\[
	\begin{aligned}
	\delta_{L \Subset M}(\txC) &= \sum_{L' \pitchfork_r M'} r_{L'} r_{M'} \cdot a_{L'} a_{M'} \cdot \delta_{L \Subset M}(\gamma_{L'} \gamma_{M'}) \\
	&= - r_M a_M \sum_{H \in \mH_L \,|\, H \pitchfork_r M} a_H \,-\,
	2B(L, M) \sum_{H \pitchfork_r H' \in \mR_M^2} a_H a_{H'} \,+\,  \sum_{H \in \mH_L \,|\, H \pitchfork_r M} \sum_{H' \in \mH_M} a_H a_{H'} \\
	&= - r_M a_M \sum_{H \in \mH_L \,|\, H \pitchfork_r M} a_H \,-\,
	2B(L, M) \sum_{H \pitchfork_r H' \in \mR_M^2} a_H a_{H'} \,+\,  r_M a_M \sum_{H \in \mH_L \,|\, H \pitchfork_r M} a_H \\
	&= -2B(L, M) \sum_{H \pitchfork_r H' \in \mR_M^2} a_H a_{H'} \,.
	\end{aligned}
	\]
	
	Finally, we calculate \(- c_{L \Subset M} = \delta_{L \Subset M}(\txA) + \delta_{L \Subset M}(\txB) + \delta_{L \Subset M}(\txC)\). 
	\begin{itemize}
		\item If \(r_M = 2\) then \(\delta_{L \Subset M}(\txA) = - \delta_{L \Subset M}(\txB)\) and \(\delta_{L \Subset M}(\txC) = 0\), so \(c_{L \Subset M} = 0\).
		\item If \(r_M \geq 3\) then
		\[
		\begin{aligned}
		- c_{L \Subset M} &= 2 B(L, M) \sum_{L' \in \mG_M^2} a_{L'}^2 \\ 
		&+\, 2B(L, M) \binom{r_M}{2} a_M^2 - 4 B(L, M) \sum_{L' \in \mG_M^2} a_{L'}^2 \\
		&-\, 2B(L, M) \sum_{H \pitchfork_r H' \in \mR_M^2} a_H a_{H'} \\
		&= 2B(L, M) \cdot \left( \binom{r_M}{2} a_M^2 - \sum_{L' \in \mG_M^2} a_{L'}^2 - \sum_{H \pitchfork_r H' \in \mR_M^2} a_H a_{H'}
		\right) \,.
		\end{aligned}
		\]
		From Equation \eqref{eq:quadforlocalt} we deduce that \(- c_{L \Subset M} = 2B(L, M) Q_M(\ba)\) as wanted. \qedhere
	\end{itemize}
\end{proof}

\begin{lemma}\label{lem:cLpitchforkM}
	If \(L, M \in \mG^{\geq 2}\) and \(L \pitchfork_r M\) then \(c_{L\pitchfork_r M} = 0\).
\end{lemma}

\begin{remark}
	The condition that \(L \pitchfork_r M\) automatically implies that  \(L\) and \(M\) are both non-empty. Indeed, \(\mH\) is essential and irreducible, which means \(\emptyset \in \mG\); while the condition \(L \pitchfork_r M\) requires that \(L \cap M \notin \mG\). 
\end{remark}

\begin{proof}
	By Corollary \ref{cor:proj}(i)-(v), we have
	\[
	\delta_{L \pitchfork_r M} (\txA) = \delta_{L \pitchfork_r M}(\txB) = 0 \,.
	\]
	We deduce that
	\[
	\begin{aligned}
	-c_{L \pitchfork_r M} &=  \delta_{L \pitchfork_r M} (\txC) \\
	&= \sum_{L'\pitchfork_r M'} r_{L'}r_{M'} a_{L'}a_{M'} \cdot \delta_{L \pitchfork_r M}(\gamma_{L'}\gamma_{M'})
	\end{aligned}
	\]
	\begin{itemize}
		\item If \(r_{L'}, r_{M'} \geq 2\) then \(\gamma_{L'}\gamma_{M'}\) is basic so
		\[
		\delta_{L \pitchfork_r M}(\gamma_{L'}\gamma_{M'}) = \begin{cases}
		1 &\text{ if } \{L, M\} = \{L', M'\} \\
		0 &\text{ otherwise. }
		\end{cases}
		\]
		\item If \(M' = H\) is a hyperplane and \(r_{L'} \geq 2\) then by 
		Corollary \ref{cor:proj}(vi):
		\[
		\delta_{L \pitchfork_r M}(\gamma_{L'}\gamma_{H}) = \begin{cases}
		-1 &\text{ if } L' = L \text{ and } H \supset M \,,\\
		-1 &\text{ if } L' = M \text{ and } H \supset L \,,\\
		0 &\text{ otherwise. }
		\end{cases}
		\]
		\item If \(M' = H\) and \(L' = H'\) are both hyperplanes then by Corollary \ref{cor:proj} (vii):
		\[
		\delta_{L \pitchfork_r M}(\gamma_{H}\gamma_{H'}) = 
		\begin{cases}
		1 &\text{ if } H \supset L \text{ and } H' \supset M \,,\\
		1 &\text{ if } H \supset M \text{ and } H' \supset L \,,\\
		0 &\text{ otherwise. }
		\end{cases}
		\]
	\end{itemize} 

Splitting the sum over pairs \((L', M')\) into \(3\) parts according to the above items, we have
\[
\delta_{L \pitchfork_r M}(\txC) = \txC 1 + \txC 2 + \txC 3 \,,
\]
where
\[
\begin{aligned}
\txC 1 &= r_L r_M a_L a_M \\
\txC 2 &= - r_L a_L \sum_{H \in \mH_M} a_H \,-\, r_M a_M \sum_{H \in \mH_L} a_H = - 2 r_L r_M a_L a_M \\
\txC 3 &=  \sum_{H \supset L} \sum_{H' \supset M} a_H a_{H'} = r_L r_M a_L a_M 
\end{aligned}
\]
Therefore
\[
\delta_{L \pitchfork_r M}(\txC) = (1 - 2 + 1) r_L r_M a_L a_M = 0 
\]
and we conclude that \(c_{L\pitchfork_r M} = 0\).
\end{proof}

\subsection{Proof of (1) \(\implies\) (2)}\label{sec:pf1implies2}

\begin{lemma}\label{lem:kltmet}
    Let \(g\) be a PK metric on \(\CP^n\) whose singular set is a hyperplane arrangement \(\mH\) and whose cone angles along \(H \in \mH\) are \(2\pi\alpha_H\) with \(\alpha_H \in (0,1)\). 
	Then the weighted arrangement \((\mH, \ba)\) with weights \(a_H \in (0,1)\) given by \(a_H = 1-\alpha_H\) is klt.
\end{lemma}

\begin{proof}
    We prove that the function \(\prod_{H \in \mH} |h|^{-2a_H}\), where \(h\) are local defining equations of the hyperplanes \(H \in \mH\), is locally integrable with respect to ordinary Lebesgue measure.
    The equivalence of this with condition (ii) in Theorem \ref{thm:main} follows from \cite[Lemma 7.18]{pkc}.
    
	We need to show that for every \(x \in \CP^n\) there is an open ball \(x \in B\) and local defining equations \(h \in \mO(B)\) of the hyperplanes \(H \in \mH\) such that the function \(f: =\prod_{H \in \mH} |h|^{-2a_H}\) is locally integrable with respect to the standard Lebesgue measure on \(B\). (As usual, if \(x \notin H\) we take \(h=1\).) The local integrability of \(f\) follows from the (obvious) fact that the polyhedral metric \(g\) has finite volume, as we explain below. 
	
	By Proposition \ref{prop:lincoordmet} we can take linear coordinates \((z_1, \ldots, z_n)\) at \(x\).
	By Corollary \cite[7.22]{pkc} the volume form of \(g\) in these coordinates is equal (up to a constant factor) to \(f \cdot dV_{\C^n}\) where  \(dV_{\C^n}\) is the Euclidean volume form and \(f= \prod |h|^{-2a_H}\) with \(h\) linear  in \(z_1, \ldots, z_n\). Since \(g\) has finite volume, it follows that \(f\) is locally integrable, as we wanted to show.
\end{proof}

\begin{proof}[Proof of (1) implies (2) in Theorem \ref{thm:main}]
	Let \(g\) be a PK metric on \(\CP^n\) with cone angles \(2\pi\alpha_H\) along the hyperplanes \(H \in \mH\) with \(\alpha_H \in (0,1)\). We show that the weighted arrangement \((\mH, \ba)\) satisfies items (i), (ii), and (iii) of (2) in Theorem \ref{thm:main}.

	(i) Equations \eqref{eq:c1ohtsuki} and \eqref{eq:etadef} imply that the cohomology class \(\eta \in H^2(X, \R)\) vanishes. On the other hand, by Equation \eqref{eq:etabasis}, \(\eta=0\) if and only if \(\sum a_H = n+1\). 

    (ii) This is Lemma \ref{lem:kltmet}.
	
	(iii) By Equation \eqref{eq:ohtsuki2} and the definition of the cohomology class \(\Omega \in H^4(X, \R)\) given by Equations \eqref{eq:Omega} and \eqref{eq:ABC}, we have \(\Omega = 0\). On the other hand, since the basic monomials make a basis of \(H^4(x, \R)\), the cohomology class \(\Omega\) vanishes if and only if all the coefficients \(c_L, c_{L \Subset M}, c_{L \pitchfork_r M}\) in Equation \eqref{eq:Omegabasis} are zero. By Proposition \ref{prop:ohtsukibasis} the coefficients vanish if and only if \(Q_L(\ba) = 0\) for all \(L \in \mG\) with \(r_L \geq 3\). In particular, since \(\emptyset \in \mG\) (as the arrangement \(\mH\) is essential and irreducible), we have that \(Q(\ba) = Q_{\emptyset}(\ba) = 0\).
\end{proof}

\begin{remark}
    The proof above shows that a condition stronger than item (iii) in Theorem \ref{thm:main} holds. Namely, it shows that \(Q_L(\ba) = 0\) for \emph{any} proper irreducible subspace \(L \in \mG\).
\end{remark}

\section{Topological constraints \(\implies\) PK metric}\label{sec:2implies1}

In this section, we prove the remaining part of Theorem \ref{thm:main}. Specifically, we show that if \((\mH, \ba)\) is a weighted arrangement with weights \(a_H \in (0,1)\) satisfying conditions (i), (ii), and (iii) of (2) in Theorem \ref{thm:main}, then there exists a PK metric on \(\CP^n\) with cone angles \(2\pi(1 - a_H)\) along the members of \(\mH\).

Our proof is based on the Kobayashi--Hitchin correspondence for parabolic bundles due to Mochizuki~\cite{mochizuki}. In Section~\ref{sec:kobhit}, we state the precise result from Mochizuki that we require (Theorem~\ref{thm:moch}). Section~\ref{sec:parbun} recalls the parabolic structure on the pullback of \(T\CP^n\) via the log resolution \(\pi: X \to \CP^n\), as introduced in~\cite{miyaokayau}. Proposition~\ref{prop:pch2} computes the coefficients of the second parabolic Chern character \(\parch_2(\mE_*)\) when expressed in the basis of \(H^4(X, \R)\) given by basic monomials. In Section~\ref{sec:flatunitconn}, we apply Mochizuki’s theorem to deduce the existence of a flat unitary logarithmic connection compatible with the parabolic structure. Finally, in Section~\ref{sec:pfexistencePK}, we establish the existence of the desired PK metric by applying Theorem~\ref{thm:connmet} from Section \ref{sec:pkmetviaconn} to the constructed connection.

\subsection{Kobayashi-Hitchin correspondence for parabolic bundles}\label{sec:kobhit}
Let \(X\) be a compact complex manifold of dimension \(n\) and let \(D\) be a simple normal crossing divisor with irreducible decomposition \(D = \sum_{i \in I} D_i\). Let \(\mE_*\) be a parabolic bundle on \((X, D)\).
More precisely; \(\mE \to X\) is a holomorphic vector bundle and the parabolic structure is defined by an increasing filtration by locally free subsheaves \(\mE^i_a \subset \mE\) indexed by \(i \in I\) and \(a \in (0,1]\) as in \cite[Definition 4.1]{miyaokayau}.

For each irreducible component \(D_i\) we have an increasing filtration of \(\mE|_{D_i}\) by vector subbundles \(F^i_a \subset \mE|_{D_i}\) given by restricting the skyscraper sheaf \(\mE^i_a \,\big/ \, \mE(-D)\) to \(D_i\).
The parabolic first Chern class of \(\mE_{*}\) (see {\cite[\S 3.1.2]{mochizuki}}) is the element of \(H^2(X, \R)\) given by
\begin{equation}
	\parc_1(\mE_*) = c_1(\mE) - \sum_{i \in I} \sum_{0<a \leq 1} a \cdot \rk \big(F^i_a / F^i_{<a}\big)  \cdot  c_1(D_i) \,.
\end{equation}

Let \(L\) be an ample line bundle on \(X\). The parabolic degree of \(\mE_{*}\) is defined as usual \(\pardeg_L(\mE_{*}) := \parc_1(\mE_{*}) \cdot c_1(L)^{n-1}\). Assume that \(\pardeg_L(\mE_{*}) = 0\) then \(\mE_{*}\) is said to be \(\mu_L\)-stable if for every proper saturated subsheaf \(\mV \subset \mE\) we have \(\pardeg_L(\mV_*) < 0\), where \(\mV_*\) is the induced parabolic structure \cite[Definition 4.18]{miyaokayau}.	

The parabolic bundle \(\mE_*\) is \emph{locally abelian} (as defined by Iyer and Simpson \cite{simpson}) if it is locally the direct sum of parabolic line bundles, this is equivalent to the existence of local frames splitting all the filtrations. 
The locally abelian property is implicit in Mochizuki's definition of parabolic bundle in the form of compatible filtrations, see second item of \cite[Definition 3.12]{mochizuki}. Mochizuki's compatibility of filtrations is equivalent to the locally abelian condition, see \cite[Lemma 4.15]{miyaokayau}. 

Suppose that \(\mE_{*}\) is a parabolic bundle on \((X, D)\) that is locally abelian in codimension \(2\) as in \cite[\S 3.1.5]{mochizuki}. The parabolic second Chern character of \(\mE_{*}\)  is the element of \(H^4(X, \R)\) given by
\begin{equation}\label{eq:defpch2}
	\begin{aligned}
	\parch_2(\mE_{*}) &= \ch_2(\mE) - \sum_{i \in I} \sum_{0 < a \leq 1} a \cdot \imath_{*} \big(c_1(F^i_a / F^i_{<a})\big) \\
	&+ \,\, \frac{1}{2} \cdot \sum_{i \in I} \sum_{0 < a \leq 1} a^2 \cdot \rk (F^i_a / F^i_{<a}) \cdot c_1(D_i)^2 \\ 
	&+ \,\, \frac{1}{2} \cdot \sum_{\substack{i, j \in I\\ i \neq j}} \sum_{0 < a, b \leq 1} a b \cdot \rk (\Gr^{i, j}_{a, b}) \cdot c_1(D_i) \cdot c_1(D_j) \,,
	\end{aligned}
\end{equation}
where \(\imath_{*}: H^2(D_i, \Z) \to H^4(X, \Z)\) is the Gysin map of the inclusion \(D_i \subset X\),  and
\begin{equation}\label{eq:grij}
\Gr^{i, j}_{a, b} = \left( F^i_{a} \cap F^j_{b} \right) \, \big/ \, \left(F^i_{<a} \cap F^j_{b}  +  F^i_{a}  \cap F^j_{<b} \right) 
\end{equation}
is a skyscraper sheaf supported along \(D_i \cap D_j\). The locally abelian in codimension two condition implies that \(Gr^{i, j}_{a, b}\) is a vector bundle in the complement of a proper analytic subset of \(D_i \cap D_j\), and \(\rk Gr^{i, j}_{a, b}\) is the rank of this vector bundle.

The theory of parabolic bundles is naturally related to the theory of logarithmic connections through the notion of a logarithmic connection compatible with a parabolic structure, that we introduce next. In the case \(n=2\), the definition we adopt agrees with \cite[Definition 6.8]{panov}.
 
\begin{definition}
	Let \(\nabla\) be a unitary flat logarithmic connection on \(\mE \to X\) with poles along \(D\).
	We say that \(\nabla\) is \emph{compatible} with the parabolic structure \(\mE_*\) if for every the following holds:
	\begin{enumerate}[label=\textup{(\roman*)}]
		\item for every \(i \in I\) and \(a \in (0,1]\) the subbundle \(F^i_a \subset \mE|_{D_i}\) is preserved by \(\Res_{D_i}(\nabla)\)\,;
		\item \(\Res_{D_i}(\nabla)\) acts on \(F^i_a \big/ F^i_{< a}\) by scalar multiplication by \(a\)\,;
		\item for every \(i \in I\) and \(a \in (0,1]\) the holomorphic connection \(\nabla|_{X \setminus D}\) extends across \(D\)  as a logarithmic connection on \(\mE^i_a\).
	\end{enumerate}
\end{definition}

The main result that we need is the following case of the Kobayashi-Hitchin correspondence for parabolic bundles.

%It easily follows from the commutativity of the residues (Lemma \ref{lem:ressnc}) that if \(\nabla\) is compatible with \(\mE_*\) then \(\mE_*\) is locally abelian.

\begin{theorem}\label{thm:moch}
	Let \(\mE_*\) be a locally abelian parabolic bundle on \((X, D)\). Suppose that: \textup{(i)} \(\mE_*\) is \(\mu_L\)-stable; \textup{(ii)} both \(\parc_1(\mE_*) = 0\) and \(\parch_2(\mE_*) = 0\).	
	Then there is a unitary flat logarithmic connection \(\nabla\)  compatible with \(\mE_*\).
\end{theorem}

The above result can be deduced from \cite[Theorem 9.4]{mochizuki} which shows the existence of an adapted flat Hermitian metric adapted to \(\mE_{*}\) together with \cite{mochizukiasymptotic} which shows that the associated metric connection extends with logarithmic poles along \(D\) and is adapted to \(\mE_{*}\).

\subsection{Parabolic bundle and parabolic second Chern character}\label{sec:parbun}

Let \((\mH, \ba)\) be a weighted arrangement consisting of a finite collection of hyperplanes \(H \subset \CP^n\) together with weights \(a_H \in (0,1)\). Throughout this section, we assume that \((\mH, \ba)\) satisfies (i), (ii), and (iii) in Theorem \ref{thm:main}. The main result of this section is Proposition \ref{prop:pch2}.

\begin{definition}\label{def:aL}
	For \(L \in \mG\) we let
	\begin{equation}\label{eq:aL}
		a_L = r(L)^{-1} \cdot \left( \sum_{H \supset L} a_H \right) \,.
	\end{equation}
\end{definition}

Let \(\pi: X \to \CP^n\) be the minimal De Concini-Procesi model of \(\mH\). Let \(D = \pi^{-1}(\mH)\) and let \(\mE = \pi^*T\CP^n\). We consider the parabolic bundle \(\mE_{*}\) on \((X, D)\) introduced in \cite[Section 4.2]{miyaokayau} that is defined by the increasing filtrations of \(\mE|_{D_L}\) given by the vector subbundles
\begin{equation}
F^L_a = \begin{cases}
\pi^*(TL) &\text{ if } a < a_L, \\
\mE|_{D_L} &\text{ if } a \geq a_L .
\end{cases}
\end{equation}

The parabolic second Chern character \(\parch_2(\mE_{*}) \in H^4(X, \R)\) is given by (see \cite[Corollary 4.37]{miyaokayau})
\begin{equation}\label{eq:pch2}
	\parch_2 (\mE_{*}) = - \frac{n+1}{2} \cdot \gamma_{\emptyset}^2 
	+  \frac{1}{2} \cdot \sum_{\emptyset \subsetneq L} r(L) \cdot a_L^2 \cdot \gamma_L^2 + \sum_{\emptyset \subsetneq L \subsetneq M} r(M) \cdot a_L  a_{M} \cdot  \gamma_L  \gamma_M \,.
\end{equation}

Our goal is to write the components of \(\parch_2(\mE_{*})\) in the basis \(\Delta_2\) of \(H^4(X, \R)\). We begin by noticing that \(\parch_2(\mE_{*})\) belongs to a certain linear subspace.

\begin{definition}\label{def:subspaceH4}
	Let \(V \subset H^4(X, \R)\) be the linear subspace spanned by the basic monomials in Lemma \ref{lem:basisH4} of types (i) and (ii). Specifically,
	\[
	V = \spn_{\R} \{\gamma_{L}^2 \,|\, r(L) \geq 3\} \oplus \spn_{\R} \{\gamma_{L_1} \cdot \gamma_{L_2} \,|\, L_1 \Subset L_2 \text{ and } r(L_2) \geq 2 \} \,.
	\]
	Equivalently, \(m \in V\) if \(\delta_{L_1 \pitchfork_r L_2}(m) = 0\) for all \(L_1, L_2 \in \mG^{\geq 2}\).
\end{definition}

\begin{lemma}\label{lem:subspace}
	\(\parch_2(\mE_{*}) \in V\)\,.
\end{lemma}

\begin{proof}
	This follows by inspection of the right hand side of Equation \eqref{eq:pch2}. The terms which are basic monomials are clearly in \(V\).  The non-basic monomials \(\gamma_{L}^2\) for \(r(L) \leq 2\) and \(\gamma_{L}\cdot\gamma_M\) for \(L \inclusiondot M\) or \(r(M)=1\) lie in the kernel of \(\delta_{L_1 \pitchfork_r L_2}\) for any \(L_1, L_2 \in \mG^{\geq 2}\) by Corollary \ref{cor:proj} items (i) to (v); and therefore lie in \(V\).
\end{proof}

By Lemma \ref{lem:subspace} we can write
\begin{equation}\label{eq:pch2components}
	\parch_2(\mE_{*}) = \sum_{r(L) \geq 3} c_L \cdot \gamma_{L}^2 + \sum_{L_1 \Subset L_2\,,\, r(L_2) \geq 2} c_{L_1 \Subset L_2} \cdot \gamma_{L_1}  \gamma_{L_2}
\end{equation}
for some \(c_L, c_{L_1 \Subset L_2} \in \R\). To calculate the coefficients \(c_L\) and \(c_{L_1 \Subset L_2}\) we evaluate the dual basis elements \(\delta_{L}\) and \(\delta_{L_1 \Subset L_2}\) of Section \ref{sec:cordproj} on \(\parch_2(\mE_{*})\). The coefficients \(c_L\) and \(c_{L_1 \Subset L_2}\) of \(\parch_2(\mE_{*})\) in Equation \eqref{eq:pch2components} are then given by
\[
c_L = \delta_L(\parch_2(\mE_{*})) \,\,\text{ and }\,\, c_{L_1 \Subset L_2} = \delta_{L_1 \Subset L_2}(\parch_2(\mE_{*})) \,.
\]
The main calculation is the following.

\begin{proposition}\label{prop:pch2}
	The coefficients \(c_L\) and \(c_{L_1 \Subset L_2}\)  of \(\parch_2(\mE_{*})\) in Equation \eqref{eq:pch2components} are given as follows:
	\begin{enumerate}[label=\textup{(\arabic*)}]
		\item if \(r(L) \geq 3\) then \(c_L = Q_L\) ;
		
		\item if \(L_1 \Subset L_2\) and \(r(L_2) \geq 3\) then \(c_{L_1 \Subset L_2} = 2B(L_1, L_2) \cdot Q_{L_2}\) ;
		
		\item if \(L_1 \Subset L_2\) and \(r(L_2) = 2\) then \(c_{L_1 \Subset L_2} = 0\).
	\end{enumerate}
\end{proposition}

\subsubsection{Proof of Proposition \ref{prop:pch2}}
Rewrite Equation \eqref{eq:pch2} as
\[
\parch_2 (\mE_{*}) = - \frac{n+1}{2} \cdot \gamma_{\emptyset}^2 
\,+\, \alpha  \,+\, \beta  
\]
with
\[
\alpha = \frac{1}{2} \cdot \sum_{\emptyset \subsetneq L'} r(L') \cdot a_{L'}^2 \cdot \gamma_{L'}^2 
\]
and
\[
\beta = \sum_{\emptyset \subsetneq L'' \subsetneq L'} r(L') \cdot a_{L''}  a_{L'} \cdot  \gamma_{L''}  \gamma_{L'} \,.
\]

\begin{lemma}\label{lem:cLempty}
	The coefficient \(c_{\emptyset}\) of \(\parch_2(\mE_*)\) is equal to
	\(\delta_{\emptyset}(\parch_2(\mE_{*})) = Q\).
\end{lemma}

\begin{proof}
	Write
	\[
	\delta_{\emptyset}(\parch_2(\mE_{*})) = \textup{A} + \textup{B}
	\]
	with
	\[
	\textup{A} = -\frac{n+1}{2} + \delta_{\emptyset}(\alpha) \,\,\textup{ and }\,\, \txB = \delta_{\emptyset}(\beta) \,.
	\]
	We need to evaluate \(\delta_{\emptyset}\) only on non-basic monomials.
	Corollary \ref{cor:proj} (i) and (ii) implies
	\[
	\delta_{\emptyset}(\alpha) = - \frac{1}{2} \cdot \sum_{H \in \mH} B_H \cdot a_{H}^2  - \sum_{L \in \mG^2} a_L^2 \,.
	\]
	Corollary \ref{cor:proj} (iii), (iv) and (v) implies
	\[
	\textup{B} =  \sum_{L \in \mG^2} \sum_{H \supset L} a_{L} \cdot a_H 
	=  2 \cdot \sum_{L \in \mG^2} a_{L}^2 \,.
	\]
	Therefore,
	\[
	\textup{A} + \textup{B} =
	-\frac{n+1}{2} - \frac{1}{2} \sum_{H \in \mH} B_H \cdot a_{H}^2  + \sum_{L \in \mG^2} a_{L}^2 \,,
	\]
	which equals \(Q\).
\end{proof}

\begin{lemma}\label{lem:cL}
	If \(L\) is non-empty and \(r(L) \geq 3\) then \(c_L\) is equal to
	\begin{equation}
		\delta_L(\parch_2(\mE_{*})) = Q_L \,.	
	\end{equation}
\end{lemma}

\begin{proof}
	Write \(\delta_L(\parch_2(\mE_{*})) = \textup{A} + \textup{B}\) with \(\txA = \delta_L(\alpha)\) and \(\txB = \delta_L(\beta)\). Items (i) and (ii) of Corollary \ref{cor:proj} give us
	\[
	\textup{A} = \frac{r(L)}{2} \cdot a_L^2 - \frac{1}{2} \sum_{H \in \mH_L} B(L, H) \cdot a_{H}^2  - \sum_{L' \in \mG_L^2} a_{L'}^2 \,.
	\]
	Items (iii), (iv), and (v) of Corollary \ref{cor:proj} give us
	\[
	\begin{aligned}
		\textup{B} &= - \sum_{H \in \mH_L} a_L \cdot a_{H}  + \sum_{L' \in \mG_L^2} \sum_{H \supset L'} a_{L'} \cdot a_H \\
		&= - r(L) \cdot a_L^2 + 2 \cdot \sum_{L' \in \mG_L^2} a_{L'}^2 \,.
	\end{aligned}
	\]
	Therefore
	\[
	\textup{A} + \textup{B} 
	= -\frac{r(L)}{2} \cdot a_L^2 - \frac{1}{2} \sum_{H \in \mH_L} B(L, H) \cdot a_{H}^2  + \sum_{L' \in \mG_L^2} a_{L'}^2 
	\]
	which equals \(Q_L\).
\end{proof}

\begin{lemma}\label{lem:cL1L23}
	If \(L_1 \Subset L_2\) and  \(r(L_2) \geq 3\)
	then
	\begin{equation}
		\delta_{L_1 \Subset L_2}(\parch_2(\mE_{*})) = 2 B(L_1, L_2) \cdot Q_{L_2} \,.	
	\end{equation}
\end{lemma}

\begin{proof}
	Write \(\delta_{L_1 \Subset L_2}(\parch_2(\mE_{*})) = \textup{A} + \textup{B}\) with \(\txA = \delta_{L_1 \Subset L_2}(\alpha)\) and \(\txB = \delta_{L_1 \Subset L_2}(\beta)\). Items (i) and (ii) of Corollary \ref{cor:proj} gives us
	\[
	\textup{A} = -B(L_1, L_2) \sum_{H \supset L_2} B(L_2, H) \cdot a_{H}^2  + 2B(L_1, L_2) \cdot \sum_{L \in \mG_{L_2}^2} a_{L}^2 \,.
	\]
	We split the term 
	\[
	\txB = 
	\sum_{\emptyset \subsetneq L'' \subsetneq L'} r(L') \cdot a_{L''}  a_{L'} \cdot  \delta_{L_1 \Subset L_2}(\gamma_{L''}  \gamma_{L'})
	\]
	as a sum of four terms given as follows.
	\begin{itemize}
		\item If \((L'', L') = (L_1, L_2)\) then 
		\[
		\txB1 = 	(\codim L_2) \cdot a_{L_1} a_{L_2} \,.
		\]
		(If \(L_1 = \emptyset\) then \(\txB1 = 0\).)
		
		\item Sum over \(L'' \inclusiondot L'\). By Corollary \ref{cor:proj}(iii) we can assume that \(L_1 \subset L'' \inclusiondot L' = L_2\) and
		\[
		\txB2 =  - r(L_2) \cdot a_{L_2} \cdot \sum_{L_1 \subset L'' \inclusiondot L_2} a_{L''}  \,.
		\]
		
		\item Sum over \((L'', L')\) with \(r(L'') \geq 3\) and \(L' \in \mH\).
		By Corollary \ref{cor:proj}(iv) we have
		\[
		\begin{aligned}
			\txB3 &= B(L_1, L_2)  \sum_{H \supset L_2} a_{L_2} a_H  
			- \sum_{H \supset L_2} a_{L_1} a_H 
			+ \sum_{L_1 \subset L \inclusiondot L_2} \sum_{H \supset L_2} a_L a_H \\
			&=  B(L_1, L_2) \cdot r(L_2) \cdot a_{L_2}^2 - \txB1 - \txB2
		\end{aligned}
		\]
		(If \(L_1 = \emptyset\) then there is no middle term.)
        
		\item Sum over \((L'', L')\) with \(r(L'') = 2\) and \(L' \in \mH\).
		By Corollary \ref{cor:proj}(v) we have
		\[
		\txB4 = - 2 B(L_1, L_2) \cdot \sum_{L \in \mG_{L_2}^2} \sum_{H \supset L} a_{L} a_H =  - 4 B(L_1, L_2) \cdot \sum_{L \in \mG_{L_2}^2} a_{L}^2 \,.
		\]
        Here, we used that only option in Corollary \ref{cor:proj}(v) for which \(\delta_{L_1 \Subset L_2}(\gamma_{L''}\gamma_{L'})\) is non-zero is the middle one, i.e., when \(L_2 \subsetneq L''\), as \(r(L_2) \geq 3\). This is the only place where the assumption \(r(L_2)\geq 3\) is used.
	\end{itemize} 
	
	Adding \(\txB1 + \ldots + \txB4\) we get
	\[
	\txB = B(L_1, L_2) \cdot (\codim L_2) \cdot a_{L_2}^2 
	- 4 B(L_1, L_2) \cdot \sum_{L \in \mG_{L_2}^2} a_{L}^2 \,.
	\]
	Finally
	\[
	\begin{aligned}
		\textup{A} + \textup{B} 
		&= -B(L_1, L_2) \sum_{H \supset L_2} B(L_2, H) \cdot a_{H}^2  - 2B(L_1, L_2) \cdot \sum_{L \in \mG_{L_2}^2} a_{L}^2 \\
		&+ B(L_1, L_2) \cdot (\codim L_2) \cdot a_{L_2}^2 = 2B(L_1, L_2) \cdot Q(\mH_{L_2})
	\end{aligned}
	\]
	and this finishes the proof of the lemma.
\end{proof}

\begin{lemma}\label{lem:cL1L22}
	If \(L_1 \Subset L_2\) and  \(r(L_2) = 2\) then the coefficient \(c_{L_1 \Subset L_2}\) vanishes, i.e., 
	\begin{equation}
		\delta_{L_1 \Subset L_2}(\parch_2(\mE_{*})) = 0 \,.	
	\end{equation}
\end{lemma}

\begin{proof}
	Write \(\delta_{L_1 \Subset L_2}(\parch_2(\mE_{*})) = \textup{A} + \textup{B}\)
	with
	\[
	\textup{A} = 2B(L_1, L_2) \cdot a_{L_2}^2
	\]
    (note that \(\delta_{L_1 \Subset L_2}(\gamma_{H}^2) = 0\) because \(B(L_2, H) = 0\) as \(r(L_2) = 2\))
    and \(\txB = \delta_{L_1 \Subset L_2}(\beta)\).
    Same as the proof of Lemma \ref{lem:cL1L23}, we split the B term into four \(\txB = \txB1 + \ldots + \txB4\) given as follows.
	\begin{itemize}
	\item If \((L'', L') = (L_1, L_2)\) then \(\txB1 = 	r(L_2) \cdot a_{L_1} a_{L_2}\)
	(if \(L_1 = \emptyset\) then \(\txB1 = 0\));
		
	\item sum over \(L'' \inclusiondot L'\), by Corollary \ref{cor:proj}(iii) we have
    \[
	\txB2 =  - r(L_2) \cdot a_{L_2} \cdot \sum_{L_1\subset L \inclusiondot L_2} a_L  \,;
	\]
			
	\item sum over \((L'', L')\) with \(r(L'') \geq 3\) and \(L' \in \mH\), by Corollary \ref{cor:proj}(iv) we have
    \[
	\txB3 =  - \sum_{H \supset L_2} a_{L_1} a_H 
		+ \sum_{L_1 \subset L \inclusiondot L_2} \sum_{H \supset L_2} a_L a_H \\
		=  - \txB1 - \txB2
	\]
	(if \(L_1 = \emptyset\) then there is no \(\txB1\) term);
        
	\item Sum over \((L'', L')\) with \(r(L'') = 2\) and \(L' \in \mH\), by Corollary \ref{cor:proj}(v) we have
    \[
	\txB4 = - B(L_1, L_2) \cdot \sum_{H \supset L_2} a_{L_2} a_H =  - 2 B(L_1, L_2) \cdot a_{L_2}^2 \,.
	\]
	\end{itemize} 
	
	Adding \(\txB1 + \ldots + \txB4\) we get
	\[
	\txB = 
	- 2 B(L_1, L_2) \cdot a_{L_2}^2 \,.
	\]
	Finally \(\delta_{L_1 \Subset L_2}(\parch_2(\mE_{*}))  = \txA + \txB = 0\).
\end{proof}

\begin{proof}[Proof of Proposition \ref{prop:pch2}]
	Item (i) follows from Lemmas \ref{lem:cLempty} and \ref{lem:cL}; item (ii) follows from Lemma \ref{lem:cL1L23}; and item (iii) follows from Lemma \ref{lem:cL1L22}.
\end{proof}

\subsection{The flat unitary connection}\label{sec:flatunitconn}

We begin by proving the following striking result. 

\begin{lemma}\label{lem:QL}
	Let  \((\mH, \ba)\) be a  klt and CY weighted arrangement. Moreover, suppose that \(Q(\ba) = 0\). Then for every proper irreducible subspace \(L \in \mG\), we have \(Q_L(\ba) = 0\).
\end{lemma}

\begin{proof}
	The klt and CY assumptions imply that \((\mH, \ba)\) is \emph{stable} \cite[\S 6.2]{miyaokayau}. Since \(Q(\ba) = 0\), by \cite[Theorem 3.1]{dunkl} there is a Dunkl metric \(\langle \cdot, \cdot \rangle\) on \(\C^{n+1}\) adapted to \((\mH, \ba)\) - here, to be consistent with \cite{dunkl}, we regard \(\mH\) as an arrangement of complex linear hyperplanes in \(\C^{n+1}\). Let \(L \in \mG\) be an irreducible subspace. Let \( \inn \vert_{L^{\perp}}\) be the restriction of \(\inn\) to the orthogonal complement of \(L \subset \C^{n+1}\) with respect to \(\inn\). It is easy to show (see last paragraph in page 96 of \cite{chl}) that \(\inn_{L^{\perp}}\) is a Dunkl metric on \(L^{\perp}\) adapted to the weighted arrangement \(\big(\mH_L/L, \pi_L(\ba)\big)\), where \(\mH_L / L\) is the (essential and irreducible) arrangement \(\mH_L / L = \{H \cap L^{\perp} \,|\, H \in \mH_L\}\) and \(\pi_L(\ba)\) are the obvious weights induced from the inclusion \(\mH_L \subset \mH\) as in Remark \ref{rmk:restrweights}. Then it follows from \cite[Theorem 3.4]{dunkl} that \(Q_L(\ba) = 0\).
\end{proof}

Let \((H, \ba)\) be a weighted arrangement in \(\CP^n\) with \(\ba \in (0,1)^{\mH}\) satisfying conditions (i), (ii), and (iii) in Theorem \ref{thm:main}. Let \(\mE_{*}\) be the parabolic bundle in Section \ref{sec:parbun}.

\begin{lemma}\label{lem:pch2is0}
    \(\parch_2(\mE_*) = 0\).
\end{lemma}

\begin{proof}
    By Equation \eqref{eq:pch2components}, the parabolic second Chern character \(\parch_2(\mE_*)\) vanishes if and only if the all of the coefficients \(c_L\) for \(L \in \mG^{\geq 3}\) and \(c_{L_1 \Subset L_2}\) for \(L_1, L_2 \in \mG\) with \(r(L_2) \geq 2\) vanish. By Proposition \ref{prop:pch2} (3), \(c_{L_1 \Subset L_2} = 0\) if \(r(L_2) = 2\). On the other hand, if \(r(L_2)\geq 3\) then
    by Proposition \ref{prop:pch2} (2) the coefficient \(c_{L_1 \Subset L_2}\) is proportional to \(Q_{L_2}(\ba)\) which vanishes Lemma \ref{lem:QL}. Finally, if \(r(L) \geq 3\) then by Proposition \ref{prop:pch2} (1) we have \(c_L = Q_L(\ba)\) and again \(Q_L(\ba) = 0\) by Lemma \ref{lem:QL}.
\end{proof}

\begin{theorem}\label{thm:flatunitary}
	There exists a flat unitary logarithmic connection \(\tn\) on \(\mE \to X\) compatible with the parabolic bundle \(\mE_{*}\).
\end{theorem}

\begin{proof}
	By \cite[Theorem 4.29]{miyaokayau} the parabolic bundle \(\mE_{*}\) is locally abelian.
	The Calabi-Yau condition (i) together with the klt condition (ii) imply that \(\parc_1(\mE_{*}) = 0\) (\cite[Lema 4.32]{miyaokayau}) and that \(\mE_{*}\) is \(\mu_L\)-stable (\cite[Theorem 5.1]{miyaokayau}).
    By Lemma \ref{lem:pch2is0}, \(\parch_2(\mE_{*}) = 0\). The result then follows by Theorem \ref{thm:moch}.
\end{proof}

\begin{lemma}\label{lem:nabla} Let \(\tn\) be the logarithmic connection constructed on \(\mE \to X\) in Theorem \ref{thm:flatunitary}.
	There is a (unique) logarithmic connection \(\nabla\) on \(T\CP^n\) with poles along \(\mH\) such that \(\tn = \pi^*\nabla\). Moreover, \(\nabla\) is adapted to the weighted arrangement \(\{(H, a_H)\}\).
\end{lemma}

\begin{proof}
	The facts that \(\nabla\) exists, unique, and logarithmic follow from \(\pi: X \to \CP^n\) having an inverse defined outside a set of codimension \(\geq 2\) in \(\CP^n\). We only need to show that \(\nabla\) is adapted to \(\{(H, a_H)\}\), i.e., we verify the three items in Definition \ref{def:adaptconn}.

	Items (i) and (ii). Let \(H \in \mH\) and \(\tH \subset X\) be its proper transform. Then \(\Res_H(\nabla)\) is holomorphic on the generic locus \(H^{\gen}:= H \setminus \bigcup_{H' \in \mH \setminus \{H\}} H'\) where \(\pi\) has an inverse and \(\Res_{\tH}(\tn) = \pi^*(\Res_H(\nabla))\) along \(\pi^{-1}(H^{\gen})\).
    Since \(\tn\) is compatible with the parabolic bundle \(\mE_{*}\), we have \(\ker (\Res_{\tH}(\tn)) = \pi^*TH\) and \(\tr(\Res_{\tH}(\tn)) = a_H\).  On the other hand, if \(x \in H^{\gen}\) then \(\Res_H(\nabla)(x) = \Res_{\tH}(\tn)(\tx)\) where \(\tx = \pi^{-1}(x)\). It follows that \(\ker (\Res_{H}(\nabla)) = TH\) and \(\tr(\Res_{H}(\nabla)) = a_H\) along \(H^{\gen}\).
	
	Item (iii). We are left to show that \(\nabla\) has holomorphic residues. This follows from Lemma \ref{prop:pullback1form2} in the appendix after taking local trivializations.
\end{proof}

\subsubsection{Digression: torsion}
Let \(M\) be a complex manifold and let \(\nabla\) be a logarithmic connection on \(TM\) with poles along \(D\). 
The \emph{torsion} of \(\nabla\) is the holomorphic section of \((\Lambda^2T^*M \otimes TM)|_{M \setminus D}\) given by
\[
T(u, v) = \nabla_u v - \nabla_v u - [u,v] \,.
\]	

\begin{lemma}\label{lem:torsion}
	Let \(M\) be a complex manifold and let \(\nabla\) be a logarithmic connection on \(TM\) with poles along a divisor \(D \subset X\). Then the torsion of \(\nabla\) extends holomorphically across \(D\) to the whole \(M\) as an element 
	\[
	T \in H^0(\Lambda^2T^*M \otimes TM) 
	\] 
	if and only if \(TD \subset \ker \Res_D(\nabla)\) along the smooth locus \(D^{\circ}\) of \(D\).
\end{lemma}

\begin{proof}
	Suppose that \(TD \subset \ker \Res_D(\nabla)\) along \(D^{\circ}\). We want to show that \(T\) extends holomorphically across \(D\) to the whole \(M\). 
	Let \(p \in D^{\circ}\) and let \((z_1, \ldots, z_n)\) be local complex coordinates centred at \(p\) such that \(D=\{z_1=0\}\) and \(\nabla = d - \Omega\) where
	\[
	\Omega = A_1 \frac{dz_1}{z_1} + A_2 dz_2 + \ldots + A_n dz_n 
	\]
	with \(A_1, \ldots, A_n\) holomorphic sections of \(\End TM\). Writing \(\p_i = \p / \p z_i\) for the coordinate vector fields, we have that
	\begin{equation}\label{eq:Tij}
	T(\p_i, \p_j) = \begin{cases}
	A_j(\p_i) - A_i(\p_j) &\text{ if } i, j > 1 \,, \\
	A_j(\p_1) - \frac{1}{z_1} \cdot A_1(\p_j) &\text{ if } i=1 \text{ and } j>1 \,.
	\end{cases}		
	\end{equation}
	
	The assumption that \(TD \subset \ker \Res_D(\nabla)\) implies that the vector fields \(A_1(\p_j)\) for \(j > 1\) vanish along \(\{z_1 = 0\}\) and therefore the terms \((1/z_1)A_1(\p_j)\) are holomorphic. It follows from this and Equation \eqref{eq:Tij} that for all \(i, j \in \{1, \ldots, n\}\) the vector fields \(T(\p_i, \p_j)\) are holomorphic in a neighbourhood of \(p\). This shows that \(T\) extends across \(D^{\circ}\) to \(M \setminus \Sing(D)\) as a holomorphic section of \(\Lambda^2T^*M \otimes TM\). By Hartogs, since \(\Sing(D)\) is an analytic subset of \(M\) of codimension \(\geq2\), the section \(T\) extends holomorphically to the whole \(M\).
	
	Conversely, if \(T\) is holomorphic, then it follows from Equation \eqref{eq:Tij} that \(A_1(\p_j)\) must vanish along \(z_1=0\) for all \(j>1\), therefore \(TD \subset \ker \Res_D(\nabla)\).
\end{proof}	

\begin{corollary}\label{cor:torfree}
	Let \(M\) be a complex manifold and let \(\nabla\) be a logarithmic connection on \(TM\) with poles along \(D \subset M\). If \(TD \subset \ker \Res_{D}(\nabla)\) along \(D^{\circ}\) and \(H^0(\Omega^2 \otimes TM) = 0\), then \(\nabla\) is torsion free.
\end{corollary}

\subsubsection{Holomorphic \(2\)-forms on \(\CP^n\) with values on \(T\CP^n\)}

\begin{lemma}\label{lem:nonscalar}
	Let \(V\) be a complex vector space of dimension \(n\). Let \(\Omega\) be non-zero two-form with values in \(V\). Then there is a vector \(v \in V\) such that the endomorphism \(\imath_v \Omega \in \End(V)\) defined by \(x \mapsto \Omega(v, x)\) is a non-scalar operator.  
\end{lemma}

\begin{proof}
	Let \(e_1, \ldots, e_n\) be a basis of \(V\) and write \(\Omega = \sum \omega_i \otimes e_i\)
	with \(\omega_i \in \Lambda^2 V^*\). Since \(\Omega\) is non-zero, we can choose \(i\) such that the two-form \(\omega_i\) is non-zero. Clearly, there is \(v \in V\) such that \(\omega_i(v, \cdot)\) is not proportional to \(e_i^*\). Then \(\imath_v \Omega\) is a non-scalar operator. 
\end{proof}

\begin{lemma}\label{lem:nohol2forms}
	There are no non-zero holomorphic \(2\)-forms on \(\CP^n\) with values on \(T\CP^n\).
\end{lemma}

\begin{proof}
	By contradiction, suppose that \(\Omega\) is a non-zero holomorphic \(2\)-form on \(\CP^n\) with values on \(T\CP^n\). Choose a point \(p \in \CP^n\) and a holomorphic vector field \(v\) such that \(\imath_v \Omega\) is non-scalar on \(T\CP^n\) - as guaranteed by Lemma \ref{lem:nonscalar}.
	This contradicts the fact that \(T\CP^n\) is simple, i.e., every element of \(H^0 (\End T\CP^n)\) is a scalar multiple of the identity -see \cite[Lemma 4.1.2]{okonek}.
\end{proof}

\begin{corollary}\label{cor:tf}
	The connection \(\nabla\) on \(T\CP^n\) of Lemma \ref{lem:nabla} is torsion-free.
\end{corollary}

\begin{proof}
	This follows from the fact that \(TH = \ker (\Res_H(\nabla))\) for all \(H \in \mH\) together with Corollary \ref{cor:torfree} and Lemma \ref{lem:nohol2forms}.
\end{proof}

\subsection{Proof of (2) \(\implies\) (1)}\label{sec:pfexistencePK}

\begin{proof}[Proof of (2) implies (1) in Theorem \ref{thm:main}]
	Let \((\mH, \ba)\) be a weighted arrangement in \(\CP^n\) with \(\ba \in (0,1)^{\mH}\) satisfying the conditions (i), (ii), and (iii) of Theorem \ref{thm:main}. 
	
	Let \(\nabla\) the logarithmic connection on \(T\CP^n\) with poles along \(\mH\) given by Lemma \ref{lem:nabla}. Then \(\nabla\) satisfies the following properties:
	\begin{enumerate}[label=\textup{(\roman*)}]
		\item \(\nabla\) is flat and unitary because \(\tn = \pi^*\nabla\) in Theorem \ref{thm:flatunitary} is flat and unitary;
		\item \(\nabla\) is torsion-free by Corollary \ref{cor:tf};
		\item \(\nabla\) is adapted to \(\{(H, a_H)\}\) by Lemma \ref{lem:nabla}.
	\end{enumerate}
	
	By Theorem \ref{thm:connmet}, the Hermitian metric \(g\) on the arrangement complement preserved by \(\nabla\) defines a PK metric on \(\CP^n\) with cone angles \(2\pi(1-a_H)\) along the hyperplanes \(H \in \mH\).
\end{proof}

\appendix

\section{Appendix}
\label{appendix}

In Section~\ref{sec:log1forms}, we recall standard definitions of logarithmic \(1\)-forms, as presented for example in~\cite{saito}. In Section~\ref{sec:holres}, we introduce the condition that a logarithmic \(1\)-form has \emph{holomorphic residues}. In Section~\ref{sec:pullback}, we relate the holomorphic residues condition to the property that the pullback by a log resolution has logarithmic poles. Finally, in Section~\ref{app:conn}, we review standard material on connections, included here for completeness and to fix conventions.

\subsection{Logarithmic \(1\)-forms}\label{sec:log1forms}
Let \(X\) be a complex manifold.
We denote by \(\mO\) the sheaf of holomorphic functions on \(X\), and by \(\Omega^p\) the sheaf of holomorphic  \(p\)-forms.
 
\begin{definition}\label{def:log1form}
	Let \(D \subset X\) be a divisor and
	let \(\omega \in \Omega^1(X \setminus D)\) be a holomorphic \(1\)-form on \(X \setminus D\).
	We say that \(\omega\) is a logarithmic \(1\)-form on \(X\) with poles along \(D\) if for every \(x \in D^{\circ}\) (the smooth locus of \(D\)) there is an open set \(U \subset X\) that contains \(x\) and a defining equation \(h \in \mO(U)\) of \(D\) such that
	\begin{equation}\label{eq:log1form}
	\omega = a \frac{dh}{h} \, + \, \eta \,\,
	\end{equation}
	for some \(a \in \mO(U)\) and \(\eta \in \Omega^1(U)\).
\end{definition}

\begin{definition}
	Let \(\omega\) be a logarithmic \(1\)-form on \(X\) with poles along \(D\).
	The \emph{residue} of \(\omega\) along \(D\) is the holomorphic function on \(D^{\circ}\) locally given by
	\begin{equation}
	\Res_{D}(\omega) := a|_{\{h=0\}} \,,
	\end{equation}
	where \(\omega\) is as in \eqref{eq:log1form}.	
\end{definition}

\begin{remark}\label{rmk:linearityresidue}
	The residue is obviously linear: (i) \(\Res_D(\omega_1 + \omega_2) = \Res_D(\omega_1) + \Res_D(\omega_2)\); (ii) if \(f\) is a holomorphic function on \(X\), then \(\Res_D(f \omega) = f \Res_D(\omega)\).	
\end{remark}

\begin{notation}
	Let \(D = \sum D_i\) be the irreducible decomposition of \(D\).
	We write \(\Res_{D_i}(\omega)\) for the restriction of  \(\Res_D(\omega)\) to \(D_i \cap D^{\circ}\). 
\end{notation}

In general, the residues might have poles at \(\Sing(D)\), as the following example illustrates.

\begin{example}[{\cite[p. 278]{saito}}]\label{ex:3lines}
	Consider the meromorphic \(1\)-form on \(\C^2\) given by
	\[
	\omega = \frac{1}{w-z} \left( \frac{dz}{z} - \frac{dw}{w} \right) \,.
	\]
	It is easy to check that \(\omega\) is a logarithmic \(1\)-form on \(\C^2\) with poles along \(D = D_1 + D_2 + D_3\), where \(D_1 = \{z=0\}\)\,,\, \(D_2 = \{w=0\}\) and \(D_3 = \{z=w\}\). 
	
	A straightforward calculation shows that
	the residues of \(\omega\) are given by: \(\Res_{D_1}(\omega) = 1/w\),
	\(\Res_{D_2}(\omega) = 1/z\), and  \(\Res_{D_3}(\omega) = -1/z\).
	In all \(3\) cases, \(D_i \cap D^{\circ} \cong \C^*\) and \(\Res_{D_i}(\omega)\) has a simple pole at the origin.
\end{example}

\subsubsection{Auxiliary lemmas}
The following elementary lemmas will be useful later on.

\begin{lemma}\label{lem:zeroresidues}
	Let \(\omega\) be a logarithmic \(1\)-form on \(X\) with poles along \(D\).
	If \(\Res_{D}(\omega) = 0\), then \(\omega \in \Omega^1(X)\).
\end{lemma}

\begin{proof}
	The hypothesis \(\Res_D(\omega) = 0\) implies that \(\omega\) is holomorphic outside \(\Sing(D)\).
	By Hartogs, \(\omega \in \Omega^1(X)\).
\end{proof}

\begin{lemma}\label{lem:dlogf}
	Let \(f\) be a non-zero holomorphic function on a complex manifold \(X\) and let \(\omega = df / f\). Then the following holds:
	\begin{enumerate}[label=\textup{(\roman*)}]
		\item \(\omega\) is a logarithmic \(1\)-form on \(X\) with poles along \(D = \{f=0\}\);
		\item if \(x \in D^{\circ}\) then
		\(\Res_D(\omega)(x) = m\), where \(m\) is the order of vanishing of \(f\) along the irreducible component of \(D\) that contains \(x\).
	\end{enumerate}
\end{lemma}

\begin{proof}
	Take \(x \in D^{\circ}\). Then we can find an open neighbourhood \(U\) of \(x\) such that \(f = h^{m} g\) with \(h, g \in \mO(U)\), where \(h\) is a defining equation of \(D\) and \(g\) is nowhere vanishing. Then \(\omega = m dh /h + dg/g\) and \(dg /g\) is a holomorphic \(1\)-form on \(U\). Both (i) and (ii) follow from this.
\end{proof}

\subsection{The holomorphic residues condition}\label{sec:holres}
Let \(\omega\) be a logarithmic \(1\)-form on \(X\) with poles along \(D\). 

\begin{definition}\label{def:holres0}
	We say that \(\omega\) has \emph{holomorphic residues} if the irreducible components \(D_i\) of \(D\) are smooth and the residues \(\Res_{D_i}(\omega)\) extend holomorphically across \(\Sing(D)\) over the whole divisors \(D_i\) as holomorphic functions on \(D_i\).
\end{definition}

\begin{lemma}\label{lem:holres1form}
	Let \(\omega\) be a logarithmic \(1\)-form on \(X\) with poles along \(D\). Assume that the irreducible components \(D_i\) of \(D\) are smooth. Then \(\omega\) has holomorphic residues along \(D\) if and only if for every \(x \in X\) there is an open set \(U \subset X\) with \(x \in U\) such that
	\begin{equation}\label{eq:holres1form}
	\omega = \sum_{i\,|\, x \in D_i} a_i \, \frac{dh_i}{h_i} \,+\, \hol \,,
	\end{equation}
	where \(a_i \in \mO(U)\) and \(h_i \in \mO(U) \) are defining equations of \(D_i\).
\end{lemma}

\begin{proof}
	One direction is clear, if \(\omega\) is locally of the form given by Equation \eqref{eq:holres1form} then \(\Res_{D_i}(\omega) = a_i|_{\{h_i=0\}}\) are holomorphic on \(D_i\). 
	
	To show the converse, let \(x \in X\) and let \(U\) be an open neighbourhood of \(x\) such that \(U \cap D_j = \emptyset\) if \(x \notin D_j\). Shrinking \(U\) if necessary, we can also assume that for every \(D_i\) such that \(x \in D_i\) we have a defining equation \(h_i \in \mO(U)\). 
	Since the divisors \(D_i\) are smooth and \(\Res_{D_i}(\omega)\) are holomorphic on \(D_i\), we can take holomorphic functions \(a_i\) on (a possibly smaller) \(U\) such that \(a_i|_{D_i} = \Res_{D_i}(\omega)\). Consider the logarithmic \(1\)-form on the open set \(U\)  given by
	\[
	\eta = \omega \,-\, \sum_{i\,|\, x \in D_i} a_i \frac{dh_i}{h_i} \,.
	\]
	By construction, the residues of \(\eta\) are identically zero. By Lemma \ref{lem:zeroresidues} we have \(\eta \in \Omega^1(U)\) and Equation \eqref{eq:holres1form} follows with \(\hol = \eta\).
\end{proof}

\subsection{Pullback of a logarithmic \(1\)-form by a log resolution}\label{sec:pullback}

\subsubsection{Properties of regular birational maps}
A regular map \(\pi: X \to Y\) between complex varieties is called \emph{birational} if it restricts to an isomorphism between non-empty Zariski open subsets of \(X\) and \(Y\). We recall below some standard properties of regular birational maps that will be used later.

\begin{lemma}\label{lem:shafarevich}
	Let \(\pi: X \to Y\) be a regular birational map between smooth quasi-projective complex varieties. Suppose that \(\pi\) is not an isomorphism. Then there exists a divisor \(E \subset X\) such that \(Z = \pi(E)\) is a subvariety of \(Y\) of codimension \(\geq 2\) and \(\pi: X \setminus E \to Y \setminus Z\)
	is an isomorphism.
\end{lemma}

\begin{proof}
	See \cite[Ch. II, \S 4.4]{shafarevich1}.
\end{proof}

The divisor \(E \subset X\) is called the \emph{exceptional divisor} of \(\pi\).
The subvariety \(Z \subset Y\) is the indeterminacy locus of the inverse rational map. 
The complements \(X \setminus E\) and \(Y \setminus Z\) are the largest open subsets of \(X\) and \(Y\), respectively, on which the morphism \(\pi: X \to Y\) restricts to an isomorphism.

If \(V\) is a subvariety of \(Y\) with \(V \not\subset Z\) then the \emph{proper transform} of \(V\), denoted by \(\tV\), is the subvariety of \(X\) given by 
\[
\tV = \overline{\pi^{-1}(V \setminus Z)} \,.
\]
The restriction of \(\pi\) to \(\tV\) is a regular birational map \(\tV \xrightarrow{\pi} V\). 

\begin{lemma}\label{lem:pullbacksections}
	Let \(f: P \to Q\) be a regular birational map between smooth quasi-projective varieties. Let \(\mE\) be a holomorphic vector bundle on \(Q\) and let \(f^*\mE\) be its pullback to \(P\). Then
	\[
	f^*H^0(\mE) = H^0(f^*\mE) \,.
	\]
\end{lemma}

\begin{proof}
	Clearly, \(f^*H^0(\mE) \subset H^0(f^*\mE)\). Conversely, let \(\ts \in H^0(f^*\mE)\). By Lemma \ref{lem:shafarevich}, there is a section \(s\) of \(\mE\) defined outside a subvariety of \(Q\) of codimension \(\geq 2\) such that \(\ts = \pi^*s\) on a non-empty Zariski open set. By Hartogs, \(s \in H^0(\mE)\).
\end{proof}

\subsubsection{Pullback of a logarithmic \(1\)-form by a log resolution} 

\begin{definition}\label{def:logres}
	Let \(Y\) be a smooth quasi-projective complex variety, and let \(\Gamma \subset Y\) be a divisor.  
	A \emph{log resolution} of the pair \((Y, \Gamma)\) consists of a smooth quasi-projective variety \( X \) together with a regular map \(\pi : X \to Y\) satisfying the following conditions:
	\begin{enumerate}[label=\textup{(\roman*)}]
		\item the total transform \(D := \pi^{-1}(\Gamma)\) is a simple normal crossing divisor on \( X \);
		\item the map \(\pi\) restricts to an isomorphism of algebraic varieties \(X \setminus D \cong Y \setminus \Gamma\).
	\end{enumerate}
\end{definition}

\begin{remark}
	Item (ii) implies that the map \(\pi\) is birational.	
\end{remark}

The irreducible decomposition of the total transform \(D = \pi^{-1}(\Gamma)\) is 
\begin{equation}\label{eq:irrdecD}
D = \sum_i \tGamma_i \,+\, \sum_k E_k \,,
\end{equation} 
where \(\tGamma_i\) are the proper transforms of the irreducible components \(\Gamma_i\) of \(\Gamma\)
and \(E = \sum E_k\) is the irreducible decomposition of the exceptional divisor. By definition, the divisors \(\tGamma_i\) and \(E_k\) are smooth and have normal crossing intersections.

\begin{proposition}\label{prop:pullback1form}
	Let \(Y\) be a smooth quasi-projective complex variety and let \(\Gamma \subset Y\) be a divisor.
	Let \(\omega\) be a logarithmic \(1\)-form on \(Y\) with poles along \(\Gamma\). 
	Suppose that \(\omega\) has holomorphic residues. 
	Let \(\pi: X \to Y\) be a log resolution of the pair \((Y, \Gamma)\) and let \(D = \pi^{-1}(\Gamma)\). 
	Then the following holds:
	\begin{enumerate}[label=\textup{(\roman*)}]
		\item the pull-back \(\pi^*\omega\) is a logarithmic \(1\)-form on \(X\) with poles along \(D\);
		\item the residues of \(\pi^*\omega\) are given by
		\begin{equation}
		\Res_{\tGamma_i}\left(\pi^*\omega\right) = \pi^* \left(\Res_{\Gamma_i}(\omega) \right) \,,
		\end{equation}
		\begin{equation}
		\Res_{E_k}\left(\pi^*\omega\right)  = \sum_{i \,|\, \pi(E_k) \subset \Gamma_i} m_{ik} \cdot  \pi^* \left( \Res_{\Gamma_i}(\omega) \right) \,,	
		\end{equation} 
		where \(m_{ik}\) is the order of vanishing of \(\pi^*h_i\) along \(E_k\) for any choice of defining equation \(h_i\) of \(\Gamma_i\).
	\end{enumerate}
\end{proposition}

\begin{proof}
	By Lemma \ref{lem:holres1form}\,, we can cover \(Y\) with open subsets \(V \subset Y\) such that
	\(\omega = \sum a_i dh_i / h_i + \hol\) with \(a_i, h_i \in \mO(V)\), where \(h_i\) is a defining equation of \(\Gamma_i\). (If  \(\Gamma_i \bigcap V = \emptyset\) we take \(h_i=1\).) Then
	\(\pi^*\omega = \sum (\pi^*a_i) df_i / f_i + \hol\) with \(f_i = \pi^* h_i\). Given this, items (i) and (ii) of the proposition follow from Lemma \ref{lem:dlogf} and the linearity of residues (see Remark \ref{rmk:linearityresidue}).
\end{proof}

The next example illustrates the necessity of the holomorphic residues condition in the above proposition.

\begin{example}
	Let \(\omega\) be as in Example \ref{ex:3lines}. Take coordinates \((s, t)\) on \(\Bl_0\C^2\) such that the exceptional divisor is \(E = \{s=0\}\) and \(\pi(s, t) = (z, w)\) with \(z = s\) and \(w = s  t\). An easy calculation shows that
	\[
	\pi^*\omega = \frac{dt}{st(1-t)} \,.
	\]
	In particular, near the exceptional divisor \(s=0\) and away from \(t=0, 1\), the pullback \(\pi^*\omega\) presents a simple pole of `the wrong' or `mixed' type \(dt/s\) and therefore it is not logarithmic.
\end{example}

Conversely, we have the following.

\begin{proposition}\label{prop:pullback1form2}
	Let \(\tomega\) be a logarithmic \(1\)-form on \(X\) with poles along \(D\).
	Then \(\tomega = \pi^*\omega\), where \(\omega\) is a logarithmic \(1\)-form on \(Y\) with poles along \(\Gamma\). Moreover, if the irreducible components of \(\Gamma\) are smooth, then \(\omega\) has holomorphic residues.
\end{proposition}

\begin{proof}
	Let \(E \subset X\) be the exceptional divisor of \(\pi\) and let \(Z = \pi(E)\) be the indeterminacy locus of \(\pi^{-1}\). Clearly,
	there is a logarithmic \(1\)-form \(\omega\) defined on the complement of \(Z\) such that \(\tomega = \pi^*\omega\) on \(X \setminus E\).  Since the subvariety \(Z \subset Y\) has codimension \(\geq 2\), by Hartogs (see item iv in \cite[\S 1.1]{saito} for a precise version in the context of logarithmic differential forms), the \(1\)-form \(\omega\) extends across \(Z\) as a logarithmic \(1\)-form on \(Y\) with poles along \(\Gamma\). This proves the first part of the proposition. 
	
	Suppose now that the irreducible components of \(\Gamma\) are smooth, we want to show that \(\omega\) has holomorphic residues. Fix an irreducible component of \(\Gamma\), say 
	\(\Gamma_i\). Let \(a = \Res_{\Gamma_i}(\omega)\) be the residue of \(\omega\) along \(\Gamma_i\). By definition, \(a\) is a holomorphic function on \(\Gamma_i \cap \Gamma^{\circ}\). We want to show that \(a\) extends across \(\Sing(\Gamma)\) as an holomorphic function on \(\Gamma_i\).
	Let \(\tGamma_i \subset X\) be the proper transform of \(\Gamma_i\) and let \(\ta = \Res_{\tGamma_i}(\tomega)\). 
	Since \(\tGamma_i\) is an irreducible component of \(D\) and \(D\) is a simple normal crossing divisor, by \cite[Theorem 3.8]{novikovyakovenko}, the residue \(\ta\) extends holomorphically to the whole \(\tGamma_i\). 
	
	Consider the regular birational map \(f: \tGamma_i \to \Gamma_i\) obtained by restriction of the log resolution \(\pi: X \to Y\) to \(\tGamma_i\). Let \(V\) be the dense open subset of \(\Gamma_i\) defined as \(V = \Gamma_i \setminus (\Sing(\Gamma) \bigcup Z)\) and let \(U = f^{-1}(V)\). By definition, \(\ta = f^*a\) on \(U\). On the other hand, by Lemma \ref{lem:pullbacksections} applied to \(f: \tGamma_i \to \Gamma_i\) and the trivial line bundle \(\mE = \mO_{\Gamma_i}\), there is a holomorphic function \(\varphi\) on \(\Gamma_i\) such that \(\ta = f^*\varphi\). It suffices to show that \(\varphi = a\) on \(\Gamma_i \cap \Gamma^{\circ}\). This follows from the fact that \(\varphi = a\) on \(V\) and \(V \subset \Gamma_i \cap \Gamma^{\circ}\) is dense.
\end{proof}

\subsection{Connections}\label{app:conn}
Let \(X\) be a complex manifold of dimension \(n\) and let \(\mE\) be a holomorphic vector bundle on \(X\) of rank \(r\).
A \emph{holomorphic connection} on \(\mE\) is a \(\C\)-linear map of sheaves 
\[
\nabla: \mE \to \Omega_X^1 \otimes \mE
\] 
satisfying the Leibniz rule \(\nabla(f  s) = df \otimes s + f \nabla s\)\,.

\subsubsection{Local representation of connections}
Let \(\nabla\) be a holomorphic connection on \(\mE\).
Consider a trivialization of \(\mE\) given by a frame of holomorphic sections \(s = (s_1, \ldots, s_r)\) defined on an open set \(U \subset X\). 
Let \(\omega_{ij} \in \Omega^1(U)\) be given by
\begin{equation}
\nabla s_j = - \sum_{i=1}^{r} \omega_{ij} \otimes s_i \,.
\end{equation}
The \emph{connection matrix} of \(\nabla\) in the frame \(s\) is the matrix of \(1\)-forms \(\Omega = (\omega_{ij})\) whose \((i,j)\)-th entry is equal to \(\omega_{ij}\).

In the trivialization of \(\mE\) given by the frame \(s\), a section \(\sigma\) of \(\mE\) corresponds to an \(r\)-tuple of holomorphic functions \(f_i \in \mO(U)\) given by \(\sigma = \sum f_i s_i\). 
Similarly, a section \(\tau\) of \(\Omega^1_X \otimes \mE\) corresponds to an \(r\)-tuple of holomorphic \(1\)-forms  \(\eta_i \in \Omega^1(U)\) given by \(\tau = \sum \eta_i \otimes s_i\). 
Using the frame \(s\) to represent sections of \(\mE\) and \(\Omega_X^1 \otimes \mE\) as column vectors of holomorphic functions and \(1\)-forms, we have 
\begin{equation}
\nabla = d - \Omega \,.	
\end{equation}

Suppose now that \(s' = (s_1', \ldots, s_r')\) is another frame of holomorphic sections of \(\mE\) over \(U\). 
Let \(\Omega'\) be the connection matrix of \(\nabla\) in the frame \(s'\). We want to relate the connection matrices \(\Omega'\) and \(\Omega\). To do this, write \(s_j = \sum_i g_{ij}s'_i\) with \(g_{ij} \in \mO(U)\). The matrix valued function \(G = (g_{ij})\) defines a \emph{gauge transformation} 
\[
G : U \to GL(r, \C) \,.
\]

A holomorphic section \(\sigma\) of \(\mE\) can be written in the frames \(s, s'\) as \(f, f' \in \mO(U)^{\oplus r}\). The relation \(\sum f'_i s'_i = \sum f_j s_j\) (both being equal to \(\sigma\)) implies that \(f' = Gf\).
Similarly, the holomorphic section \(\tau =\nabla \sigma\) of \(\Omega_X^1 \otimes \mE\) can be written in the frames \(s, s'\) as two column vectors of \(1\)-forms \(\eta = (\eta_i)\) and \(\eta' = (\eta'_i)\). The relation \(\sum \eta'_i \otimes s'_i = \sum \eta_j \otimes s_j\) (both being equal to \(\tau\)) implies that \(\eta' = G \eta\). Since \(\eta' = df' - \Omega' f'\) and \(\eta = df - \Omega f\), we get that
\[
df' - \Omega' f' = G \cdot \left( df - \Omega f \right) \,.
\]
Replacing \(f' = Gf\) into the above equation, we deduce that
\begin{equation}\label{eq:gaugetransf}
\Omega' =  G \cdot \Omega \cdot G^{-1} + (dG) \cdot G^{-1} \,,
\end{equation}
which gives the desired relation between the two connection matrices.

\subsubsection{Flatness}
Let \(\nabla\) be a holomorphic connection on \(\mE \to X\). 
The connection \(\nabla\) is flat if near any point \(x \in X\) we can find a frame of parallel holomorphic sections of \(\mE\).
In a local trivialization, the connection \(\nabla = d - \Omega\) is flat if the connection matrix satisfies
\begin{equation}
d \Omega = \Omega \wedge \Omega \,.
\end{equation}

\addcontentsline{toc}{section}{References}
\bibliographystyle{alpha}
\bibliography{refs}	

@article{simpson,
	title={The {C}hern character of a parabolic bundle, and a parabolic {R}eznikov theorem in the case of finite order at infinity},
	author={Iyer, Jaya N and Simpson, Carlos T},
	journal={ar{X}iv: math/0612144},
	year={2006}
}

@article{dunkl,
	title={A numerical characterization of {D}unkl systems},
	author={de Borbon, Martin and Panov, Dmitri},
	journal={arXiv:2601.15430},
	year={2026}
}

@article{mochizukiasymptotic,
	title={Asymptotic behaviour of tame harmonic bundles and an application to pure twistor {D}-modules. {Part} 1},
	author={Mochizuki, Takuro},
	journal={Memoirs of the American Mathematical Society},
	volume={185},
	number={869},
	pages={xii+324 pp},
	year={2007},
	publisher={American Mathematical Society (AMS)}
}

@article {chl,
	AUTHOR = {Couwenberg, Wim and Heckman, Gert and Looijenga, Eduard},
	TITLE = {Geometric structures on the complement of a projective
	arrangement},
	JOURNAL = {Publ. Math. Inst. Hautes \'{E}tudes Sci.},
	FJOURNAL = {Publications Math\'{e}matiques. Institut de Hautes \'{E}tudes
	Scientifiques},
	NUMBER = {101},
	YEAR = {2005},
	PAGES = {69--161},
	ISSN = {0073-8301,1618-1913},
	MRCLASS = {32S22 (20F55)},
	MRNUMBER = {2217047},
	MRREVIEWER = {V.\ A.\ Golubeva},
	DOI = {10.1007/s10240-005-0032-3},
	URL = {https://doi.org/10.1007/s10240-005-0032-3},
}

@article {troyanov,
	AUTHOR = {Troyanov, Marc},
	TITLE = {Les surfaces euclidiennes {\`a} singularit{\'e}s coniques},
	JOURNAL = {Enseign. Math. (2)},
	FJOURNAL = {L'Enseignement Math\'ematique. Revue Internationale. 2e
	S\'erie},
	VOLUME = {32},
	YEAR = {1986},
	NUMBER = {1-2},
	PAGES = {79--94},
	ISSN = {0013-8584},
	MRCLASS = {30F20},
	MRNUMBER = {850552},
	MRREVIEWER = {E.\ Bujalance},
}

@article{petruninlebedeva,
	title={Local characterization of polyhedral spaces},
	author={Lebedeva, Nina and Petrunin, Anton},
	journal={Geometriae Dedicata},
	volume={179},
	number={1},
	pages={161--168},
	year={2015},
	publisher={Springer}
}

@article{takano,
	title={A reduction theorem for a linear {P}faffian system with regular singular points},
	author={Takano, Kyoichi},
	journal={Archiv der Mathematik},
	volume={31},
	number={1},
	pages={310--316},
	year={1978},
	publisher={Springer}
}

@article{ohtsuki,
	title={A residue formula for {C}hern classes associated with logarithmic connections},
	author={Ohtsuki, Makoto},
	journal={Tokyo Journal of Mathematics},
	volume={5},
	number={1},
	pages={13--21},
	year={1982},
	publisher={Publication Committee for the Tokyo Journal of Mathematics}
}

@article{pkc,
	title={Polyhedral {K}\"{a}hler cone metrics on $\mathbb{C}^n$ singular at hyperplane arrangements},
	author={de Borbon, Martin and Panov, Dmitri},
	journal={ar{X}iv: 2106.13224},
	year={2021}
}

@article{miyaokayau,
	title={A {M}iyaoka-{Y}au inequality for hyperplane arrangements in $\mathbb{CP}^n$},
	author={de Borbon, Martin and Panov, Dmitri},
	journal={ar{X}iv: 2411.09573},
	year={2024}
}

@article{novikovyakovenko,
	title={Lectures on meromorphic flat connections},
	author={Novikov, Dmitry and Yakovenko, Sergei},
	journal={Normal forms, bifurcations and finiteness problems in differential equations},
	volume={137},
	pages={387--430},
	year={2004}
}

@article {panovpetrunin, 
title={ Ramification conjecture and {H}irzebruch’s property of line arrangements}, volume={152}, 
DOI={10.1112/S0010437X16007648}, 
number={12}, 
journal={Compositio Mathematica}, author={Panov, D. and Petrunin, A.}, year={2016}, 
pages={2443–2460}
}

@article{novikovtahar,
  title={On limit sets for geodesics of meromorphic connections},
  author={Novikov, Dmitry and Shapiro, Boris and Tahar, Guillaume},
  journal={Journal of dynamical and control systems},
  volume={29},
  number={1},
  pages={55--70},
  year={2023},
  publisher={Springer}
}

@article {conciniprocesi,
	AUTHOR = {De Concini, C. and Procesi, C.},
	TITLE = {Wonderful models of subspace arrangements},
	JOURNAL = {Selecta Math. (N.S.)},
	FJOURNAL = {Selecta Mathematica. New Series},
	VOLUME = {1},
	YEAR = {1995},
	NUMBER = {3},
	PAGES = {459--494},
	ISSN = {1022-1824,1420-9020},
	MRCLASS = {14D99 (32G13 52B30)},
	MRNUMBER = {1366622},
	MRREVIEWER = {V.\ Leksin},
	DOI = {10.1007/BF01589496},
	URL = {https://doi.org/10.1007/BF01589496},
}

@article {conciniprocesi2,
	AUTHOR = {De Concini, C. and Procesi, C.},
	TITLE = {Hyperplane arrangements and holonomy equations},
	JOURNAL = {Selecta Math. (N.S.)},
	FJOURNAL = {Selecta Mathematica. New Series},
	VOLUME = {1},
	YEAR = {1995},
	NUMBER = {3},
	PAGES = {495--535},
	ISSN = {1022-1824,1420-9020},
	MRCLASS = {14D99 (32G34 32S40 52B30 57M25)},
	MRNUMBER = {1366623},
	MRREVIEWER = {V.\ Leksin},
	DOI = {10.1007/BF01589497},
	URL = {https://doi.org/10.1007/BF01589497},
}

@article {dima,
	AUTHOR = {Panov, Dmitri},
	TITLE = {Real line arrangements with the {H}irzebruch property},
	JOURNAL = {Geom. Topol.},
	FJOURNAL = {Geometry \& Topology},
	VOLUME = {22},
	YEAR = {2018},
	NUMBER = {5},
	PAGES = {2697--2711},
	ISSN = {1465-3060,1364-0380},
	MRCLASS = {14N20 (20F55 32Q15 32S22 51F15 52B70)},
	MRNUMBER = {3811768},
	MRREVIEWER = {Zach\ Teitler},
	DOI = {10.2140/gt.2018.22.2697},
	URL = {https://doi.org/10.2140/gt.2018.22.2697},
}

@book{BuragoBuragoIvanov2001,
  author    = {Dmitri Burago and Yuri Burago and Sergei Ivanov},
  title     = {A Course in Metric Geometry},
  series    = {Graduate Studies in Mathematics},
  volume    = {33},
  publisher = {American Mathematical Society},
  address   = {Providence, RI},
  year      = {2001},
  pages     = {xiv + 415},
  isbn      = {0821821296, 9780821821299}
}

@article{feichtner,
	title={Chow rings of toric varieties defined by atomic lattices},
	author={Feichtner, Eva Maria and Yuzvinsky, Sergey},
	journal={Inventiones mathematicae},
	volume={155},
	number={3},
	pages={515--536},
	year={2004},
	publisher={Springer}
}

@incollection {hirzebruch2,
	AUTHOR = {Hirzebruch, F.},
	TITLE = {Algebraic surfaces with extremal {C}hern numbers (based on a
	dissertation by {T}. {H}\"{o}fer, {B}onn, 1984)},
	JOURNAL = {Uspekhi Mat. Nauk},
	FJOURNAL = {Akademiya Nauk SSSR i Moskovskoe Matematicheskoe Obshchestvo.
	Uspekhi Matematicheskikh Nauk},
	VOLUME = {40},
	YEAR = {1985},
	NUMBER = {4(244)},
	PAGES = {121--129},
	ISSN = {0042-1316},
	MRCLASS = {14J25 (14J29 32J15)},
	MRNUMBER = {807793},
	MRREVIEWER = {Werner\ Kleinert},
}

@article{mochizuki,
	title={Kobayashi-{H}itchin correspondence for tame harmonic bundles and an application},
	author={Mochizuki, Takuro},
	journal={Ast{\'e}risque},
	volume={309},
	pages={1--117},
	year={2006}
}

@book {okonek,
	AUTHOR = {Okonek, Christian and Schneider, Michael and Spindler, Heinz},
	TITLE = {Vector bundles on complex projective spaces},
	SERIES = {Modern Birkh\"{a}user Classics},
	NOOPNOTE = {Corrected reprint of the 1988 edition,
	With an appendix by S. I. Gelfand},
	NOOPPUBLISHER = {Birkh\"{a}user/Springer Basel AG, Basel},
	YEAR = {2011},
	PAGES = {viii+239},
	ISBN = {978-3-0348-0150-8},
	MRCLASS = {14J60 (14D20)},
	MRNUMBER = {2815674},
}

@book{orlikterao,
	title={Arrangements of hyperplanes},
	author={Orlik, Peter and Terao, Hiroaki},
	volume={300},
	year={1992},
	publisher={Springer Science \& Business Media}
}

@article {panov,
	AUTHOR = {Panov, Dmitri},
	TITLE = {Polyhedral {K}\"{a}hler manifolds},
	JOURNAL = {Geom. Topol.},
	FJOURNAL = {Geometry \& Topology},
	VOLUME = {13},
	YEAR = {2009},
	NUMBER = {4},
	PAGES = {2205--2252},
	ISSN = {1465-3060,1364-0380},
	MRCLASS = {53C56 (32Q15 53C55)},
	MRNUMBER = {2507118},
	MRREVIEWER = {Igor\ Belegradek},
	DOI = {10.2140/gt.2009.13.2205},
	URL = {https://doi.org/10.2140/gt.2009.13.2205},
}

@article {saito,
	AUTHOR = {Saito, Kyoji},
	TITLE = {Theory of logarithmic differential forms and logarithmic
	vector fields},
	JOURNAL = {J. Fac. Sci. Univ. Tokyo Sect. IA Math.},
	FJOURNAL = {Journal of the Faculty of Science. University of Tokyo.
	Section IA. Mathematics},
	VOLUME = {27},
	YEAR = {1980},
	NUMBER = {2},
	PAGES = {265--291},
	ISSN = {0040-8980},
	MRCLASS = {32G11 (14D05 32B30)},
	MRNUMBER = {586450},
	MRREVIEWER = {Zoghman\ Mebkhout},
}

@incollection {stanley,
	AUTHOR = {Stanley, Richard P.},
	TITLE = {An introduction to hyperplane arrangements},
	BOOKTITLE = {Geometric combinatorics},
	SERIES = {IAS/Park City Math. Ser.},
	VOLUME = {13},
	PAGES = {389--496},
	PUBLISHER = {Amer. Math. Soc., Providence, RI},
	YEAR = {2007},
	ISBN = {978-0-8218-3736-8; 0-8218-3736-2},
	MRCLASS = {52C35 (05B35 55R80)},
	MRNUMBER = {2383131},
	DOI = {10.1090/pcms/013/08},
	URL = {https://doi.org/10.1090/pcms/013/08},
}

@book{griffithsharris,
	title={Principles of algebraic geometry},
	author={Griffiths, Phillip and Harris, Joseph},
	year={2014},
	publisher={John Wiley \& Sons}
}

@book {shafarevich1,
	AUTHOR = {Shafarevich, Igor R.},
	TITLE = {Basic algebraic geometry. Volume 1},
	EDITION = {Second},
	NOOPNOTE = {Varieties in projective space},
	PUBLISHER = {Springer-Verlag, Berlin},
	YEAR = {1994},
	PAGES = {xx+303},
	ISBN = {3-540-54812-2},
	MRCLASS = {14-01},
	MRNUMBER = {1328833},
}

@article{yuzvinsky,
	title={Cohomology bases for the {D}e {C}oncini-{P}rocesi models of hyperplane arrangements and sums over trees},
	author={Yuzvinsky, Sergey},
	journal={Inventiones mathematicae},
	volume={127},
	number={2},
	pages={319--336},
	year={1997},
	publisher={Berlin, Springer-Verlag.}
}

\Address

\end{document}